\documentclass[preprint,aos]{imsart}
\usepackage{amsfonts}
\usepackage{xcolor}
\usepackage{amsmath}
\usepackage{mathrsfs} 
\usepackage{float}
\usepackage{amssymb,tabularx,multicol,multirow,booktabs}
\usepackage{mhequ}
\usepackage{relsize}
\usepackage{rotating}
\usepackage{mathtools}
\usepackage{bm, bbm}
\usepackage{bibunits}
\usepackage[top=1in, bottom=1in, left=1in, right=1in]{geometry}
\usepackage[toc,page]{appendix}

\usepackage{color}
\definecolor{darkblue}{rgb}{0.0, 0.0, 0.55}
\RequirePackage{natbib}
\setattribute{journal}{name}{}
\RequirePackage[OT1]{fontenc}
\RequirePackage{amsthm,amsmath,natbib,graphicx,enumerate}
\RequirePackage[colorlinks,citecolor=darkblue,urlcolor=darkblue]{hyperref}
\RequirePackage{hypernat}
\usepackage{cleveref}
\startlocaldefs
\numberwithin{equation}{section}
\theoremstyle{plain}
\endlocaldefs
\input{tcilatex}

\def \be{\begin{equs}}
\def \ee{\end{equs}}

\def \E{\mathbb{E}}
\def \P{\mathbb{P}}
\def \var{\mathrm{var}}
\def \EB{\mathrm{EB}}
\def \TB{\mathrm{TB}}
\def \cov{\mathrm{cov}}

\def \bhat{\hat{b}}
\def \phat{\hat{p}}

\def \k0 {k_0}

\def \k1 {k_1}
\def \k2 {k_2}
\def \k3 {k_3}

\def \sumin {\sum_{i=1}^n}

\def \n2 {\frac{1}{n(n-1)}}

\let\hat\widehat
\let\tilde\widetilde

\def \Leb{\mathrm{Leb}}

\def \bmax {\beta_{\max}}

\def \zk {\bar{Z}_k}
\def \omegahat {\widehat{\Omega}}

\def \ghat {\hat{g}}
\def \epsp {\varepsilon_{\phat}}
\def \epsb {\varepsilon_{\bhat}}
\def \Etheta {\E_{\theta}}

\def \IF {\widehat{\mathrm{IF}}}

\def \psihat {\hat{\psi}}
\def \psihatmk {\psihat_{m,k}}
\def \psihatac {\psihatmk^{\rm ac}}
\def \psihatemp {\psihatmk^{\rm emp}}
\def \Holder{\text{H\"{o}lder}}

\def \BL{\mathbb{L}}
\def \sPi{\mathsf{\Pi}}
\def \IIFF{\mathbb{IF}}
\def \calM{\mathcal{M}}
\def \diff{\mathrm{d}}

\usepackage{bbold}
\usepackage{orcidlink}

\theoremstyle{definition}
\newtheorem{innercustomthm}{Condition}
\newenvironment{customthm}[1]
  {\renewcommand\theinnercustomthm{#1}\innercustomthm}
  {\endinnercustomthm}
 
\usepackage{todonotes}

\begin{document}

\begin{bibunit}[imsart-nameyear]

\begin{frontmatter}
\title{Semiparametric Efficient Empirical Higher Order Influence Function Estimators}
\runtitle{Empirical HOIF}
\begin{aug}
	\author{\fnms{Lin} \snm{Liu}\orcidlink{0000-0002-9883-7962}\thanksref{t4}\ead[label=e4]{linliu@sjtu.edu.cn}},
	\author{\fnms{Rajarshi} \snm{Mukherjee}\orcidlink{0000-0002-5761-8958}\thanksref{t1}\ead[label=e1]{ram521@mail.harvard.edu}},
	\author{\fnms{Whitney} \snm{K. Newey}\orcidlink{0000-0003-2699-4704}\thanksref{t2}\ead[label=e2]{wnewey@mit.edu}},
	\author{\fnms{James} \snm{M. Robins}\orcidlink{0000-0001-6609-209X}\thanksref{t3}\ead[label=e3]{robins@hsph.harvard.edu}}
	\thankstext{t4}{Assistant Professor, Institute of Natural Sciences, MOE-LSC, School of Mathematical Sciences, CMA-Shanghai and SJTU-Yale Joint Center for Biostatistics and Data Science, Shanghai Jiao Tong University}
	\thankstext{t1}{Associate Professor, Department of Biostatistics, Harvard University}
	\thankstext{t2}{Professor, Department of Economics, Massachusetts Institute of Technology}
	\thankstext{t3}{Professor, Department of Epidemiology and Biostatistics, Harvard University}
\end{aug}

\begin{abstract} 
\cite{robins2008higher} applied the theory of higher order influence functions (HOIFs) to derive an estimator of the mean $\psi$ of an outcome Y in a missing data model with Y missing at random conditional on a vector X of continuous covariates; their estimator, in contrast to other existing estimators but ours, is semiparametric efficient under the minimal H\"{o}lder smoothness conditions derived in \citet{robins2009semiparametric}, together with an additional (non-minimal) H\"{o}lder smoothness condition on the density $g$ of $X$, because that particular estimator depends on a non-parametric estimate of $g$. In this paper, we introduce a new HOIF estimator that has the same asymptotic properties as the previous one, but imposes no smoothness requirement on $g$. This improvement is significant for two reasons. First, one rarely has the knowledge about the smoothness properties of $g$. Second, even when $g$ is smooth, and even if $X$ is just multivariate with fixed dimensions, accurate nonparametric estimation of its density is generally not feasible at the sample sizes often encountered in practice. In fact, to our knowledge, this new HOIF estimator to be studied here remains the \textit{only} semiparametric efficient estimator of $\psi$ under minimal H\"{o}lder smoothness conditions, despite the rapidly growing literature on causal effect estimation. We also show that our estimator can be generalized to the entire class of functionals considered by \cite{robins2008higher} which includes the average effect of a treatment on a response $Y$ when a vector $X$ suffices to control for confounding and the expected conditional variance of $Y$ given $X$. Simulation experiments are also conducted, which demonstrate that our new estimator outperforms previous ones proposed in earlier works on HOIFs in finite samples, when $g$ is not very smooth.
\end{abstract}
\begin{keyword}[class=AMS]
\kwd[Primary ]{62G05}
\kwd{62G20}
\kwd{62F12}
\end{keyword}
\begin{keyword}
\kwd{Higher Order Influence Functions}
\kwd{Doubly Robust Functionals}
\kwd{Semiparametric Efficiency}
\kwd{Higher Order U-Statistics}
\kwd{Causal Inference and Missing Data}
\end{keyword}

\end{frontmatter}
\section{Introduction}
\label{sec:intro}
\cite{robins2008higher}, together with a companion technical report \cite{robins2016technical} containing more details, introduced novel U-statistic based estimators of a class of nonlinear functionals in semi- and non-parametric models. Construction of these estimators was based on the theory of Higher Order Influence Functions (henceforth referred to as HOIFs) \citep{robins2008higher}. HOIFs are U-statistics that represent higher order derivatives of a functional. The authors of the aforementioned papers used HOIFs to construct minimax rate-optimal estimators of an important class of functionals in models with $n^{-1/2}$ minimax rates and in models with higher complexity and hence slower minimax rates, where the model complexity was defined in terms of H\"{o}lder smoothness exponents. This class of functionals is of central importance in biostatistics, epidemiology, economics, and other social sciences and is formally defined in Section~\ref{section_general} below. As specific examples, the class includes the mean of a response $Y$ when $Y$ is missing at random (MAR), the average effect of a treatment on a response $Y$ when treatment assignment is ignorable given a vector $X$ of baseline covariates, and the expected conditional covariance of two variables $A$ and $Y$ given a vector $X$. \cite{robins2008higher} describe other important functionals in this class. Following \cite{robins2008higher}, we shall refer to functionals as $\sqrt{n}$\textit{-estimable} if its minimax estimation rate is $n^{-1/2}$ and as \textit{non-}$\sqrt{n}$\textit{-estimable} if slower.

One may wonder why HOIFs are of interest in the $\sqrt{n}$-estimable case studied in this paper, given the recent progress \citep{newey2018cross, kennedy2023towards} in attaining $\sqrt{n}$-consistency with refined first-order doubly robust estimators under conditions close to the minimal H\"{o}lder smoothness conditions on the nuisance parameters (abbreviated as ``the minimal H\"{o}lder smoothness conditions'' in the sequel) for $\sqrt{n}$-consistency proved in Theorem 3.1 of \cite{robins2009semiparametric}. The initial version of the current article has been available on arXiv since 2017. Yet as the literature stands, the new ``empirical HOIF estimators'' to be studied here {\it remain the only existing $\sqrt{n}$-consistent estimator for the mean of a response $Y$ under MAR under the exact minimal \Holder{} smoothness conditions \citep{robins2009semiparametric}}. All other simpler estimators by refining the first-order doubly robust estimators but not based on HOIFs can only achieve $\sqrt{n}$-consistency under strictly stronger smoothness assumptions. More surprisingly in this case, HOIF estimators offer a ``free lunch'', at least asymptotically and information-theoretically: one may obtain semiparametric efficiency with HOIF estimators whose variance is dominated by the linear term associated with the usual first order influence function but whose bias is corrected using higher order U-statistics, i.e. HOIFs. 

The key contribution of this paper is to introduce empirical HOIF estimators for $\sqrt{n}$-estimable parameters that, unlike previously constructed estimators based on HOIFs, avoid non-parametric estimation of a multidimensional density $g$. This is practically important because accurate multidimensional non-parametric density estimation is generally infeasible at the sample sizes often encountered in applications. Indeed, in Section~\ref{section_simulations} we present the results of a simulation study demonstrating that our new empirical HOIF estimator can improve upon existing HOIF estimators in finite samples. Arguably more importantly, the previous HOIF estimators constructed in \citet{robins2008higher} still rely on an extra smoothness assumption on the density $g$ to attain $\sqrt{n}$-rate beyond the minimal H\"{o}lder smoothness conditions. However, the new empirical HOIF estimators, since they completely bypass nonparametric density estimation, do not need to impose an extra smoothness assumption on $g$, and hence achieve $\sqrt{n}$-consistency or semiparametric efficiency \emph{exactly under the minimal H\"{o}lder smoothness conditions}, This article therefore closes an important theoretical gap in the literature.

To our surprise, the idea behind our new estimator is exceedingly simple. For $\sqrt{n}$-estimable parameters, all HOIF estimators considered heretofore have required an estimate of the inverse of a large Gram matrix of dimension of order $o (n)$ whose entries are expectations under a nonparametric estimate $\hat{g}$ of the true density $g$. Our new HOIFs estimator simply uses the inverse of the empirical/sample Gram matrix estimator (expectations under the empirical distribution rather than $\hat{g}$), thereby avoiding estimation of $g$. We refer to the new estimators as empirical HOIF estimators due to the use of the sample Gram matrix. Our main technical contribution is a proof that the new estimator is $\sqrt{n}$-consistent and semiparametric efficient under the minimal H\"{o}lder smoothness conditions derived in \citet{robins2009semiparametric}. 


The paper is organized as follows. In Sections~\ref{section_missing_data_obs}, we motivate the need for HOIF estimators and explain when and why HOIF estimators could have improved properties compared to the more commonly used first-order estimators. For the sake of concreteness, we do so in the context of the specific example of estimating the mean of a response subject to missing at random (MAR). This example is isomorphic to the problem of estimating the mean of the potential outcome in the treatment arm under no unmeasured confounding. In Section~\ref{section_missing_data_estimator} we introduce our new empirical HOIF estimator. In Section~\ref{section_missing_data_theorem} we analyze the large sample statistical properties of our estimator and compare its behavior to the HOIF estimator of \cite{robins2008higher, robins2017minimax}\footnote{The original proof on the variance bound of the estimator in \citet{robins2017minimax} used Lemma 14.1 in the supplementary materials of \citet{robins2017minimax}. But Lemma 14.1 is incorrect. This error was spotted while writing the current paper. The updated arXiv version \citep{robins2023minimax} of \citet{robins2017minimax} has corrected the proof by using Hoeffding decomposition.}. In Section~\ref{section_missing_data_adapt} we show that in contrast with the estimators in \citet{robins2008higher, robins2017minimax}, the empirical HOIF estimator is semiparametric efficient under \textit{minimal conditions} when the complexity of the model is defined in terms of \text{H\"{o}lder}{} smoothness classes. In Section~\ref{section_general} we extend the results of Section~\ref{section_missing_data} to the more general class of doubly robust functionals studied by \cite{robins2008higher}. Section~\ref{section_simulations} provides simulation experiments that support the theoretical results developed in this paper. Section~\ref{section_literature} provides a literature review and Section~\ref{section_discussions} discusses implications of the results and open problems. Finally we collect the proofs and required technical lemmas in the Appendix.

\section{Review of and motivation for HOIF estimators}
\label{section_missing_data_obs} 

To explain why HOIF estimators can be useful in the $\sqrt{n}$-estimable case, we focus on the following example of estimating the mean response $Y$ when $Y$ is MAR. We observe $N$ i.i.d. copies of observed data vector $W = (X^{\top}, A, A Y)^{\top}$. Here $A \in \{0, 1\}$ is the indicator of the event that a response $Y$ is observed and $X$ is a $d$-dimensional vector of covariates with density $f(x)$ with respect to Lebesgue measure on a connected and compact set in $\mathbb{R}^{d}$, which we assume to be $[0, 1]^{d}$ from now on. Define 
\begin{align*}
B \coloneqq b(X) \coloneqq \mathbb{E} (Y | X, A = 1) \text{ and } \Pi \coloneqq \pi (X) \coloneqq \mathbb{P} (A = 1 | X) = \mathbb{E} (A | X),
\end{align*}
where $x \mapsto b (x)$ is the outcome regression function and $x \mapsto \pi (x)$ is the probability of missingness. Our goal is to estimate $\psi \coloneqq \mathbb{E} \left[ \frac{AY}{\pi (X)} \right] = \mathbb{E}\left[ b(X) \right] = \int b (x) f (x) \diff x$. Interest in $\psi $ lies in the fact that it is the marginal mean of $Y$ under MAR that $\mathbb{P} (A = 1 | X, Y) = \pi (X)$. It will be useful to reparametrize the model by $\theta = (b, p, g)$ for functions $x \mapsto b (x), x \mapsto p (x), x \mapsto g (x)$ where $p (\cdot) \coloneqq \pi^{-1} (\cdot), g (\cdot) \coloneqq \mathbb{P} (A = 1 | X = \cdot) f(\cdot) = \pi (\cdot) f(\cdot) = f (\cdot | A = 1) \mathbb{P} (A = 1)$. Further, it is easy to see that the parameters $b, p, g$ are variationally independent, meaning that the range of possible values that any one of $b, p, g$ can take is invariant to the values of the other two parameters. As discussed in \cite{robins2008higher,robins2017minimax}, the parametrization $(b, p, g)$ is more natural than $(b, \pi, f)$, as will be evident from the formulas provided below. We also assume that $g$ is absolutely continuous with respect to the Lebesgue measure for notational convenience. However, this assumption is not needed for the main results of our paper; see Remark~\ref{rem:abs_cont}.2 below. We write the corresponding probability measure, expectation, and variance operators as $\mathbb{P}_{\theta}, \Etheta$, and $\var_{\theta}$ respectively. Finally, we write the target functional $\theta \mapsto \psi (\theta)$ of interest as
\begin{equation}
\label{param}
\chi(\mathbb{P}_{\theta}) \coloneqq \psi (\theta) = \int b (x) p (x) g (x) \diff x. 
\end{equation}

We assume that the law of $W$ belongs to a model 
\begin{equation*}
\mathcal{M} \coloneqq \mathcal{M} (\Theta) \coloneqq \left\{ \mathbb{P}_{\theta}, \theta \in \Theta \right\},
\end{equation*}
where for some fixed $\overline{\sigma} > \underline{\sigma} > 0$, 
\begin{equation} \label{eqn:theta_def}
\Theta \subseteq \{\theta: \inf_{x} \pi (x) \geq \underline{\sigma}, \inf_{x} g(x) \geq \underline{\sigma}, \sup_{x} g (x) \leq \overline{\sigma}\}.
\end{equation}
We also assume that the model $\mathcal{M}$ is locally nonparametric, in the sense that the tangent space at each $\mathbb{P}_{\theta} \in \mathcal{M}$ equals $L_{2} (\mathbb{P}_{\theta})$, intersected with all zero-mean functions under $\mathbb{P}_{\theta}$. \cite{ritov1990achieving} and \cite{robins1997toward} have shown that no uniformly consistent estimator for $\psi (\theta)$, let alone a $\sqrt{n}$-consistent estimator, exists under $\mathcal{M}$ if we do not impose any smoothness or other structural assumptions on the nuisance parameters $(b, p, g)$. 

One common strategy is to impose \Holder-type smoothness conditions on the nuisance parameters \citep{stone1982optimal}, which is the structural assumption that we choose to focus on in this paper. We define \Holder{} classes in detail later in Section~\ref{section_missing_data_adapt}. The lower bounds in \cite{robins2009semiparametric} and upper bounds in \cite{robins2017minimax} together proved that if $\mathcal{M}$ specifies that $b$, $p$ and $g$ belong to \Holder{} balls with exponents $\beta_{b}$, $\beta_{p}$ and $\beta_{g}$ (see Definition \ref{def_Holder} for the precise meaning of \Holder{} balls), then, \textit{provided that $\beta_{g} > \epsilon$ for some $\epsilon > 0$}, (i) $\beta_{b} + \beta_{p} \geq \frac{d}{2}$ is necessary and sufficient for the existence of a $\sqrt{n}$-consistent estimator of $\psi (\theta)$ and (ii) if $\beta_{b} + \beta _{p} > \frac{d}{2}$, there exists a semiparametric efficient estimator, i.e. a regular and asymptotically linear (RAL) estimator of $\psi (\theta)$ with the first order influence function $\mathrm{IF}_{1} (\theta)$ defined in the paragraph below. In this paper, we show that the above results continue to hold, even without imposing any smoothness condition on $g$ except for it being bounded from above and below as in \eqref{eqn:theta_def}. We obtain this result by exhibiting a new semiparametric efficient (resp. $\sqrt{n}$-consistent) estimator of $\psi (\theta)$ whenever $\beta_{b} + \beta _{p} > \frac{d}{2}$ (resp. $\beta_{b} + \beta _{p} \geq \frac{d}{2}$), oblivious to the smoothness condition on $g$.

It is well-known \citep{robins1995semiparametric, hahn1998role} that the unique first order influence function \citep{newey1990semiparametric, bickel1993efficient, ichimura2022influence} for $\psi$ at $\theta$ in Model~$\mathcal{M}$ is 
\begin{equation*}
\mathrm{IF}_{1} (\theta) = A p(X) (Y - b(X)) + b(X) - \psi (\theta),
\end{equation*}
which we can also succinctly write as $A P (Y - B) + B - \psi (\theta)$, where we recall that we denote $b (X)$ as $B$ and $p (X)$ as $P$ in the beginning of this section. To construct the usual first order estimator, we first divide the whole sample with size $N$ into an estimation sample with size $n$ and a training sample with size $n_{\rm tr} = N - n$ satisfying $n \asymp n_{\rm tr}$. Because $\mathrm{IF}_{1} (\theta)$, like all influence functions, has mean zero by definition, the natural first order estimator $\hat{\psi}_{1}$ of $\psi (\theta)$ is: 
\begin{equation*}
\hat{\psi}_{1} \coloneqq \frac{1}{n} \sum_{i=1}^{n} A_{i} \hat{p}(X_{i}) (Y_{i} - \hat{b}(X_{i})) + \hat{b}(X_{i}) = \frac{1}{n} \sum_{i=1}^{n} A_{i} \hat{P}_{i} (Y_{i} - \hat{B}_{i}) + \hat{B}_{i},
\end{equation*}
where $\hat{b} (\cdot)$ and $\hat{p} (\cdot)$ are estimated nuisance functions computed from the training sample and we similarly denote $\hat{B} \coloneqq \hat{b} (X)$ and $\hat{P} \coloneqq \hat{p} (X)$ for short. Conditional on the training sample, $\hat{\psi}_{1}$ is the sum of $n$ i.i.d. random variables, and hence it is asymptotically normally distributed with mean $\psi (\theta) + \mathsf{cBias}_{\theta} (\hat{\psi}_{1})$ and variance of order $1 / n$, where by straightforward algebra
\begin{align*}
\mathsf{cBias}_{\theta} (\hat{\psi}_{1}) & \coloneqq \mathbb{E}_{\theta} \left[ A (\hat{P} - P) (B - \hat{B}) \mid \textrm{training sample} \right]  \\
& =\int (\hat{p}(x) - p(x)) (b(x) - \hat{b}(x)) g (x) \diff x
\end{align*}
is the conditional bias of $\hat{\psi}_{1}$ (see Appendix~\ref{app:cbias} for its derivation). Henceforth, we shall often suppress the dependence on the training sample in the notation for convenience. $\mathsf{cBias}_{\theta} (\hat{\psi}_{1})$ needs to be $o_{\mathbb{P}_{\theta}} (n^{-1/2})$ (resp. $O_{\mathbb{P}_{\theta}} (n^{-1/2})$) to ensure semiparametric efficiency (resp. $\sqrt{n}$-consistency) of $\hat{\psi}_{1}$. Under the \Holder{} smoothness conditions on $\theta = (b, p, g)$, if $\hat{b}$ and $\hat{p}$ are minimax rate optimal estimators of $b$ and $p$, their respective rates of convergence in $L_{2} (\mathbb{P}_{\theta})$-norm are $n^{-\frac{\beta_{b}}{2\beta_{b}+d}}$ and $n^{-\frac{\beta_{p}}{2\beta_{p}+d}}$, and hence, by the Cauchy-Schwarz inequality, $\mathsf{cBias}_{\theta} (\hat{\psi}_{1})$ is $O_{\mathbb{P}_{\theta}} (n^{-\frac{\beta_{b}}{2\beta_{b}+d} - \frac{\beta_{p}}{2\beta_{p}+d}})$. This suggests that when $\frac{\beta_{b}}{2\beta_{b}+d}+\frac{\beta_{p}}{2\beta _{p}+d} < 0.5$, $\hat{\psi}_{1}$ may fail to be $\sqrt{n}$-consistent. As a concrete example, suppose $\beta_{b} = \beta_{p} = \frac{d}{4}$, then $\mathsf{cBias}_{\theta} (\hat{\psi}_{1})$ is $O_{\mathbb{P}_{\theta}} (n^{-\frac{\beta_{b}}{2\beta_{b}+d} - \frac{\beta_{p}}{2\beta_{p}+d}}) = O_{\mathbb{P}_{\theta}} (n^{-1/3})$, which suggests that $\hat{\psi}_{1}$ may not be $\sqrt{n}$-consistent. 

A natural idea to improve $\hat{\psi}_{1}$ is to estimate its (conditional) bias $\mathsf{cBias}_{\theta} (\hat{\psi}_{1})$ and then to construct a new estimator $\hat{\psi}_{2}$ of $\psi (\theta)$ that subtracts the estimate of $\mathsf{cBias}_{\theta} (\hat{\psi}_{1})$ from $\hat{\psi}_{1}$. HOIF estimators can be viewed as a general scheme for instantiating this bias reduction idea; see \cite{van2014higher} for a pedagogical review. In the special case of our MAR missing data problem, the bias reduction scheme proceeds as follows. One first chooses a vector $\bar{z}_{k} (\cdot) = \left( z_{1} (\cdot), \ldots, z_{k} (\cdot) \right)^{\top}$ of $k$ (basis) functions of the covariates $X$ (see Section~\ref{section_missing_data_adapt} for further discussions on the requirements on these functions)\footnote{In this paper, we restrict ourselves not to choose the basis functions using any data-driven methods, because data-driven basis selection is a difficult open problem for HOIFs; see the end of Section~\ref{section_simulations} and Section~\ref{section_discussions}.}. Let $Z_{j} \coloneqq z_{j} (X)$ for $j = 1, \cdots, k$ and $\bar{Z}_{k} \coloneqq \bar{z}_{k} (X)$. Then by Pythagorean theorem, $\mathsf{cBias}_{\theta} (\hat{\psi}_{1})$ can be decomposed as follows: 
\begin{align*}
& \ \mathsf{cBias}_{\theta} (\hat{\psi}_{1}) = \int (\hat{p}(x) - p(x)) (b(x) - \hat{b}(x)) g (x) \diff x \\
= & \ \int \mathsf{\Pi}_{g, \bar{z}_{k}} [\hat{p} - p] (x) \mathsf{\Pi}_{g, \bar{z}_{k}} [b - \hat{b}] (x) g (x) \diff x + \int \mathsf{\Pi}_{g, \bar{z}_{k}}^{\perp} [\hat{p} - p] (x) \mathsf{\Pi}_{g, \bar{z}_{k}}^{\perp} [b - \hat{b}] (x) g (x) \diff x,
\end{align*}
where $h \mapsto \mathsf{\Pi}_{g, \bar{z}_{k}} [h]$ denotes the $L_{2} (g)$-projection of any function $h$ onto the orthogonal complement to the linear space spanned by $\bar{z}_{k}$ and reads as:
\begin{align*}
& \mathsf{\Pi}_{g, \bar{z}_{k}} [h] (\cdot) = \left( \int h(x) \bar{z}_{k}(x) g(x) \diff x \right)^{\top} \Omega^{-1} \bar{z}_{k} (\cdot) = \mathbb{E}_{\theta} [A h (X) \bar{z}_{k} (X)]^{\top} \Omega^{-1} \bar{z}_{k} (\cdot), \\
& \text{with } \Omega \coloneqq \int \bar{z}_{k}(x) \bar{z}_{k}(x)^{\top} g(x) \diff x = \mathbb{E}_{\theta} [A \bar{z}_{k}(X) \bar{z}_{k}(X)^{\top}],
\end{align*}
and $h \mapsto \mathsf{\Pi}_{g, \bar{z}_{k}}^{\perp} [h] (\cdot)$ denotes orthogonal complement of the projection operation $\mathsf{\Pi}_{g, \bar{z}_{k}}$. Following \cite{robins2008higher} and \cite{li2011higher}, we refer to the first term in the above bias decomposition as the first-order estimation bias $\EB_{1, k} (\theta)$ and the second term as the truncation bias $\TB_{k} (\theta)$ for reasons explained below.

Noting that
\begin{equation*}
\sPi_{g, \bar{z}_{k}} [h] (X) = \mathbb{E}_{\theta} [A h(X) \bar{z}_{k}(X)]^{\top} \Omega^{-1} \bar{z}_{k} (X), 
\end{equation*}
we thus have 
\begin{align*}
\EB_{1, k} (\theta) & = \mathbb{E}_{\theta} \left[ \mathbb{E}_{\theta} [A (\hat{P} - P) \bar{Z}_{k}^{\top}]
\Omega^{-1} A \bar{Z}_{k} \bar{Z}_{k}^{\top} \Omega^{-1} \mathbb{E}_{\theta} [\bar{Z}_{k} A (B - \hat{B})] \right] \\
& = \mathbb{E}_{\theta} [A (\hat{P} - P) \bar{Z}_{k}]^{\top} \Omega^{-1} \mathbb{E}_{\theta} [\bar{Z}_{k} A (B - \hat{B})] \\
& = \mathbb{E}_{\theta} [(A \hat{P} - 1) \bar{z}_{k} (X)]^{\top} \Omega^{-1} \mathbb{E}_{\theta} [\bar{z}_{k} (X) A (B - \hat{B})] \\
& = - \mathbb{E}_{\theta} [(1 - A \hat{P}) \bar{z}_{k} (X)]^{\top} \Omega^{-1} \mathbb{E}_{\theta} [\bar{z}_{k} (X) A (Y - \hat{B})].
\end{align*}

From this last expression it follows that were $\Omega$ known, then $-\EB_{1, k} (\theta)$ can be unbiasedly estimated by the following oracle second-order U-statistic: 
\begin{equation*}
\widehat{\mathbb{IF}}_{2, 2, k} (\Omega) \coloneqq \frac{(n - 2)!}{n!} \sum_{1 \leq i_{1} \neq i_{2} \leq n} \IF_{2, 2, k, \bar{i}_{2}} (\Omega),
\end{equation*}
where $\IF_{2, 2, k, \bar{i}_{2}} (\Omega) \coloneqq [(1 - A \hat{p} (X)) \bar{z}_{k} (X)^{\top}]_{i_{1}} \Omega^{-1} [\bar{z}_{k} (X) A (Y - \hat{b} (X))]_{i_{2}}$\footnote{We briefly comment on our choice of notation. $\IF_{2, 2, k, \bar{i}_{2}}$ and $\widehat{\mathbb{IF}}_{2, 2, k}$ introduced here, and the more general $\IF_{j, j, k, \bar{i}_{j}}$ and $\widehat{\mathbb{IF}}_{j, j, k}$ to be introduced in Section~\ref{section_missing_data} are chosen to be consistent with the notation system in \citet{robins2008higher} and \citet{robins2016technical}, in which the original HOIF theory was developed. In this paper, the two $j$'s in the subscript always take the same value.}. Here we introduce the short-hand notation $[f (O)]_{i} \coloneqq f (O_{i})$ for any function $f$, which will be used throughout this paper. We call this statistic oracle because it depends on the true $\Omega$. We then obtain the bias corrected oracle estimator $\hat{\psi}_{2, k} (\Omega) \coloneqq \hat{\psi}_{1} + \widehat{\mathbb{IF}}_{2, 2, k} (\Omega)$. It follows that $\hat{\psi}_{2, k} (\Omega)$ is an unbiased estimator of the so-called truncated parameter $\widetilde{\psi}_{2, k} (\theta) = \psi (\theta) + \TB_{k} (\theta)$. The truncation bias $\TB_{k} (\theta)$ is defined as the difference between the truncated parameter $\widetilde{\psi}_{2, k} (\theta)$ and the parameter $\psi (\theta)$ of actual interest, hence the name. The bias of $\hat{\psi}_{1}$ as an estimator of the truncated parameter $\widetilde{\psi}_{2, k} (\theta)$ is equal to $\EB_{1, k} (\theta)$, which is, as we have seen above, unbiasedly estimated by $-\hat{\IIFF}_{2, 2, k} (\Omega)$.

\cite{robins2008higher} (Theorem 3.21) show that $\var_{\theta} [\widehat{\mathbb{IF}}_{2, 2, k} (\Omega)]$ is of order $\frac{k}{n^{2}} + \frac{1}{n}$, which for $k = O (n)$ is smaller than or equal to the order of $\var_{\theta} [\hat{\psi}_{1}]$; hence, asymptotically, we do not increase the order $n^{-1}$ of the variance of $\hat{\psi}_{1}$ when using $\hat{\psi}_{2, k} (\Omega)$ to correct bias. \cite{robins2008higher} (Theorems 3.13 \& 3.14) define HOIFs and prove that $\hat{\psi}_{2, k} (\Omega)$ is the efficient second order influence function of the truncated parameter $\widetilde{\psi}_{2,k} (\theta)$. However the current paper can be read without knowing either the definition or theory of HOIFs, even though the estimators (e.g. $\hat{\psi}_{2, k} (\Omega)$) in \cite{robins2008higher} were derived using such theory.

In contrast with $\EB_{1, k} (\theta)$, $\TB_{k} (\theta)$ cannot be unbiasedly estimated from data. However if the approximations of functions in $L_{2} (g)$ by $\bar{z}_{k} (\cdot)$ are sufficiently accurate for $\TB_{k} (\theta)$ to be of $O_{\mathbb{P}_{\theta}} (n^{-1/2})$, then the bias of $\hat{\psi}_{2, k} (\Omega)$ as an estimator of $\psi (\theta)$ is of $O_{\mathbb{P}_{\theta }}(n^{-1/2})$. When $b$ and $p$ are assumed to lie in certain \text{H\"{o}lder}{} balls with exponents $\beta_{b}$ and $\beta_{p}$, it is well-known that wavelet/B-spline basis functions can be chosen to ensure that $\TB_{k}(\theta)$ is of order $k^{- \frac{\beta_{b} + \beta _{p}}{d}}$. Thus under the minimal H\"{o}lder smoothness condition $\beta_{b} + \beta_{p} \geq \frac{d}{2}$ for $\psi (\theta)$ to be $\sqrt{n}$-estimable, $\TB_{k} (\theta)$ is of order $O_{\mathbb{P}_{\theta}} (n^{-1/2})$ if $\frac{k}{n} \rightarrow c$, for some constant $c > 0$.  This implies $\hat{\psi}_{2, k} (\Omega)$ is minimax rate optimal in view of the lower bound proved in \cite{robins2009semiparametric}. We remark that later in our paper, we will take $k = o (n)$ throughout because of an important issue that we discuss next.

Of course in practice $\Omega$, the population expectation of the Gram matrix of $A \bar{z}_{k}(X)$, is not known and must be estimated. \cite{robins2008higher, robins2017minimax} proposed to estimate $\Omega$ by (1) estimating $g$ by $\hat{g}$ under additional smoothness assumptions on $g$, and then (2) estimating $\Omega$ by $\widehat{\Omega}^{\rm ac} \coloneqq \int \bar{z}_{k}(x) \bar{z}_{k}(x)^{\top} \hat{g} (x) \diff x$ using numerical integration with respect to $\hat{g}$. The second order estimation bias $\EB_{2, k} (\theta) \coloneqq \mathbb{E}_{\theta} [\hat{\psi}_{2, k} (\widehat{\Omega}^{\rm ac}) - \hat{\psi}_{2, k} (\Omega)]$ is defined as the bias of the feasible estimator $\hat{\psi}_{2, k} (\widehat{\Omega}^{\rm ac})$ as an estimator of the truncated parameter $\widetilde{\psi}_{2, k} (\theta)$. \citet{robins2008higher} (Theorem 3.17) prove that $\EB_{2, k} (\theta)$ is $O_{\mathbb{P}_{\theta}} (\Vert \hat{p} - p \Vert \cdot \Vert b - \hat{b} \Vert \cdot \Vert g - \hat{g} \Vert)$ while $\EB_{1, k} (\theta)$ is $O_{\mathbb{P}_{\theta}} (\Vert \hat{p} - p \Vert \cdot \Vert b - \hat{b} \Vert)$. Thus the bias of $\hat{\psi}_{2, k} (\widehat{\Omega}^{\rm ac})$ as an estimator of $\widetilde{\psi}_{2, k} (\theta)$ is of third rather than second order. However the total bias of $\hat{\psi}_{2} (\widehat{\Omega}^{\rm ac})$ for $\psi (\theta)$ is $\EB_{2, k} (\theta) + \TB_{k} (\theta)$ which may still be of larger order than $\TB_{k} (\theta)$. The HOIF estimator $\hat{\psi}_{m, k} (\widehat{\Omega}^{\rm ac})$ of order $m$ is an $m$-th order U-statistics with variance $O_{\mathbb{P}_{\theta}} (n^{-1})$ when (roughly) $\frac{k m^{2}}{n} \rightarrow 0$ and the basis vector $\bar{z}_{k}$ satisfies the technical conditions given in Condition~\ref{def_conditions} presented later; it has bias $\EB_{m, k} (\theta)$ for the truncated parameter $\widetilde{\psi}_{2, k} (\theta)$ of order $$O_{\mathbb{P}_{\theta}} (\Vert \hat{p} - p \Vert \cdot \Vert b - \hat{b} \Vert \cdot \Vert g - \hat{g} \Vert^{m - 1}).$$ By choosing $m = m(n)$ sufficiently large, say of order $\sqrt{\log n}$, the estimation bias will be $O_{\mathbb{P}_{\theta}} (n^{-1/2})$ provided that $\Vert g - \hat{g} \Vert^{m - 1} \rightarrow 0$ at a sufficiently fast rate as $m \rightarrow \infty$. 

There are at least three potential difficulties that may arise when estimating $\Omega$ by first estimating $g$: (1) as just noted, $g$ must be sufficiently smooth to ensure $\Vert g - \hat{g} \Vert^{m - 1} \rightarrow 0$; (2) even when the dimension $d$ of $X$ is moderate, estimating a multidimensional density and then numerically integrating over a multidimensional domain is often computationally prohibitively expensive, and (3) the finite sample accuracy of a nonparametric $d$-dimensional density estimator may be poor at the sample sizes often encountered. Fortunately, by eliminating the need to estimate $g$, difficulties (1)--(3) do not arise for our new empirical HOIF estimator. As a consequence, we show both in theory and through simulations that our new estimator can outperform the estimator $\hat{\psi}_{m, k} (\widehat{\Omega}^{\rm ac})$. Theoretically, we will show that the new estimator is semiparametric efficient (resp. $\sqrt{n}$-consistent) when $\beta_{b} + \beta_{p} > \frac{d}{2}$ (resp. $\beta_{b} + \beta_{p} \geq \frac{d}{2}$), which is also necessary \citep{robins2009semiparametric}. It is again worth noting that, despite active research progress in the past decades, no other (simpler) estimators can yet achieve such a tight theoretical guarantee under the \Holder-type smoothness assumptions.

\section{A New Higher Order Influence Function Estimator in a Missing Data Model}
\label{section_missing_data} 

In this section, we study a particular $\Theta$ defined by membership of the functions $b, p, g$ in certain H\"{o}lder smoothness balls and show that the proposed estimator is adaptive and semiparametric efficient in the corresponding model $M(\Theta)$. However, for now, we work with any $\Theta$ satisfying \eqref{eqn:theta_def}.

We are now ready to define both the estimators of \cite{robins2008higher,robins2017minimax} and then the new estimator of this paper.

\subsection{The Estimators}
\label{section_missing_data_estimator} 
Our estimators will depend on a random variable $H_{1}$\footnote{The reason for attaching a subscript `1' in $H_{1}$ will be made clear in Section~\ref{section_general}.} that will vary depending on the functional in the doubly robust class of \cite{robins2008higher} under investigation in Section~\ref{section_general}. $H_{1}$ will not change sign w.p.1. In the MAR example, we have $H_{1} = - A$, which is non-positive w.p.1. We shall consider estimators $\widehat{\psi}_{m, k}$ constructed as follows where the indices $m$ and $k$ are defined below.

\begin{enumerate}
\item[(i)] The sample is randomly split into two parts: an estimation sample of size $n$ and a training sample of size $n_{\rm tr}=N-n$ with $n / N \rightarrow c^*$ and $n \rightarrow \infty$ with $0 < c^* < 1$.

\item[(ii)] Estimators $\widehat{b}, \widehat{p}, \hat{g}$ are constructed from the training sample data. We do not restrict the form of these estimators unless stated otherwise. Let $\widehat{\theta} \coloneqq (\widehat{b}, \widehat{p}, \widehat{g})$.

\item[(iii)] Given a sequence of basis functions $z_{1} (\cdot), z_{2} (\cdot), \ldots$ over $L_2 ([0, 1]^d)$, let $$\bar{z}_{k} (\cdot) \coloneqq \left( z_{1} (\cdot), z_{2} (\cdot), \ldots, z_{k} (\cdot) \right)^{\top}, Z_{j} \coloneqq z_{j} (X) \text{ for $j = 1, \cdots, k$ and } \bar{Z}_{k} \coloneqq \left(Z_{1}, Z_{2}, \ldots, Z_{k} \right)^{\top},$$ and define the following Gram matrices 
\begin{eqnarray*}
\Omega & \coloneqq & \mathbb{E}_{\theta} [|H_{1}| \bar{Z}_{k} \bar{Z}_{k}^{\top}] = \int \bar{z}_{k} (x) \bar{z}_{k} (x)^{\top} g (x) \diff x, \\
\widehat{\Omega}^{\rm ac} & \coloneqq & \mathbb{E}_{\hat{\theta}} [|H_{1}| \bar{Z}_{k} \bar{Z}_{k}^{\top}] = \int \bar{z}_{k} (x) \bar{z}_{k} (x)^{\top} \hat{g} (x) \diff x, \\
\widehat{\Omega}^{\rm emp} & \coloneqq & \frac{1}{n_{\rm tr}} \sum_{i \in \text{training sample}} [|H_{1}| \bar{Z}_{k} \bar{Z}_{k}^{\top}]_{i}.
\end{eqnarray*}

\item[(iv)] Set 
\begin{equation*}
\widehat{\psi}_{1} \coloneqq \widehat{\psi} + \frac{1}{n} \sum_{i = 1}^{n} \IF_{1, i},
\end{equation*}
where $\widehat{\psi}$ and $\IF_{1}$ are $\psi (\theta)$ and $\mathrm{IF}_{1} (\theta)$ with $\widehat{\theta}$ replacing $\theta$. The estimator $\widehat{\psi}_{1}$ is usually referred to as the one-step estimator that adds the estimated first order influence function to the plug-in estimator.

\item[(v)] Let $\varepsilon_{b} = H_1 (B - Y)$, $\varepsilon_{p} = H_1 P + 1$. For $m = 2 , \ldots$, and any generic invertible estimator $\omegahat$ of $\Omega$, define 
\begin{equation*}
\widehat{\psi}_{m, k} (\omegahat) \coloneqq \widehat{\psi}_1 + \sum_{j = 2}^{m} \widehat{\mathbb{IF}}_{j, j, k}(\omegahat),
\end{equation*}
where $\widehat{\mathbb{IF}}_{j, j, k} (\omegahat)$ is the $j$-th order U-statistic and takes the form:
\begin{equation*}
\widehat{\mathbb{IF}}_{j, j, k} (\omegahat) = \frac{(n - j)!}{n!} \sum_{\bar{i}_{j} \in I_{n, j}} \IF_{j, j, k, \bar{i}_{j}} (\omegahat). 
\end{equation*}
Here $\bar{i}_j \coloneqq \{i_1, i_2, \ldots, i_j\}$ and $I_{n, j} \coloneqq \{\bar{i}_{j}: 1 \leq i_{1} \neq i_{2} \neq \cdots \neq i_{j} \leq n\}$ denotes all possible length-$j$ multi-indices with distinct coordinates out of $\{1, \cdots, n\}$, the indices for all subjects in the estimation sample. And for $j = 2$, 
\begin{align*}
\IF_{2, 2, k, \bar{i}_{2}} (\omegahat) & \coloneqq - (-1)^{I \{H_{1, i_1} \leq 0\}}  [\varepsilon_{\widehat{p}} \bar{Z}_{k}^{\top}]_{i_{1}} \widehat{\Omega}^{-1} [\bar{Z}_{k} \varepsilon_{\widehat{b}}]_{i_{2}},
\end{align*}
whereas for $j > 2$
\begin{align*}
\IF_{j, j, k, \bar{i}_{j}} (\omegahat) & \coloneqq (-1)^{j - 1} (-1)^{I \{H_{1, i_1} \leq 0\}} \left\{ 
\begin{array}{c}
\left[ \varepsilon_{\widehat{p}} \bar{Z}_{k}^{\top} \right]_{i_{1}} \widehat{\Omega}^{-1} \times \\ 
\prod\limits_{s = 3}^{j} \left[ \left\{ [|H_{1}| \bar{Z}_{k} \bar{Z}_{k}^{\top}]_{i_{s}} - \widehat{\Omega} \right\} \widehat{\Omega}^{-1} \right] \\ 
\times \left[ \bar{Z}_{k} \varepsilon_{\widehat{b}} \right]_{i_{2}}
\end{array} \right\}.
\end{align*}
In Appendix~\ref{app:higher-order}, we will explain heuristically why adding $\hat{\IIFF}_{j, j, k} (\hat{\Omega})$ for $j \geq 3$ can reduce the estimation bias, echoing the discussion near the end of Section~\ref{section_missing_data_obs}.
\end{enumerate}

Finally we introduce the short-hand notation:
\begin{align}
\psihatac \coloneqq \psihatmk (\omegahat^{\rm ac}), \quad \psihatemp \coloneqq \psihatmk (\omegahat^{\rm emp}),
\end{align}
where, by convention, we set an estimator to be zero in case the associated Gram matrix estimator $\omegahat$ fails to be invertible, either $\omegahat^{\rm ac}$ or $\omegahat^{\rm emp}$. Note that $\widehat{\psi}_{1}$ is the sample average of $A \widehat{P} (Y - \widehat{B}) + \widehat{B}$ and thus does not depend on $\widehat{g}$. In the above construction, sample-splitting necessarily incurs efficiency loss, so eventually we can employ cross-fitting to restore the efficiency as follows. Analogous to $\psihatac$ and $\psihatemp$, we respectively define $\bar{\hat{\psi}}^{\rm ac}_{m, k}$ and $\bar{\hat{\psi}}^{\rm emp}_{m, k}$ but with the roles of the training and estimation samples reversed. Then we define the cross-fit estimators as
\begin{equation}
\hat{\psi}^{\rm ac}_{m, k, \mathrm{cf}} \coloneqq \frac{\bar{\hat{\psi}}^{\rm ac}_{m, k} + \psihatac}{2}, \hat{\psi}^{\rm emp}_{m, k, \mathrm{cf}} \coloneqq \frac{\bar{\hat{\psi}}^{\rm emp}_{m, k} + \psihatemp}{2}.
\end{equation}
The purpose of defining these cross-fit estimators is to restore the information loss due to sample splitting, as in \citet{chernozhukov2018double}. As will be clear in Corollary~\ref{theorem_semipar_eff_emp} later, estimators without cross-fit have variance of order $n^{-1}$ instead of $N^{-1}$, where $N$ is the total sample size.

\begin{remark} \label{rem:abs_cont}\leavevmode
\begin{enumerate}
\item[\emph{(i)}] Note that in contrast to $\psihatac$ and $\hat{\psi}^{\rm ac}_{m, k, \mathrm{cf}}$, $\psihatemp$ and $\hat{\psi}^{\rm emp}_{m, k, \mathrm{cf}}$ completely bypass the need of estimating the (transformed) density function $g$.

\item[\emph{(ii)}] In Section~\ref{section_missing_data_obs}, we define the parameter of interest $\psi (\theta)$ in equation \eqref{param} under the assumption $g$ is absolutely continuous. Though we do not further pursue in this paper, our results below concerning the statistical properties of $\psihatemp$ and $\hat{\psi}^{\rm emp}_{m, k, \mathrm{cf}}$ should hold in most cases when the distribution of $X$ does not even have a density with respect to the Lebesgue measure, in which case we replace $g (x) \diff x$ by $\diff G (x)$ in equation \eqref{param}. Here $G (\cdot)$ denotes the joint probability distribution of $(X, A = 1)$. For example, it is immediate that our results continue to hold when $X$ is discrete with finite support $\{x_{1}, \cdots, x_{M}\}$ for some bounded integer $M$, and for some $c \in (0, 0.5)$, $c < G (x_{m}) < 1 - c$ for all $m = 1, \cdots, M$. This boundedness assumption and the finite-support assumption are needed to ensure that the population Gram matrix $\Omega$ has bounded eigenvalues.

\item[\emph{(iii)}] In the above definitions, $\hat{\Omega}^{-1}$ can be interpreted as the generalized inverse without actually worrying about the invertibility of $\hat{\Omega}$. It will be clear when discussing the statistical properties of the new estimators $\hat{\psi}_{m, k}^{\rm emp}$ and $\hat{\psi}_{m, k, \mathrm{cf}}^{\rm emp}$ later in Section~\ref{section_missing_data_theorem} that $\hat{\Omega}^{\rm emp}$ is invertible with probability converging to 1 as $n \rightarrow \infty$ under the imposed conditions. See the paragraph right after Theorem~\ref{theorem_bias_variance_emp} later for further comments on this issue.
\end{enumerate}
\end{remark}

\subsection{Analysis of the Estimators}
\label{section_missing_data_theorem} 

\cite{robins2008higher, robins2017minimax} analyzed the statistical properties of estimators $\widehat{\psi}_{m, k}^{\rm ac}$ and $\hat{\psi}^{\rm ac}_{m, k, \mathrm{cf}}$, by assuming that there exists $\epsilon > 0$ such that $g$ lies in a \Holder{} ball with smoothness $\beta_{g} > \epsilon$ and the density estimator $\hat{g}$ converges to $g$ in $L_{\infty}$-norm.

In this paper, we shall instead analyze the statistical properties of the estimator $\widehat{\psi}_{m, k}^{\rm emp}$, which has the advantage of not requiring an estimate $\widehat{g}$ of $g$. The statistical properties of $\widehat{\psi}_{m, k, \mathrm{cf}}^{\rm emp}$ will be an immediate corollary. 

First, we rephrase a previous result from \cite{robins2008higher, robins2017minimax}, giving the conditional bias of a generic HOIF estimator $\widehat{\psi}_{m, k}$, with a generic estimator $\widehat{\Omega}$ computed from the training sample. 

\begin{proposition}\label{theorem_bias_general_omega}
For any invertible $\omegahat$ one has conditional on the training sample,
$$\Etheta \left[ \widehat{\psi}_{m, k} - \psi (\theta) \right] = \EB_{m, k} (\theta) + \TB_{k} (\theta),$$
where
\begin{align*}
\EB_{m, k} (\theta) & = (-1)^{(m - 1) + I \{h_{1} (W) \leq 0\}} \left\{ \begin{array}{c}
\E_{\theta} \left[ H_{1} \left( P - \widehat{P} \right) \bar{Z}_{k}^{\top} \right] \Omega^{-1} \left[ \left\{ \Omega - \widehat{\Omega} \right\} \widehat{\Omega}^{-1} \right]^{m - 1} \\ 
\times \E_{\theta} \left[ \bar{Z}_{k} H_{1} \left( B - \widehat{B} \right) \right] 
\end{array} \right\}, \\
\TB_{k} (\theta) & = (-1)^{I \{h_{1} (W) \leq 0\}} \left\{ \begin{array}{c}
\int (b - \hat{b}) (x) (p - \hat{p}) (x) g (x) \diff x \\ 
- \int \int g (x_{1}) g (x_{2}) (b - \hat{b}) (x_{1}) K_{g, k} (x_{1}, x_{2}) (p - \hat{p}) (x_{2}) \diff x_{2} \diff x_{1}
\end{array} \right\} \\
& = (-1)^{I \{h_{1} (W) \leq 0\}} \int \left( \mathsf{I} - \mathsf{\Pi}_{g, \bar{z}_{k}} \right) (b - \hat{b}) (x) \left( \mathsf{I} - \mathsf{\Pi}_{g, \bar{z}_{k}} \right) (p - \hat{p}) (x) g (x) \diff x,
\end{align*}
with $K_{g, k} (x', x) = \bar{z}_{k}^{\top} (x') \Omega^{-1} \bar{z}_{k} (x)$ the orthogonal projection kernel onto $\bar{z}_{k} (x)$ in $L_{2} (g)$, $\mathsf{\Pi}_{g, \bar{z}_{k}} [h] (x) = \int \diff x' g (x') h (x') K_{g, k} (x, x')$ the corresponding orthogonal projection of any function $x \mapsto h (x)$, and $\mathsf{I} [h] (x) \coloneqq h (x)$ denoting the identity operator.		
\end{proposition}

Throughout the sequel of this paper, we impose the following technical condition: 
\begin{customthm}{B}
\label{def_conditions} 
We say that a choice of basis functions $\{z_l, l \geq 1\}$, and tuple of functions $\tilde{\theta} = (\tilde{b}, \tilde{p}, \tilde{g})$ in $\mathbb{R}^{[0, 1]^d}$ satisfies Condition~\ref{def_conditions} if the following hold for some $1 < B < \infty$ and every $n, k \geq 1$ with $\lambda_{\min}(\Omega)$ and $\lambda_{\max} (\Omega)$ being the minimum and maximum eigenvalues of $\Omega$.
\begin{enumerate}
	\item [(1)] The basis functions $\{z_l, l \geq 1\}$ satisfy $\sup_x \bar{z}_{k}^{\top} (x) \bar{z}_{k} (x) \leq B \cdot k$; 
	
	\item [(2)] $\frac{1}{B} \leq \lambda_{\min}(\Omega) \leq \lambda_{\max}(\Omega) \leq B$;
	
	\item [(3)] $\Vert \tilde{b} \Vert_{\infty}, \Vert \tilde{p} \Vert_{\infty} \leq B$.
\end{enumerate}
\end{customthm}

\begin{remark}
Most commonly used basis functions in nonparametric regression, including wavelets, splines, local polynomial partition series, Fourier series, and Legendre polynomials, satisfy Condition~\ref{def_conditions}(1). Condition~\ref{def_conditions}(2) is met under $\calM$ due to the boundedness constraint \eqref{eqn:theta_def}. To see this, we can first choose a set of basis functions $\bar{z}_{k}$ with size $k$ that are orthonormal with respect to the Lebesgue measure on $[0, 1]^{d}$. We allow $k$ to increase with $n$. For these basis functions, we know that $\int \bar{z}_{k} (x) \bar{z}_{k} (x) \mathrm{d} x = \mathrm{I}_{k}$, where $\mathrm{I}_{k}$ denotes the identity matrix of dimension $k$. Obviously, neither the largest nor the smallest eigenvalue of $\mathrm{I}_{k}$ depends on $n$, although $k$ grows with $n$. $\Omega$ simply replaces $\diff x$ by $g (x) \diff x$ in $\mathrm{I}_{k}$. Therefore, under the assumption that $g$ is bounded from above and below by absolute constants, we can conclude that there exist absolute constants sandwiching the largest and smallest eigenvalues of $\Omega$. Condition~\ref{def_conditions}(3) requires that $\tilde{b}$ and $\tilde{p}$ to be bounded by some constant uniformly over $[0, 1]^{d}$, which we believe is a mild condition. When $\tilde{b}, \tilde{p}$ are further assumed to belong to \Holder{} balls, Condition~\ref{def_conditions}(3) is automatically satisfied.
\end{remark} 

Before stating the next result, we introduce some additional notation on different norms of the residuals between the true nuisance parameters and their estimates obtained from the training sample: for $f \in \{b, p, g\}$, $\hat{f}$ the corresponding estimator, and $\varepsilon_{\hat{f}} \in \{\epsb, \epsp\}$, let
\begin{align*}
& \BL_{q, \hat{f}, k} \coloneqq \Vert \sPi_{g, \bar{z}_{k}} [\varepsilon_{\hat{f}}] \Vert_{q}, \BL_{q, \hat{f}} \coloneqq \Vert f - \hat{f} \Vert_{q}, \BL_{\infty, \hat{f}, k} \coloneqq \Vert \sPi_{g, \bar{z}_{k}} [\varepsilon_{\hat{f}}] \Vert_{\infty}, \BL_{\infty, \hat{f}} \coloneqq \Vert f - \hat{f} \Vert_{\infty}, \\
& \text{and also let } \BL_{2, \widehat{\Omega}, k} \coloneqq \Vert \widehat{\Omega} - \Omega \Vert_{\rm op} \text{ where in this paper } \hat{\Omega} \in \{\hat{\Omega}^{\rm ac}, \hat{\Omega}^{\rm emp}\}.
\end{align*}

The next result characterizes the bias and variance bounds of a generic HOIF estimator $\hat{\psi}_{m, k}$ under the above regularity Condition~\ref{def_conditions}. This is the first new theoretical result of this paper. The results below allow one to deduce the rate of convergence of $\hat{\psi}_{m, k}$ even if $\hat{b}, \hat{p}$ are not consistent estimators of $b, p$, respectively. In contrast to \citet{robins2017minimax} or its corrected version \citep{robins2023minimax}, the bias and variance bounds of $\hat{\psi}_{m, k}$ are represented in terms of $\BL_{2, \hat{\Omega}, k}$, the difference between $\hat{\Omega}$ and $\Omega$ in operator norm, instead of $\Vert \hat{g} - g \Vert_{\infty}$. The latter representation cannot be applied to $\hat{\psi}_{m, k}^{\rm emp}$ as it essentially uses the empirical measure to estimate the law of $X | A = 1$.

\begin{proposition}
\label{theorem_bias_variance_ghat}
Assume that $\{z_l, \ l\geq 1\}$ satisfies Condition~\ref{def_conditions}(1) and~\ref{def_conditions}(2) and $(\bhat, \phat)$ satisfy Condition~\ref{def_conditions}(3). Then the following hold:
\begin{enumerate}
\item[\emph{1.}] $\TB_{k} (\theta) = O_{\mathbb{P}_{\theta}} (\Vert \left( \mathsf{I} - \mathsf{\Pi}_{g, \bar{z}_{k}} \right) [b - \widehat{b}] \Vert_{2} \cdot \Vert \left( \mathsf{I} - \mathsf{\Pi}_{g, \bar{z}_{k}} \right) [p - \widehat{p}] \Vert_{2})$;
	
\item[\emph{2.}] There exists a constant $C > 0$ such that $$\EB_{m, k} \left( \theta \right) = O_{\mathbb{P}_{\theta}} (\BL_{2, \hat{b}, k} \cdot \BL_{2, \hat{p}, k} \cdot \{C \cdot \BL_{2, \omegahat, k}\}^{m - 1}) = O_{\mathbb{P}_{\theta}} (\BL_{2, \hat{b}} \cdot \BL_{2, \hat{p}} \cdot \{C \cdot \BL_{2, \omegahat, k}\}^{m - 1});$$

\item[\emph{3.}] Restricted to the event that $\omegahat$ is invertible, the general form of an upper bound of $\var_{\theta} [\widehat{\psi}_{m, k} - \widehat{\psi}_{1}]$ that holds for any $m, k$ is given in \eqref{eq:var_psim} in Section~\ref{app:var_bound_1}. If we further take $m \asymp \log n$ and $k \lesssim \frac{n}{\log^{3} n}$, and if $\BL_{2, \omegahat, k} = o_{\mathbb{P}_{\theta}} (\log^{-1} n)$, $$\var_{\theta} [\widehat{\psi}_{m, k} - \widehat{\psi}_{1}] = O_{\mathbb{P}_{\theta}} \Big( \dfrac{1}{n} \Big\{ \dfrac{k}{n} + \Big( \BL_{2, \hat{b}, k}^{2} + \BL_{2, \hat{p}, k}^{2} \Big) + \min\limits_{(\eta, \zeta) \in [1, \infty]^{2}: 1 / \eta + 1 / \zeta = 1} \BL_{2 \eta, \hat{b}, k}^{2} \BL_{2 \zeta, \hat{p}, k}^{2} \Big\} \Big).$$
\end{enumerate}
\end{proposition}

\begin{remark}
In the variance bound statement (part 3 of the above theorem), $(\eta, \zeta) \in [1, \infty]^{2}: 1 / \eta + 1 / \zeta = 1$ forms a so-called \Holder{} conjugate pair \citep{valiant2017automatic}. To avoid clutter, we simply write $(\eta, \zeta): 1 / \eta + 1 / \zeta = 1$ instead in the sequel.
\end{remark}

The proof of the above proposition can be found in Appendix~\ref{section_proofs}. In part 3, the variance bound is stated under further restrictions on $m, k$ just to simplify the exposition. The particular choice of $m$ and $k$ is sufficient for $k m^{2} / n = o (1)$, such that $\var_{\theta} [\hat{\psi}_{m, k} - \hat{\psi}_{1}]$ can be bounded by the claimed order in the theorem; see \eqref{ignore} in the proof of Corollary~\ref{lem:var_psim_emp} in Appendix~\ref{section_proofs}.

Let $\EB_{m, k}^{\rm emp} \left( \theta \right)$ be the corresponding estimation bias of $\hat{\psi}_{m, k}^{\rm emp}$. The truncation bias $\TB_{k} (\theta)$ as defined is independent of how $\Omega$ is estimated. When specialized to the newly proposed empirical HOIF estimator $\hat{\psi}_{m, k}^{\rm emp}$ with $\Omega$ estimated by the sample Gram matrix $\hat{\Omega}^{\rm emp}$, Proposition~\ref{theorem_bias_variance_ghat} implies the first set of results on the new estimator $\hat{\psi}_{m, k}^{\rm emp}$, the main theme of the paper.

\begin{theorem}\label{theorem_bias_variance_emp}
Under the same assumptions of Proposition~\ref{theorem_bias_variance_ghat}, the following hold
\begin{enumerate}
\item[\emph{1.}] The same conclusion in Proposition~\ref{theorem_bias_variance_ghat}.1 holds for $\TB_{k} (\theta)$;

\item[\emph{2.}] There exists a constant $C > 0$ such that $$\EB_{m, k}^{\rm emp} \left( \theta \right) = O_{\mathbb{P}_{\theta}} (\BL_{2, \hat{b}, k} \cdot \BL_{2, \hat{p}, k} \cdot \{C \cdot \BL_{2, \omegahat^{\rm emp}, k}\}^{m - 1}) = O_{\mathbb{P}_{\theta}} (\BL_{2, \hat{b}} \cdot \BL_{2, \hat{p}} \cdot \{C \cdot \BL_{2, \omegahat^{\rm emp}, k}\}^{m - 1});$$
		
\item[\emph{3.}] When $m \asymp \log n$ and $k \lesssim \frac{n}{\log^{3} n}$, conditional on the training sample restricted to the event that $\omegahat^{\rm emp}$ is invertible, 
\begin{equation}
\label{variance bound}
\var_{\theta} [\widehat{\psi}_{m, k}^{\rm emp} - \widehat{\psi}_{1}] = O_{\mathbb{P}_{\theta}} \Big( \dfrac{1}{n} \Big\{ \dfrac{k}{n} + \Big( \BL_{2, \hat{b}, k}^{2} + \BL_{2, \hat{p}, k}^{2} \Big) + \min\limits_{(\eta, \zeta): 1 / \eta + 1 / \zeta = 1} \BL_{2 \eta, \hat{b}, k}^{2} \BL_{2 \zeta, \hat{p}, k}^{2} \Big\} \Big).
\end{equation}
\end{enumerate}	
\end{theorem} 

Theorem~\ref{theorem_bias_variance_emp} is almost a carbon copy of Proposition~\ref{theorem_bias_variance_ghat}. But we again emphasize that Theorem~\ref{theorem_bias_variance_emp} is about the concrete new estimator $\hat{\psi}_{m, k}^{\rm emp}$ proposed in this paper. In part 3, unlike Proposition~\ref{theorem_bias_variance_ghat}, there is no extra condition on the order of $\BL_{2, \hat{\Omega}^{\rm emp}, k}$. This is because for $\hat{\Omega}^{\rm emp}$, by Lemma~\ref{lemma_spectral_bound_rudelson} in Appendix~\ref{section_lemmas}, $\BL_{2, \hat{\Omega}^{\rm emp}, k} = o_{\mathbb{P}_{\theta}} (\log^{-1} n)$ is automatic under the choice of $k$. $\hat{\Omega}^{\rm emp}$ is also invertible with probability converging to 1 by Lemma~\ref{lemma_spectral_bound_rudelson} so there is also no need to worry about the invertibility of $\hat{\Omega}^{\rm emp}$ in our asymptotic statement. Importantly, Theorem~\ref{theorem_bias_variance_emp} enables us to characterize the conditions under which $\hat{\psi}_{m, k}^{\rm emp}$ and $\hat{\psi}_{m, k, \mathrm{cf}}^{\rm emp}$ are $\sqrt{n}$-consistent and $\hat{\psi}_{m, k, \mathrm{cf}}^{\rm emp}$ reaches the semiparametric efficiency bound, leading to another main result of this paper.

\begin{corollary}\label{theorem_semipar_eff_emp}
Under the assumptions of Theorem~\ref{theorem_bias_variance_emp}, if we further assume
\begin{enumerate}			
\item [\emph{(1)}] $m \asymp \log n$ and $k \asymp \frac{n}{\log^{3} n}$; \emph{(2)} $\TB_{k} (\theta) = o_{\mathbb{P}_{\theta}} (n^{-1/2})$;

\item [\emph{(3)}] $\BL_{2, \hat{b}, k}$ and $\BL_{2, \hat{p}, k}$ are $o_{\mathbb{P}_{\theta}} (1)$; and \emph{(4)} $\min\limits_{(\eta, \zeta): 1 / \eta + 1 / \zeta = 1} \BL_{2 \eta, \hat{b}, k}^{2} \BL_{2 \zeta, \hat{p}, k}^{2}$ is $o_{\mathbb{P}_{\theta}} (1)$.
\end{enumerate}
Then
\begin{align*}
\sqrt{n} \left( \widehat{\psi}_{m, k}^{\rm emp} - \psi (\theta) \right) = \frac{1}{\sqrt{n}} \sum_{i = 1}^{n} \mathrm{IF}_{1, i} (\theta) + o_{\mathbb{P}_{\theta}} (1) \text{ and } \sqrt{N} \left( \widehat{\psi}_{m, k, \mathrm{cf}}^{\rm emp} - \psi (\theta) \right) = \frac{1}{\sqrt{N}} \sum_{i = 1}^{N} \mathrm{IF}_{1, i} (\theta) + o_{\mathbb{P}_{\theta}} (1).
\end{align*}
Thus $\widehat{\psi}_{m, k, \mathrm{cf}}^{\rm emp}$ is a semiparametric efficient estimator of $\psi (\theta)$. If \emph{(3)} and \emph{(4)} are replaced by
\begin{enumerate}
\item [\emph{(3')}] $\BL_{2, \hat{b}, k}$ and $\BL_{2, \hat{p}, k}$ are $O_{\mathbb{P}_{\theta}} (1)$; and \emph{(4')} $\min\limits_{(\eta, \zeta): 1 / \eta + 1 / \zeta = 1} \BL_{2 \eta, \hat{b}, k}^{2} \BL_{2 \zeta, \hat{p}, k}^{2}$ is $O_{\mathbb{P}_{\theta}} (1)$,
\end{enumerate}
then
\begin{align*}
\sqrt{n} \left( \widehat{\psi}_{m, k}^{\rm emp} - \psi (\theta) \right) = \frac{1}{\sqrt{n}} \sum_{i = 1}^{n} \mathrm{IF}_{1, i} (\theta) + O_{\mathbb{P}_{\theta}} (1) \text{ and } \sqrt{N} \left( \widehat{\psi}_{m, k, \mathrm{cf}}^{\rm emp} - \psi (\theta) \right) = \frac{1}{\sqrt{N}} \sum_{i = 1}^{N} \mathrm{IF}_{1, i} (\theta) + O_{\mathbb{P}_{\theta}} (1).
\end{align*}
Thus $\widehat{\psi}_{m, k, \mathrm{cf}}^{\rm emp}$ is a $\sqrt{n}$- and $\sqrt{N}$-consistent yet not necessarily semiparametric efficient estimator of $\psi (\theta)$.
\end{corollary}

To ease exposition in the remainder of this paper, we let $\hat{\psi}_{n}^{\rm emp} \coloneqq \widehat{\psi}_{m, k, \mathrm{cf}}^{\rm emp}, \hat{\psi}_{n}^{\rm ac} \coloneqq \widehat{\psi}_{m, k, \mathrm{cf}}^{\rm ac}$ and $\hat{\psi}_{N, \mathrm{cf}}^{\rm emp} \coloneqq \widehat{\psi}_{m, k, \mathrm{cf}}^{\rm emp}, \hat{\psi}_{N, \mathrm{cf}}^{\rm ac} \coloneqq \widehat{\psi}_{m, k, \mathrm{cf}}^{\rm ac}$, when $m$ and $k$ are set according to the choice given in Corollary~\ref{theorem_semipar_eff_emp}. The proof of Corollary~\ref{theorem_semipar_eff_emp} is a direct consequence of Theorem~\ref{theorem_bias_variance_emp} by checking if the chosen $m, k$ in part 2 of Theorem~\ref{theorem_bias_variance_emp} ensures $\EB_{m, k} (\theta) = o (n^{- 1 / 2})$. To this end, by applying Lemma~\ref{lemma_spectral_bound_rudelson} in Appendix~\ref{section_lemmas}, we have
\begin{equation}
\label{matching}
\EB_{m, k} (\theta) = O_{\mathbb{P}_{\theta}} \Big\{ \Big( \sqrt{\frac{k \log k}{n}} \Big)^{m - 1} \Big\} = O_{\mathbb{P}_{\theta}} \Big\{ \Big( \log n \Big)^{- \log n} \Big\} = o_{\mathbb{P}_{\theta}} (n^{-1 / 2}).
\end{equation}
Of course, if $\hat{b}$ and/or $\hat{p}$ enjoy faster convergence rates to $b$ and/or $p$, smaller $k$ and $m$ can be deduced. Corollary~\ref{theorem_semipar_eff_emp} paves the way for proving that $\hat{\psi}_{n}^{\rm emp}$ is semiparametric efficient under the minimal \Holder{} smoothness conditions characterized in \citet{robins2009semiparametric} in the next section, closing an important theoretical gap in the literature.

We now further comment on conditions (4) and (4') in Corollary~\ref{theorem_semipar_eff_emp}. Recall that $\BL_{2 \eta, \hat{b}, k}$ (resp. $\BL_{2 \zeta, \hat{p}, k}$) is the $L_{2 \eta} (\mathbb{P}_{\theta})$-norm (resp. $L_{2 \zeta} (\mathbb{P}_{\theta})$-norm) of the $L_{2} (\mathbb{P}_{\theta})$-projection of $b - \hat{b}$ (resp. $\hat{p} - p$). We mainly consider the following corner cases: $(\eta, \zeta) = (1, \infty)$ or $(\infty, 1)$. When $\eta = 1$, we immediately conclude that $\BL_{2, \hat{b}, k} \leq \BL_{2, \hat{b}}$ by the $L_{2} (\mathbb{P}_{\theta})$-norm contraction property of $L_{2} (\mathbb{P}_{\theta})$-projection; if furthermore $p - \hat{p}$ bounded almost surely implies $\BL_{2 \zeta, \hat{p}, k} = O_{\mathbb{P}_{\theta}} (1)$ for $\zeta = \infty$ (by convention $2 \zeta = \infty$ as well), conditions (4) and (4') hold immediately in the above theorem under conditions (3) and (3'). Similar arguments apply to the case with values of $\eta, \zeta$ reversed. Nevertheless, both upper bound strategies entail an extra $L_{\infty} (\mathbb{P}_{\theta})$-norm stability condition on the basis functions $\bar{z}_{k}$, or more generally, the following $L_{q} (\mathbb{P}_{\theta})$-norm stability condition:
\begin{customthm}{S}
\label{stable}
For any bounded $h \in L_{2} (\mathbb{P}_{\theta})$ and any $k \times k$ matrix $\Sigma$ with operator norm bounded by $M$, there exists a constant $C < \infty$ depending on $\bar{z}_{k}$ and $f_{X}$ such that
\begin{equation}\label{cond:b}
\Big\Vert \bar{z}_{k} (\cdot)^{\top} \Sigma \Etheta \left[ \bar{z}_{k} (X) h(X) \right] \Big\Vert_{q} < C M \Vert h \Vert_{q},
\end{equation}
where $\Vert f \Vert_{q}$ denotes the $L_{q} (\mathbb{P}_{\theta})$-norm of a function $f$ for some $q \in (2, \infty]$.
\end{customthm}

Fortunately, the set of basis functions that satisfy Condition~\ref{stable} is not vacuous. In Appendix~\ref{section_lemmas}, we show that Condition~\ref{stable} is met with $q = \infty$ (and hence any smaller $q > 2$) when we use Daubechies wavelets, B-splines, or local polynomial partition series to approximate $h$ (see Lemma~\ref{lem:supnorm} for details), building on results in \citet{belloni2015some}; also see \citet{huang2003local, chen2015optimal, cattaneo2020large}, or \citet{chen2018optimal}. When Condition~\ref{stable} holds for $q = \infty$, both $\BL_{\infty, \hat{b}, k}^{2}$ and $\BL_{\infty, \hat{p}, k}^{2}$ are $O_{\mathbb{P}_{\theta}} (1)$ because both $\hat{b} - b$ and $\hat{p} - p$ are $O_{\mathbb{P}_{\theta}} (1)$ under Condition~\ref{def_conditions}(3). Thus conditions (4) and (4') in Corollary~\ref{theorem_semipar_eff_emp} hold respectively under conditions (3) and (3'), as argued before Condition~\ref{stable}.

\begin{remark}\leavevmode
\label{rem:basis}
For basis functions $\bar{z}_{k}$ that satisfy \eqref{cond:b}, if $\BL_{\infty, \hat{b}} = \Vert b - \hat{b} \Vert_{\infty}$ and $\BL_{\infty, \hat{p}} = \Vert p - \hat{p} \Vert_{\infty}$ are $O_{\mathbb{P}_{\theta}} (1)$ (or $o_{\mathbb{P}_{\theta}} (1)$), then $\BL_{2, \hat{b}, k}$, $\BL_{2, \hat{p}, k}$, $\BL_{\infty, \hat{b}, k}$, $\BL_{\infty, \hat{p}, k}$ are also at most $O_{\mathbb{P}_{\theta}} (1)$ (or $o_{\mathbb{P}_{\theta}} (1)$). For basis functions that may violate \eqref{cond:b}, such as Fourier series or monomial transformations of $X$, the above statement may be false and the analysis needs to be done case by case.
\end{remark}

Before proceeding, we first compare and contrast the empirical HOIF estimator $\hat{\psi}_{n}^{\rm emp}$ or $\hat{\psi}_{N, \mathrm{cf}}^{\rm emp}$ proposed in this paper with $\hat{\psi}_{n}^{\rm ac}$ or $\hat{\psi}_{N, \mathrm{cf}}^{\rm ac}$ originally considered in \citet{robins2017minimax}. As we have seen, the new estimators proposed here differ from the original ones in how $\Omega$ is estimated. For $\hat{\psi}_{n}^{\rm ac}$ or $\widehat{\psi}_{N, \mathrm{cf}}^{\rm ac}$, we can obtain similar theoretical results as in Theorem~\ref{theorem_bias_variance_emp} and Corollary~\ref{theorem_semipar_eff_emp} by appealing to Proposition~\ref{theorem_bias_variance_ghat}, with $\hat{\Omega}$ replaced by $\hat{\Omega}^{\rm ac}$. However, to ensure $\BL_{2, \omegahat^{\rm ac}, k} = o_{\mathbb{P}_{\theta}} (\log^{-1} n)$, we need $\Vert 1 - \hat{g} / g \Vert_{\infty} = o_{\mathbb{P}_{\theta}} (\log^{-1} n)$, a consequence of Lemma~\ref{lemma_op_ac}. A common high-level condition for the latter to hold is that there exists an estimator $\hat{g}$ such that $\Vert 1 - \hat{g} / g \Vert \lesssim n^{- \delta}$ for some $\delta > 0$, which will hold if we assume $g$ is \Holder{} smooth with some positive smoothness index. $\hat{\psi}_{N, \mathrm{cf}}^{\rm ac}$ will be semiparametric efficient under the parallel assumptions as those in Corollary~\ref{theorem_semipar_eff_emp}. This is the route taken in \citet{robins2017minimax}. If $\delta$ is known or the smoothness of $g$ is given, it is surely possible to set $m$ to a value accordingly, which is possibly smaller than $\log n$. However, the new estimators $\hat{\psi}_{n}^{\rm emp}$ or $\hat{\psi}_{N, \mathrm{cf}}^{\rm emp}$ allow $\delta = 0$.

\subsection{Adaptive Efficient Estimation}
\label{section_missing_data_adapt} 

In this section we show that we can use the new empirical HOIF estimators to obtain adaptive semiparametric efficient estimators when $\Theta$ assumes that the functions $b, p$ live in H\"{o}lder balls with sufficient smoothness. Following \cite{robins2008higher,robins2017minimax}, we define the complexity of the model $\mathcal{M} \left( \Theta \right)$ in terms of H\"{o}lder smoothness classes as follows.

\begin{definition}\label{def_Holder}
	A function $x \mapsto h(x)$ with domain a compact subset $D$ of $\mathbb{R}^{d}$ is said to belong to a H\"{o}lder ball $H(\beta, C)$ with H\"{o}lder exponent $\beta > 0$ and radius $C > 0$, if and only if $h$ is uniformly bounded by $C$, all partial derivatives of $h$ up to order $\lfloor \beta \rfloor$ exist and are bounded, and all partial derivatives $\nabla^{\lfloor \beta \rfloor} h$ of order $\lfloor \beta \rfloor$ satisfy 
	\[
	\sup_{x, x + \delta x \in D} \left\vert \nabla^{\lfloor \beta \rfloor} h(x + \delta x) - \nabla^{\lfloor \beta \rfloor} h(x) \right\vert \leq C \Vert \delta x \Vert^{\beta - \lfloor \beta \rfloor}.
	\]
\end{definition} 

To construct adaptive semiparametric efficient estimators over \Holder{} balls we use specific bases that satisfy Conditions~\ref{def_conditions}(1),~\ref{def_conditions}(2), and Condition~\ref{stable} with $q = \infty$ and that additionally give optimal rates of approximation for \Holder{} classes. In particular, we shall assume our basis $\{z_{l}, l = 1, 2, \cdots\}$ has optimal approximation properties in $L_{2} (\mu)$ for \Holder{} balls $H (\beta, C)$ with respect to the Lebesgue measure ($\mu$) i.e., 
\begin{equation}
\sup_{h \in H (\beta, C)} \inf_{\varsigma_{l}} \int_{x \in [0, 1]^{d}} \Big( h (x) - \sum_{l = 1}^{k} \varsigma_{l} z_{l} (x) \Big)^{2} \diff x = O (k^{- 2 \beta / d}). \label{eqn:basis_approx}
\end{equation}
where given any $\{z_l, l \geq 1\}$ satisfying \eqref{eqn:basis_approx} the $O$-notation only depends on the H\"{o}lder radius $C$. For example:
\begin{enumerate}
\item[(i)] The basis consisting of $d$-fold tensor products of B-splines of order $s$ satisfies \eqref{eqn:basis_approx} for all $0 < \beta < s + 1$ \citep{newey1997convergence,belloni2015some};
\item[(ii)] The basis consisting of $d$-fold tensor products of a univariate Daubechies wavelet basis $\varphi (u)$ satisfying 
\begin{equation*}
\int_{[0, 1]} u^{m} \varphi (u) \diff u = 0, \ m = 0, 1, \ldots, M 
\end{equation*}
also satisfies \eqref{eqn:basis_approx} for $\beta < M + 1$ \citep{gine2016mathematical}. 
\end{enumerate}
In addition, both of these bases satisfy Conditions~\ref{def_conditions}(1) and~\ref{def_conditions}(2) for some large but fixed $1 < B < \infty$ \citep{belloni2015some, newey1997convergence} and Condition~\ref{stable} with $q = \infty$ (see, e.g., the comments after Condition~\ref{stable}). 


Then aided by Corollary~\ref{theorem_semipar_eff_emp}, together with the above optimally approximating basis functions for \Holder{} smoothness classes, we immediately have the following result:

\begin{theorem}\label{lemma_holder_bias} 
Assume the following:
\begin{enumerate}
\item[\emph{(1)}] The conditions (1), (3) and (4) of Corollary~\ref{theorem_semipar_eff_emp} hold and $\{z_l, l \geq 1\}$ satisfy Condition~\ref{def_conditions}(1), Condition~\ref{def_conditions}(2), Condition~\ref{stable} with $q = \infty$, and \eqref{eqn:basis_approx}.

\item[\emph{(2)}]  $b$ and $\hat{b}$ lie in $H(\beta_{b}, C_{b})$, and $p$ and $\hat{p}$ lie in $H(\beta_{p}, C_{p})$ with $C_p > \frac{1}{\sigma}$, where recall that $\sigma$ is the lower bound of $g$.

\item [\emph{(3)}] $\beta = \frac{\beta_{b} + \beta _{p}}{2}$ satisfies $\frac{d}{4} < \beta < \bmax$ for some known $\bmax$. 
\end{enumerate}

Then the estimators $\widehat{\psi}_{n}^{\rm emp}$ and $\widehat{\psi}_{N, \mathrm{cf}}^{\rm emp}$ satisfy $$\TB_{k} (\theta) = O_{\mathbb{P}_{\theta}} (k^{-\frac{2\beta}{d}}) = o_{\mathbb{P}_{\theta}} (n^{-1/2}),$$ and thus $\widehat{\psi}_{N, \mathrm{cf}}^{\rm emp}$ reaches the semiparametric efficiency bound adaptively (i.e. knowing neither $\beta_{b}$ nor $\beta_{p}$ as long as condition (3) is met). If conditions (3') and (4') of Corollary~\ref{theorem_semipar_eff_emp} hold instead of (3) and (4), then both $\widehat{\psi}_{n}^{\rm emp}$ and $\widehat{\psi}_{N, \mathrm{cf}}^{\rm emp}$ are adaptive $\sqrt{n}$-consistent estimators of $\psi (\theta)$.
\end{theorem}

The proof of Theorem~\ref{lemma_holder_bias} is straightforward. Since $\beta > \frac{d}{4}$, by condition (2) of Theorem~\ref{lemma_holder_bias}, we have $\TB_{k} (\theta) = O_{\P_{\theta}} (k^{- \frac{2 \beta}{d}})$. Choosing $k \asymp \frac{n}{\log^{3} n}$, it is easy to see that $\TB_{k} (\theta) = o_{\P_{\theta}} (n^{- 1 / 2})$. By applying Corollary~\ref{theorem_semipar_eff_emp}, we immediately obtain the desired results.  By presenting Corollary~\ref{theorem_semipar_eff_emp} separately from Theorem~\ref{lemma_holder_bias}, we hope that it has been made clear that \Holder{} smoothness assumptions mainly contribute to ensure that the truncation bias $\TB_{k} (\theta)$ is sufficiently small. As a result, various other types of smoothness conditions, such as Sobolev classes with $\beta$-th order weak derivatives bounded in the $L_{2} (\mathbb{P})$-norm, often denoted by $W^{\beta, 2}$, can be treated similarly under our framework. 

Theorem~\ref{lemma_holder_bias} tells us that $\widehat{\psi}_{N, \mathrm{cf}}^{\rm emp}$ is semiparametric efficient at any $\mathbb{P}_{\theta}$ that satisfies conditions of the theorem. Moreover, this result is adaptive over any $\beta\in (\frac{d}{4}, \bmax)$. Interestingly, the knowledge of an upper bound $\bmax$ only becomes crucial in constructing a sequence of basis functions $\{z_l, \ l \geq 1\}$ satisfying \eqref{eqn:basis_approx} and is not required anywhere else in the analysis. As mentioned in the end of the previous section, analogous results for $\hat{\psi}_{N, \mathrm{cf}}^{\rm ac}$ can be attained with additional smoothness conditions on $g$ and $\ghat$ [also see, e.g., \citet[Theorem 8.2]{robins2017minimax} (with the proof corrected in \citet{robins2023minimax})]. But as we have stressed throughout this paper, the result in Theorem~\ref{lemma_holder_bias} is completely oblivious to (1) the smoothness conditions on $g$ including absolute continuity and (2) the need of constructing an estimator $\ghat$ of $g$.

The statistical message of Theorem~\ref{lemma_holder_bias} is somewhat surprising. When $b$ and $p$ satisfy (2) in Theorem~\ref{lemma_holder_bias}, the following estimators $\bhat, \phat$ will do so as well \citep{vaart2006oracle} when the basis $\{z_l, l \geq 1\}$ are compactly supported Daubechies wavelets of sufficient regularity (at least $2 \bmax$): $\widehat{b} (x) = \sum_{l = 1}^{k_{b}} \widehat{\eta}_{l} z_{l} (x)$ and $\widehat{p} (x) = 1 / \widehat{\pi} (x)$ with $\widehat{\pi} (x) = \sum_{l = 1}^{k_{\pi}} \widehat{\alpha}_{l} z_{l} (x)$ with parameters estimated by least squares and $k_{b}$ and $k_{\pi}$ chosen by cross validation, all done in the training sample. Note, however, the choices $\widehat{b} (x) = 0$ and $1 / \widehat{p} (x) = c$ for $c \in (0, 1)$ still satisfy the conditions in the second part of Theorem~\ref{lemma_holder_bias}, as $\Vert \widehat{b} - b \Vert_{\infty}$ and $\Vert \widehat{p} - p \Vert_{\infty}$ are both $O_{\mathbb{P}_{\theta}} (1)$. Thus following Remark~\ref{rem:basis}, we obtain the surprising conclusion that our estimators $\bhat$ and $\phat$ do not even need to be consistent for $b$ and $p$ to obtain a $\sqrt{n}$-consistent estimator $\widehat{\psi}_{N, \mathrm{cf}}^{\rm emp}$ of $\psi (\theta)$, as long as $\beta > \frac{d}{4}$. This is an asymptotic ``free-lunch''. In fact, we can even ignore the range of $\widehat{p}$, choose $\widehat{b} = \widehat{p} = 0$ and still preserve $\sqrt{n}$-consistency. The explanation of this fact is that $\hat{\psi}_{N, \mathrm{cf}}^{\rm emp}$ is ``multiply robust'' under the \Holder{} condition in the following sense: Even when we choose $\widehat{b} = \widehat{p} = 0$, and hence $\widehat{\psi}, \IF_{1}$, and $\widehat{\psi}_{1}$ are all identically zero, nonetheless $\sum_{j = 2}^{m(n)} \widehat{\mathbb{IF}}_{j, j, k(n)}(\omegahat^{\rm emp})$ is an estimate of $\int \mathsf{\Pi}_{g, \bar{z}_{k}} [b] (x) \cdot \mathsf{\Pi}_{g, \bar{z}_{k}} [p] (x) \cdot g (x) \diff x$ with bias $o_{\mathbb{P}_{\theta}} (n^{-1 / 2})$ for $\omegahat^{\rm emp}$ as in \eqref{matching}.

Finally, the following two remarks discuss some further implications and observations regarding Theorem~\ref{lemma_holder_bias}.

\begin{remark}
Suppose model $\calM (\Theta)$ restricts $b$ and $p$ to lie in pre-specified \Holder{} balls $H(\beta_{b}, C_{b})$ and $H(\beta_{p}, C_{p})$. \cite{robins2008higher} show that the minimax rate for estimating $\psi (\theta)$ when $g$ is known is $n^{-1/2} + n^{-\frac{4 \beta}{4 \beta + d}}$, with $\beta = \frac{\beta_{b} + \beta_{p}}{2}$. Hence when $\beta < \frac{d}{4}$, the minimax rate is slower than $n^{-1/2}$ regardless of whether $g$ is known or unknown in the model $\calM \left( \Theta \right)$. However, even in such a model there exist parameters, $\theta^{\ast} = \left( b^{\ast}, p^{\ast}, g^{\ast} \right) \in \Theta$ in which $b^{\ast}$ and $p^{\ast}$ happen to lie in smaller \Holder{} balls $H(\beta_{b}^{\ast}, C_{b}^{\ast})$ and $H(\beta_{p}^{\ast}, C_{p}^{\ast})$ with $\frac{\beta_{b}^{\ast} + \beta_{p}^{\ast}}{2} > \frac{d}{4}$. Thus $\widehat{\psi}_{N, \mathrm{cf}}^{\rm emp}$ will be semiparametric efficient at $\theta^{\ast}$ under the assumptions in Theorem~\ref{lemma_holder_bias}, even though it will converge to $\psi (\theta)$ at a rate slower than $n^{-1/2}$ at nearly all $\theta \in \Theta$.
\end{remark}

\begin{remark}
\label{rem:smooth_hat}
Note even when $b$ and $p$ lie in \Holder{} balls $H (\beta_{b}, C_{b})$ and $H(\beta_{p}, C_{p})$ with $\beta > \frac{d}{4}$, we still need their estimates $\widehat{b}$ and $\widehat{p}$ to lie in these \Holder{} balls with probability approaching one to ensure $\TB_{k} (\theta) = o_{\mathbb{P}_{\theta}} (n^{-1/2})$; see condition (2) of Theorem~\ref{lemma_holder_bias}. This may place restrictions on the machine learning algorithms we can use to estimate $b$ and $p$. As an example, suppose (i) we use multiple nonparametric or machine learning algorithms to construct candidate estimators and then use cross validation or aggregation to build a data-adaptive candidate and (ii) the aforementioned series estimators $\widehat{b} (x) = \sum_{l=1}^{k_{b}}\widehat{\eta}_{l} z_{l} (x)$ and $\widehat{p} (x) = 1 / \widehat{\pi} (x)$ with $\widehat{p} (x) = \sum_{l = 1}^{k_{p}} \widehat{\alpha}_{l} z_{l} (x)$ are included among the candidates. If the only candidates were these series estimators, we know that $\TB_{k} (\theta) = o_{\mathbb{P}_{\theta}} (n^{-1/2})$ for $k \asymp \frac{n}{\log^{3} n}$ and our estimator would be semiparametric efficient. Nonetheless it may be the case at the particular law $\theta^{\ast} = \left( b^{\ast}, p^{\ast}, g^{\ast} \right)$ that generated the data, another pair of candidates $\widetilde{b}$ and $\widetilde{p}$ are chosen with high probability over these series estimators $\widehat{b}$ and $\widehat{p}$ because for these laws, $\widetilde{b}$ and $\widetilde{p}$ converge to $b$ and $p$ at faster rates than the series estimators. However, faster rates of convergence do not imply that the associated truncation bias $\TB_{k} (\theta) = \int \diff x g(x) (\mathsf{I}-\mathsf{\Pi}_{g, \bar{z}_{k}}) [b-\widetilde{b}] (x) (\mathsf{I} - \mathsf{\Pi}_{g, \bar{z}_{k}}) [p - \widetilde{p}] (x)$ is less than the truncation bias of the series estimator and thus no guarantee it is $o_{\mathbb{P}_{\theta}} (n^{-1/2})$. Fortunately, based on the results in Corollary~\ref{theorem_semipar_eff_emp} or Theorem~\ref{lemma_holder_bias}, we only need data-adaptive consistent estimators of $b$ and $p$ without any requirement on their convergence rates for semiparametric efficiency. Such weaker requirement makes it much easier to find data-adaptive estimators of $b$ and $p$ that belong to certain H\"{o}lder balls. We provide a simple example in Appendix~\ref{app:adapt}.

\end{remark}

\section{Extensions to Doubly Robust Functionals}
\label{section_general} 
In this section we extend our results to incorporate a general class of doubly robust functionals studied in \cite{robins2008higher}. We consider $N$ i.i.d observations $W= (X, V)$ from a law $\mathbb{P}_{\theta}$ with $\theta \in \Theta$ and wish to make inference on a functional $\mathcal{\chi} (\mathbb{P}_{\theta}) = \psi (\theta)$. In this section, we further assume the following.
\begin{customthm}{DR}\leavevmode
\label{cond:DR}
\begin{itemize}
\item[(1)] For all $\theta \in \Theta$, the distribution of $X$ is supported on a compact set in $\mathbb{R}^{d}$ which we take to be $[0, 1]^{d}$ and has a density $f (x)$ with respect to Lebesgue measure.

\item[(2)] The parameter $\theta$ contains components $b = b (\cdot)$ and $p = p (\cdot)$, $b: [0, 1]^{d} \rightarrow \mathbb{R}$ and $p: [0, 1]^{d} \rightarrow \mathbb{R}$ such that the functional $\psi (\theta)$ of interest has a first order influence function $\mathbb{IF}_{1, \psi} (\theta) = N^{-1} \sum_{i} \mathrm{IF}_{1, \psi, i} (\theta)$ over $N$ i.i.d. observations, where 
\begin{align}
\mathrm{IF}_{1, \psi} (\theta) & = H (b, p) - \psi (\theta), \text{ with } \label{Hdef} \\
H (b, p) & \coloneqq b (X) p (X) h_{1} (W) + b (X) h_{2} (W) + p (X) h_{3} (W) + h_{4} (W) \notag \\
& = B P H_{1} + B H_{2} + P H_{3} + H_{4}, \notag
\end{align}
and the known functions $h_{1} (\cdot), h_{2} (\cdot), h_{3} (\cdot), h_{4} (\cdot)$ do not depend on $\theta$. Furthermore, $h_{1} (\cdot)$ is either nowhere negative or nowhere positive almost surely.

\item[(3)] $\theta = (b, p, g) \in \Theta$ where $\Theta = \Theta_{b} \times \Theta_{p} \times \Theta_{g}$ with $g (x) = \mathbb{E}_\theta [| H_{1} | \vert X = x] f (x)$ bounded away from zero and infinity and absolutely continuous w.r.t. to Lebesgue measure on the support of $X$.

\item[(4)] The model $\mathcal{M} (\Theta)$ for $\mathbb{P}_{\theta}$ satisfies \eqref{eqn:theta_def} and is locally nonparametric in the sense that the tangent space at each $\mathbb{P}_{\theta} \in \mathcal{M} (\Theta)$ is all of $L_{2} (\mathbb{P}_{\theta})$.
\end{itemize}
\end{customthm}

The missing data example considered throughout this paper is the special case with $H_{1} = -A, H_{2} = 1, H_{3} = A Y, H_{4} = 0$, $p (X) = 1 / \mathbb{P}_\theta \left( A = 1 | X \right), b (X) = \mathbb{E}_\theta [Y | A = 1, X], g (X) = \mathbb{E}_\theta [A | X] f (X)$. $H (b, p)$ is doubly robust for $\psi (\theta)$ in the following sense: 
\begin{equation*}
\mathbb{E}_\theta \left[ H (b, p) \right] = \mathbb{E}_\theta \left[ H (b, p^{\ast}) \right] = \mathbb{E}_\theta \left[ H (b^{\ast}, p) \right] = \psi (\theta)
\end{equation*}
for any $\theta \in \Theta$ and functions $b^{\ast} (x)$ and $p^{\ast} (x)$. Specifically \citet{robins2008higher} prove the following result: 
\begin{lemma}[Double-Robustness (Theorem 3.2 of \citet{robins2008higher})]
\label{thm:dr}
Suppose that Condition~\ref{cond:DR} holds. Then
\begin{align*}
& \psi (\theta) = \E_{\theta} [H_{4}] - \E_\theta [BPH_{1}] = \E_{\theta} [H_{4}] - (-1)^{I \{h_{1} (W) \leq 0\}} \int b (x) p (x) g (x) \diff x, \\
& \E_{\theta} \left[ H_{1} B + H_{3} | X \right] = \E_{\theta} \left[ H_{1} P + H_{2} | X \right] = 0 \textrm{ w.p.1, and } \\
& \E_{\theta} \left[ H \left( b^{\ast}, p^{\ast} \right) \right] - \E_{\theta} \left[ H \left( b, p \right) \right] = (-1)^{I \{h_{1} (W) \leq 0 \}} \Big\{ \int \left[ b - b^{\ast} \right] (x) \left[ p - p^{\ast} \right] (x) g (x) \diff x \Big\}.
\end{align*}
\end{lemma}

The development in \citet[Theorem 3.2 and Lemma 3.3]{robins2008higher} shows that the statistical results we have obtained when $\psi (\theta)$ is the mean response subject to MAR in previous sections can be extended to parameters defined as in Condition~\ref{cond:DR}. To be more precise, we state the following theorem, which is the last major theoretical result of this paper.
\begin{theorem}
\label{thm:general}
Suppose that Condition~\ref{cond:DR} holds and redefine $\varepsilon_{b} = B H_{1} + H_{3}$, $\varepsilon_{p} = H_{1} P + H_{2}$, $g (x) = \E_\theta [|H_{1}| \vert X = x] f (x)$. Then the conclusions of Proposition~\ref{theorem_bias_general_omega}, Proposition~\ref{theorem_bias_variance_ghat}, Theorem~\ref{theorem_bias_variance_emp}, Corollary~\ref{theorem_semipar_eff_emp} and Theorem~\ref{lemma_holder_bias} continue to hold under the same assumptions on the redefined nuisance parameters $\theta =(b, p, g)$, the corresponding nuisance parameter estimates $(\hat{b}, \hat{p})$, the basis $\bar{z}_{k}$ and the population and sample Gram matrices $\Omega$ and $\hat{\Omega}$.
\end{theorem}

\begin{remark}
Lemma~\ref{thm:dr} and Theorem~\ref{thm:general} can be further extended to the entire class of parameters with the so-called mixed bias property \citep{rotnitzky2021characterization}, which subsumes the class of doubly robust functionals \citep{robins2008higher} studied here; see \citet{liu2024assumption} for more details. However, we decide not to state results for this larger class of parameters. This is only because, if we do so, we will have to introduce extra notation that is orthogonal to the main message of this paper.
\end{remark}

\section{Simulation experiments}\label{section_simulations} 

In this section, we choose the marginal mean of $Y$ under MAR, $\psi = \Etheta \left[ A Y / \pi (X) \right] = \int b (x) p (x) g (x) \diff x$, as our target estimand. The main goal of this section is to demonstrate the advantage in finite sample performance of the empirical HOIF estimators $\hat{\IIFF}_{2, 2, k} (\omegahat^{\rm emp})$ and $\hat{\psi}_{2, k}^{\rm emp}$, compared to that of $\hat{\IIFF}_{2, 2, k} (\omegahat^{\rm ac})$ and $\hat{\psi}_{2, k}^{\rm ac}$ when $g$ is not very smooth. Based on the theoretical results in this paper, we expect that $\hat{\psi}_{2, k}^{\rm emp}$ should outperform $\hat{\psi}_{2, k}^{\rm ac}$ because the bias of $\hat{\psi}_{2, k}^{\rm emp}$ does not depend on the smoothness of the covariate density $g$. Moreover, unlike $\hat{\psi}_{2, k}^{\rm ac}$ relying on $\omegahat^{\rm ac}$, a quantity computed from high-dimensional numerical integration with respect to the estimated density $\hat{g}$, $\hat{\psi}_{2, k}^{\rm emp}$ completely bypasses this step and hence is much easier to compute. The goal of the simulation studies here is to demonstrate that $\hat{\psi}_{2, k}^{\rm emp}$ is a better choice than the HOIF estimators relying on density estimation, in particular when $g$ has low regularity.

Another related estimator that requires estimating $g$ but not numerical integration is $\hat{\psi}_{2, k} (\hat{g}) \coloneqq \hat{\psi}_{1} + \widehat{\mathbb{IF}}_{2, 2, k} (\hat{g})$, where 
\begin{equation*}
\widehat{\mathbb{IF}}_{2, 2, k} (\hat{g}) = \frac{(n - 2)!}{n!} \sum_{\bar{i}_{2} \in I_{n, 2}} \left[ A (Y - \hat{b} (X)) \bar{z}_{k}^{\top} (X) / \hat{g}^{1 / 2} (X) \right]_{i_{1}} \left[ \bar{z}_{k} (X) (1 - A \hat{p} (X)) / \hat{g}^{1 / 2} (X) \right]_{i_{2}} 
\end{equation*}
and it has been considered in \cite{robins2009quadratic} and \cite{mukherjee2016adaptive}. Similar to $\hat{\psi}_{2, k}^{\rm ac}$, we also expect $\hat{\psi}_{2, k} (\hat{g})$ to have larger bias than $\hat{\psi}_{2, k}^{\rm emp}$, but for slightly different reasons, which we now briefly explain. Consider the bias of $\hat{\psi}_{2, k} (\hat{g})$:
\begin{align*}
\Etheta \left[ \hat{\psi}_{2, k} (\hat{g}) - \psi (\theta) \right] = \underbrace{\Etheta \left[ \hat{\psi}_{2, k} (\hat{g}) - \hat{\psi}_{2, k} (g) \right]}_{\eqqcolon \, \EB_{2, k} (\hat{g})} + \underbrace{\Etheta \left[ \hat{\psi}_{2, k} (g) - \psi (\theta) \right]}_{\eqqcolon \, \TB_{k} (g)}.
\end{align*}
For $\hat{\psi}_{2, k} (\hat{g})$, not only its estimation bias $\EB_{2, k} (\hat{g})$, but also its truncation bias $\TB_{k} (g)$, depends on the smoothness of $g$. To see why this is the case for $\TB_{k} (g)$, let us rewrite $\TB_{k} (g)$ as follows
\begin{align*}
\TB_{k} (g) & = \Etheta \left[ \hat{\psi}_{2, k} (g) - \psi (\theta) \right] = \int \sPi_{g, \bar{z}_{k} g^{-1/2}}^{\perp} [b - \hat{b}] (x) \cdot \sPi_{g, \bar{z}_{k} g^{-1/2}}^{\perp} [p - \hat{p}] (x) \cdot g (x) \diff x \\
& = \int \sPi_{\Leb, \bar{z}_{k}}^{\perp} [(b - \hat{b}) g^{1 / 2}] (x) \cdot \sPi_{\Leb, \bar{z}_{k}}^{\perp} [(p - \hat{p}) g^{1 / 2}] (x) \diff x \\
& = \int \sPi_{\Leb, \bar{z}_{k}}^{\perp} [b^{\dag} - \hat{b}^{\dag}] (x) \cdot \sPi_{\Leb, \bar{z}_{k}}^{\perp} [p^{\dag} - \hat{p}^{\dag}] (x) \diff x
\end{align*}
where $\sPi_{\Leb, \bar{z}_{k}} [h] (\cdot) \coloneqq \bar{z}_{k} (\cdot)^{\top} \{\int \bar{z}_{k} (x) \bar{z}_{k} (x)^{\top} \diff x\}^{-1} \int \bar{z}_{k} (x) h (x) \diff x$ is the population projection operator onto the linear span of $\bar{z}_{k}$ with respect to the Lebesgue measure and $h^{\dag} \coloneqq h g^{1 / 2}$ for any function $h$. Even when the original residuals $b - \hat{b}$ and $p - \hat{p}$ are sufficiently smooth, the smoothness of the ``new'' residuals $b^{\dag} - \hat{b}^{\dag}$ and $p^{\dag} - \hat{p}^{\dag}$, after multiplied by a non-smooth function $g^{1 / 2}$, can be as non-smooth as $g^{1 / 2}$. Thus the strategy of dividing by $g^{1 / 2}$ to avoid high dimensional numerical integration may lead to very large truncation bias when $g$ is nonsmooth, on top of the larger estimation bias (in order) because of the dependence of $\EB_{2, k} (\hat{g})$ on $\Vert \hat{g} - g \Vert$.

In terms of the simulation setup, we consider the following data generating mechanism: 
\begin{itemize}
\item We draw $X_{j}$ for $j = 1, \ldots, d$, with correlations between every two dimensions but the same marginal density $f$ supported on $[0, 1]$ with $f \in \text{H\"{o}lder}(\beta_{g} = 0.1)$ (see Appendix~\ref{app:holder} for the concrete form of $f$), according to the algorithm described in Appendix~\ref{app:multiX}. We focus on $d = 4$ such that the function $\mathsf{kde}$ from $\mathsf{R}$ package $\mathsf{ks}$ can still be used to estimate $f(\cdot \vert A = 1)$ and hence $g$.  In fact, we could not carry out our simulation study in a timely fashion for $d \geq 5$ because the $\mathsf{kde}$ function failed to return a kernel density estimate of $g$, even after running for more than 4 days in the high performance computing (HPC) cluster which we used to conduct the simulation study. The bandwidth for estimating $f(\cdot \vert A = 1)$ is selected by smoothed cross-validation \citep{jones1991simple, duong2005cross}, the default setup of $\mathsf{kde}$. 

\item We then draw $Y$ and $A$ conditioning on $X$ according to the following data generating mechanism: 
\begin{align*}
Y & \sim b(X) + N(0, 1) = \sum_{j = 1}^{d} \zeta_{b, j} h_{b} (X_{j}; 0.25) + N(0, 1), \\
A & \sim \text{Bernoulli} \Big( p^{-1} (X) = \pi (X) = \text{expit} \Big\{ \sum_{j = 1}^{d} \zeta_{p, j} h_{p} (X_{j}; 0.25) \Big\} \Big),
\end{align*}
where $h_{b} (\cdot; 0.25)$ and $h_{p} (\cdot; 0.25)$ have the same form as defined in Appendix~\ref{app:holder} and thus both belong to \Holder{} classes with smoothness 0.25. The numerical values for $(\zeta_{b, j}, \zeta_{p, j})_{j = 1}^{d}$ are provided in Table~\ref{tab:s1}. We observe $Y$ if and only if $A = 1$ in the observed data. Finally, note that the smoothness of $g$ is much lower than those of $b$ and $p$.
\end{itemize}

The key findings of the simulation study can be summarized as follows:

\begin{itemize}
\item $\widehat{\mathbb{IF}}_{2, 2, k} (\omegahat^{\rm emp})$ can correct the bias of the first order estimator $\hat{\psi}_{1}$ without inflating the standard error and it takes shorter time to compute $\widehat{\mathbb{IF}}_{2, 2, k} (\omegahat^{\rm emp})$ than $\widehat{\mathbb{IF}}_{2, 2, k} (\omegahat^{\rm ac})$.

\item The difference between $\hat{\psi}_{2, k}^{\rm emp} = \hat{\psi}_{1} + \widehat{\mathbb{IF}}_{2, 2, k} (\omegahat^{\rm emp})$ and the oracle $\hat{\psi}_{2, k} (\Omega) = \hat{\psi}_{1} + \widehat{\mathbb{IF}}_{2, 2, k} (\Omega)$ is smaller than the difference between $\hat{\psi}_{2, k}^{\rm ac} = \hat{\psi}_{1} + \widehat{\mathbb{IF}}_{2, 2, k} (\omegahat^{\rm ac})$ and $\hat{\psi}_{2, k} (\Omega)$ when $g$ is not smooth. This is consistent with our theoretical results: unlike the estimation bias of $\hat{\psi}_{2, k}^{\rm ac}$ (see Proposition~\ref{theorem_bias_variance_ghat}), the estimation bias of $\hat{\psi}_{2, k}^{\rm emp}$ does not depend on the smoothness of $g$ (see Theorem~\ref{theorem_bias_variance_emp}).

\item $\widehat{\mathbb{IF}}_{2, 2, k} (\hat{g})$ does not correct as much bias as the other estimators including $\widehat{\mathbb{IF}}_{2, 2, k} (\omegahat^{\rm emp})$, $\widehat{\mathbb{IF}}_{2, 2, k} (\omegahat^{\rm ac})$ and the oracle $\widehat{\mathbb{IF}}_{2, 2, k} (\Omega)$.
\end{itemize}


When computing $\widehat{\mathbb{IF}}_{2, 2, k}(\omegahat^{\rm ac})$, $\omegahat^{\rm ac}$ was evaluated by Monte Carlo integration over $L = 10^{7}$ independent draws of $X_{j}$ for $j = 1, \ldots, d$ from $\hat{f} (\cdot \vert A = 1)$. We choose $L$ large enough such that $\widehat{\mathbb{IF}}_{2, 2, k}(\omegahat^{\rm ac})$ stabilizes. In terms of the basis functions, we choose $\bar{z}_{k} = \{\sqrt{2^8} \varphi (2^8 x_{j} - l), j = 1, \cdots, 4\}$: we transform each dimension of $X$ by the dilated and shifted $\varphi$ at resolution $2^{8}$, with $\varphi$ the D12/db6 father wavelet function, and then we concatenate the transformed functions across all four dimensions. Here $k = (2^{8} + 4) \cdot d = 260 \cdot 4 = 1,040$. The bases used for estimating $\psi (\theta)$ are different from those used to generate the true nuisance functions, although both bases are based on Daubechies wavelets.

To compare the estimation bias of $\widehat{\mathbb{IF}}_{2, 2, k}(\omegahat^{\rm emp})$ versus $\widehat{\mathbb{IF}}_{2, 2, k}(\omegahat^{\rm ac})$, we also need to know the value of the oracle estimator $\widehat{\mathbb{IF}}_{2, 2, k}(\Omega)$ with the true $\Omega$ numerically evaluated by computing $\widehat{\Omega}^{\rm emp}$ from $L = 10^{7}$ independent samples drawn from the true data generating process. Again, we choose $L$ large enough such that $\widehat{\mathbb{IF}}_{2, 2, k}(\Omega)$ stabilizes. 

We consider two different methods for estimating the nuisance functions $b$ and $1 / p$: (1) by the following generalized linear models (GLMs) so
\begin{align*}
\tilde{b}_{glm} (x) = \sum_{j = 1}^{d} \alpha_{b, glm, j} x_{j}, \tilde{p}_{glm} (x)^{-1} = \tilde{\pi}_{glm} (x) = \text{expit} \left\{ \sum_{j = 1}^{d} \alpha_{p, glm, j} x_{j} \right\}
\end{align*}
and (2) by the following generalized additive models (GAMs) \citep{hastie1986generalized}
\begin{align*}
\tilde{b}_{\rm gam} (x) = \sum_{j = 1}^{d} \alpha_{b, \rm gam, j} \mathsf{s} (x_{j}), \tilde{p}_{\rm gam} (x)^{-1} = \tilde{\pi}_{\rm gam} (x) & = \text{expit} \Big\{ \sum_{j = 1}^{d} \alpha_{p, \rm gam, j} \mathsf{s} (x_{j}) \Big\},
\end{align*}
where $\mathsf{s} (\cdot)$ is the smoothing spline transformation wherein the smoothing parameters are selected by generalized cross validation, the default setup in \texttt{gam} function from R package \texttt{mgcv} \citep{wood2016smoothing}. 

We compare different estimators with the same $k$, but across the following training sample sizes $n_{\rm tr} = 25000, 100000, 200000$ and estimation sample sizes $n = 25000, 100000, 200000$. All our simulation results are conditioning on one single training sample at each $n_{\rm tr}$. In terms of the computational efficiency, we have the following:
\begin{itemize}
\item On average, it only takes about 20 seconds, 1 minute and 2 minutes (for $n = 25000, 100000$ and $200000$) to compute $\widehat{\mathbb{IF}}_{2, 2, k}(\omegahat^{\rm emp})$ and $\widehat{\mathbb{IF}}_{2, 2, k}(\omegahat^{\rm ac})$ from the estimation sample after $\omegahat^{\rm emp}$, $\hat{g}$, and $\omegahat^{\rm ac}$ have been computed from the training sample. $\widehat{\mathbb{IF}}_{2, 2, k}(\hat{g})$ is faster to compute because it does not involve large matrix multiplication. But later we will show that the statistical performance of $\widehat{\mathbb{IF}}_{2, 2, k}(\hat{g})$ is much worse than $\widehat{\mathbb{IF}}_{2, 2, k}(\omegahat^{\rm emp})$ or $\widehat{\mathbb{IF}}_{2, 2, k}(\omegahat^{\rm ac})$.

\item In the training sample, it takes about 5 hours, one day and two days to compute $\hat{g}$ for $n_{\rm tr} = 25000, 100000, 200000$, and 4-5 hours to compute $\omegahat^{\rm ac}$ at $k = 1,040$ given $\hat{g}$. It takes about 5 minutes, 20 minutes and 40 minutes to compute $\omegahat^{\rm emp}$. 
\end{itemize}
Thus to summarize, $\widehat{\mathbb{IF}}_{2, 2, k}(\omegahat^{\rm emp})$ is the most efficient to compute among $\widehat{\mathbb{IF}}_{2, 2, k}(\omegahat^{\rm emp})$, $\widehat{\mathbb{IF}}_{2, 2, k}(\omegahat^{\rm ac})$ and $\widehat{\mathbb{IF}}_{2, 2, k}(\hat{g})$ (if also considering the time of estimating $\hat{g}$).

Finally, the results comparing the performance of $\widehat{\mathbb{IF}}_{2, 2, k}(\omegahat^{\rm emp})$, $\widehat{\mathbb{IF}}_{2, 2, k}(\omegahat^{\rm ac})$ and $\widehat{\mathbb{IF}}_{2, 2, k}(\hat{g})$ (and also $\hat{\psi}_{2, k}^{\rm emp}$, $\hat{\psi}_{2, k}^{\rm ac}$ and $\widehat{\psi}_{2, k}(\hat{g})$) are displayed in Figures~\ref{fig:1} and~\ref{fig:2} and Tables~\ref{tab:1} to~\ref{tab:4}. To save space, we defer all the tables to Appendix~\ref{app:tables}. Note that the pairwise differences among the last four columns of Table~\ref{tab:1} (resp. Table~\ref{tab:3}) should be identical to those among the last four columns of Table~\ref{tab:2} (resp. Table~\ref{tab:4}). We summarize our findings below:
\begin{itemize}
\item On the upper-left panel of Figure~\ref{fig:1}, we compare $\widehat{\mathbb{IF}}_{2, 2, k} (\Omega)$, $\widehat{\mathbb{IF}}_{2, 2, k} (\omegahat^{\rm emp})$, $\widehat{\mathbb{IF}}_{2, 2, k} (\omegahat^{\rm ac})$ and $\widehat{\mathbb{IF}}_{2, 2, k} (\hat{g})$ when the nuisance functions $b$ and $1 / p$ are estimated by GLM. The error bars represent the inter-90\%-quantiles out of 100 Monte Carlo repetitions. As expected, when the estimation sample size increases (from left to right within each column of every panel), the variability of the corresponding $\widehat{\mathbb{IF}}_{2, 2, k}$ decreases, as the error bars become narrower. The Monte Carlo distributions of the oracle $\widehat{\mathbb{IF}}_{2, 2, k} (\Omega)$ (grey dots and error bars) and $\widehat{\mathbb{IF}}_{2, 2, k} (\omegahat^{\rm emp})$ (blue dots and error bars) are very close, as the error bars for these two statistics are almost on top of each other. However, the distribution of $\widehat{\mathbb{IF}}_{2, 2, k} (\omegahat^{\rm ac})$ (purple dots and error bars) is quite different from that of the oracle $\widehat{\mathbb{IF}}_{2, 2, k} (\Omega)$ and $\widehat{\mathbb{IF}}_{2, 2, k} (\omegahat^{\rm emp})$. The difference between $\widehat{\mathbb{IF}}_{2, 2, k} (\hat{g})$ (green dots and error bars) and the other statistics is even more striking.

\ \ \ \ On the upper-right panel of Figure~\ref{fig:1}, the nuisance functions $b$ and $1 / p$ are estimated by GAM. The difference between $\widehat{\mathbb{IF}}_{2, 2, k} (\hat{g})$ and all the other statistics is obvious. But from this panel alone, it is hard to distinguish between $\widehat{\mathbb{IF}}_{2, 2, k} (\Omega)$, $\widehat{\mathbb{IF}}_{2, 2, k} (\omegahat^{\rm emp})$ and $\widehat{\mathbb{IF}}_{2, 2, k} (\omegahat^{\rm ac})$. Therefore in Figure~\ref{fig:2}, we plot everything as in Figure~\ref{fig:1}, but discard the results of $\widehat{\mathbb{IF}}_{2, 2, k} (\hat{g})$. Thus Figure~\ref{fig:2} is a zoom-in of Figure~\ref{fig:1}. From the upper-right panel of Figure~\ref{fig:2}, we can now clearly observe that the empirical distribution of $\widehat{\mathbb{IF}}_{2, 2, k} (\omegahat^{\rm emp})$ (blue dots and error bars) is closer to the empirical distribution of $\widehat{\mathbb{IF}}_{2, 2, k} (\Omega)$ (grey dots and error bars) than that of $\widehat{\mathbb{IF}}_{2, 2, k} (\omegahat^{\rm ac})$ (purple dots and error bars). To further highlight this observation, we also display the empirical distributions of $\widehat{\mathbb{IF}}_{2, 2, k} (\omegahat^{\rm emp}) - \widehat{\mathbb{IF}}_{2, 2, k} (\Omega)$ and $\widehat{\mathbb{IF}}_{2, 2, k} (\omegahat^{\rm ac}) - \widehat{\mathbb{IF}}_{2, 2, k} (\Omega)$ in Figure~\ref{fig:3} and it is apparent that the empirical distribution of $\widehat{\mathbb{IF}}_{2, 2, k} (\omegahat^{\rm emp}) - \widehat{\mathbb{IF}}_{2, 2, k} (\Omega)$ is much closer to 0.

\ \ \ \ All the above observations from the upper panels of Figures~\ref{fig:1} and~\ref{fig:2} can also be made from Tables~\ref{tab:1} and~\ref{tab:3}, in which we display the Monte Carlo averages and standard deviations of different versions of $\widehat{\mathbb{IF}}_{2, 2, k}$ across different training and estimation sample sizes. For example, the Monte Carlo averages of $\widehat{\mathbb{IF}}_{2, 2, k} (\omegahat^{\rm emp})$ are always closer to the corresponding Monte Carlo averages of the oracle $\widehat{\mathbb{IF}}_{2, 2, k} (\Omega)$ than those of $\widehat{\mathbb{IF}}_{2, 2, k} (\omegahat^{\rm ac})$ in both Table~\ref{tab:1} (nuisance functions estimated by GLM) and Table~\ref{tab:3} (nuisance functions estimated by GAM); in addition the Monte Carlo standard deviations of $\widehat{\mathbb{IF}}_{2, 2, k} (\omegahat^{\rm emp})$ are always smaller than those of $\widehat{\mathbb{IF}}_{2, 2, k} (\omegahat^{\rm ac})$.

\item On the lower panels of Figure~\ref{fig:1}, we compare the bias of estimating $\psi$ before ($\hat{\psi}_{1}$, black dots and error bars) or after bias correction ($\hat{\psi}_{2, k} (\Omega) = \hat{\psi}_{1} + \widehat{\mathbb{IF}}_{2, 2, k} (\Omega)$ grey dots and error bars, $\hat{\psi}_{2, k}^{\rm emp} = \hat{\psi}_{1} + \widehat{\mathbb{IF}}_{2, 2, k} (\omegahat^{\rm emp})$ blue dots and error bars, $\hat{\psi}_{2, k}^{\rm ac} = \hat{\psi}_{1} + \widehat{\mathbb{IF}}_{2, 2, k} (\omegahat^{\rm ac})$ purple dots and error bars and $\hat{\psi}_{2, k} (\hat{g}) = \hat{\psi}_{1} + \widehat{\mathbb{IF}}_{2, 2, k} (\hat{g})$ green dots and error bars). In particular, $\widehat{\mathbb{IF}}_{2, 2, k} (\hat{g})$ corrects little bias, as the distribution of $\hat{\psi}_{2, k} (\hat{g}) - \psi (\theta)$ is very close to the distribution of $\hat{\psi}_{1} - \psi (\theta)$ across different estimation and training sample sizes, regardless whether the nuisance parameters are estimated by GLM (lower-left) or GAM (lower-right).

\ \ \ \ From the lower-left panel of Figure~\ref{fig:1}, we observe that after corrected by $\widehat{\mathbb{IF}}_{2, 2, k} (\Omega)$ (grey dots and error bars) or $\widehat{\mathbb{IF}}_{2, 2, k} (\omegahat^{\rm emp})$ (blue dots and error bars), the biases of estimating $\psi$ are much closer to zero than using $\hat{\psi}_{1}$ without any bias correction (black dots and error bars) or even using $\hat{\psi}_{2, k}^{\rm ac}$ with bias corrected by $\widehat{\mathbb{IF}}_{2, 2, k} (\omegahat^{\rm ac})$ (purple dots and error bars). From the lower-right panel of Figure~\ref{fig:1}, it is hard to distinguish the bias of $\hat{\psi}_{2, k}^{\rm ac}$ from that of $\hat{\psi}_{2, k} (\Omega)$ or $\hat{\psi}_{2, k}^{\rm emp}$. Instead, we can examine the lower-right panel of Figure~\ref{fig:2} in which the results of $\hat{\psi}_{2, k} (\hat{g})$ are discarded. Now we are able to observe that $\hat{\psi}_{2, k}^{\rm emp}$ is still closer to $\hat{\psi}_{2, k} (\Omega)$ than $\hat{\psi}_{2, k}^{\rm ac}$, in particular as the training sample size increases. As displayed in Table~\ref{tab:2}, $\hat{\psi}_{2, k}^{\rm ac}$ has smaller bias (for estimating $\psi (\theta)$) than $\hat{\psi}_{2, k}^{\rm emp}$ or even the oracle estimator $\hat{\psi}_{2, k} (\Omega)$. However, one should compare $\hat{\psi}_{2, k}^{\rm ac}$ and $\hat{\psi}_{2, k}^{\rm emp}$ to the oracle $\hat{\psi}_{2, k} (\Omega)$ instead, because there is no theoretical reason for $\hat{\psi}_{2, k}^{\rm ac}$ to have smaller bias than $\hat{\psi}_{2, k} (\Omega)$. Therefore the bias of $\hat{\psi}_{2, k}^{\rm ac}$ will likely be greater than that of $\hat{\psi}_{2, k} (\Omega)$ for other simulation parameters.

\ \ \ \ All the above observations made from the lower panels of Figures~\ref{fig:1} and~\ref{fig:2} can also be made from Tables~\ref{tab:2} and~\ref{tab:4}, in which we display the Monte Carlo averages and standard deviations of different versions of $\hat{\psi}_{2, k}$ across different training and estimation sample sizes. For example, we observe that the Monte Carlo averages of $\hat{\psi}_{2, k}^{\rm emp} - \psi (\theta)$ are generally closer to those of $\hat{\psi}_{2, k} (\Omega) - \psi (\theta)$ than those of $\hat{\psi}_{2, k}^{\rm ac} - \psi (\theta)$; and the Monte Carlo standard deviations of $\hat{\psi}_{2, k}^{\rm emp} - \psi (\theta)$ are also smaller than those of $\hat{\psi}_{2, k}^{\rm ac} - \psi (\theta)$.

\ \ \ \ Similarly, the observations made from Figure~\ref{fig:3} can be read from Tables~\ref{tab:5} and~\ref{tab:6}: Obviously the distribution of $\widehat{\mathbb{IF}}_{2, 2, k} (\omegahat^{\rm emp}) - \widehat{\mathbb{IF}}_{2, 2, k} (\Omega)$ is much closer to 0 than that of $\widehat{\mathbb{IF}}_{2, 2, k} (\omegahat^{\rm ac}) - \widehat{\mathbb{IF}}_{2, 2, k} (\Omega)$.

\end{itemize}
In summary, in the above simulation in which $g$ is very rough, $\widehat{\mathbb{IF}}_{2, 2, k} (\omegahat^{\rm emp})$ (and also $\hat{\psi}_{2, k}^{\rm emp}$) has better finite sample statistical performance and is relatively more efficient to compute than $\widehat{\mathbb{IF}}_{2, 2, k} (\omegahat^{\rm ac})$ and $\widehat{\mathbb{IF}}_{2, 2, k} (\hat{g})$. Finally, it is worth emphasizing that the main goal of the simulation study here is quite modest--we simply would like to show that $\hat{\psi}_{2, k}^{\rm emp}$ has smaller bias than $\hat{\psi}_{2, k}^{\rm ac}$ or $\hat{\psi}_{2, k} (\hat{g})$ as an estimator of $\psi (\theta) + \mathrm{TB}_{k} (\theta) = \mathbb{E}_{\theta} [\hat{\psi}_{2, k} (\Omega)]$. It is possible that the bias of $\hat{\psi}_{2, k}^{\rm emp}$ as an estimator of $\psi (\theta)$ may be larger than $\hat{\psi}_{2, k}^{\rm ac}$ or $\hat{\psi}_{2, k} (\hat{g})$, depending on how their biases as estimators of $\psi (\theta) + \mathrm{TB}_{k} (\theta)$ interact with the truncation bias $\mathrm{TB}_{k} (\theta)$ in finite sample. In practice, one would not have the knowledge about the truncation bias. Therefore, our goal has to be to construct good estimators of what is estimable, which is $\mathbb{E}_{\theta} [\hat{\psi}_{2, k} (\Omega)]$, the mean of the oracle estimator. Simulation results presented in Table~\ref{tab:4} exemplify such a scenario. Also, to implement the empirical HOIF estimator proposed in this paper in real applications, several open questions need to be addressed, particularly how to select basis functions and $k$ in a data-driven manner. As suggested by a referee, in Appendix~\ref{app:practice}, we also explore the empirical performance of empirical HOIF estimators when both $k$ and basis functions are altered.


\section{Literature overview and discussions}

\label{section_end}

\subsection{Literature overview}
\label{section_literature} 

We now provide a literature overview on $\sqrt{n}$-consistent estimation under relaxed conditions for the class of functionals considered in this paper. For instance, the use of HOIF for bias correction has been considered in a series of papers \citep{tchetgen2008minimax, robins2008higher, diaz2016second, carone2018higher, robins2017minimax, mukherjee2016adaptive, van2021higher, yu2024treatment} but they all require either knowing the density of the covariates or estimating the density at a sufficiently fast rate. To the best of our knowledge, \citet{newey2018cross} is the first paper demonstrating the existence of $\sqrt{n}$-consistent estimators under minimal \Holder{} smoothness conditions of \cite{robins2009semiparametric} for a subclass of the functionals considered in our paper, without using the HOIF machinery. Similar bias correction techniques can also be found in the econometric literature \citep{cattaneo2018inference, cattaneo2019two}. However, that subclass does not include the mean of a response $Y$ under MAR or the average treatment effect under ignorable treatment assignment, except for the corner cases in which $\beta_{b}$ is greater than $\beta_{p}$. \citet{hirshberg2021augmented}, \cite{armstrong2021finite}, and \cite{kennedy2023towards} also obtained $\sqrt{n}$-consistent estimators for the ATE, but only for the aforementioned corner cases. Therefore the empirical HOIF estimator of diverging order proposed here remains the only known $\sqrt{n}$-consistent estimator under the minimal \Holder{} smoothness conditions of \cite{robins2009semiparametric} alone.

\cite{liu2024assumption} demonstrated another application of empirical HOIF estimators, which were used to construct a test of the hypothesis that $\hat{\psi}_{1}$ based on the first-order influence function is of a smaller order than its standard error \citep{liu2020nearly}. When the test rejects, one can conclude that a nominal $(1 - \alpha) \times 100 \%$ large-sample Wald confidence interval centered around $\hat{\psi}_{1}$ has an actual coverage less than $(1 - \alpha) \times 100 \%$. More recently, a variant of the empirical HOIF estimator studied here has also been applied to construct covariate adjusted estimator in randomized experiments with guaranteed efficiency gain with covariate dimension diverging at any rate slower than $n$ \citep{zhao2024hoif, gu2025assumption}.

\subsection{Discussions}
\label{section_discussions} 
We have shown that for $\sqrt{n}$-estimable parameters the asymptotic properties of our new empirical HOIF estimators are identical to those of the HOIF estimators of \cite{robins2008higher, robins2017minimax}, yet eliminate the need to construct multivariate density estimates and the extra smoothness assumptions on that density. We end our paper by pointing out several research directions:

\begin{itemize}
\item It is interesting to generalize the theory of HOIFs developed in \cite{robins2008higher, robins2017minimax} and of empirical HOIFs developed in our paper to other causal parameters or more complicated scenario, such as those in \cite{tchetgen2012semiparametric, ai2021unified, bhattacharya2022semiparametric, cui2024semiparametric, breunig2024adaptive}. \citet{liu2024assumption} derive the HOIFs for the mean of a response $Y$ under MNAR under the so-called proximal causal inference framework \citep{cui2024semiparametric}. \citet{kennedy2024minimax} used second-order influence functions to estimate the Conditional Average Treatment Effect (CATE) as a function of the covariates $X$ under ignorability. However their estimator did not achieve the minimax rate even when the CATE function was very smooth except when $g$ was also very smooth, because the order of the HOIF $U$-statistic estimator was not allowed to increase with sample size. 

\item Another interesting and important open problem is to investigate if it is possible to estimate the mean of a response $Y$ under MAR or equivalently the average treatment effect under the minimal \Holder{} smoothness conditions of \cite{robins2009semiparametric} without using U-statistics of diverging orders, e.g. using procedures similar to those in \cite{newey2018cross} and \cite{kennedy2023towards}. We expect this to be not possible but we do not have a proof.

\item In this paper, since we have focused on the case where the nuisance parameters are \Holder{} smooth functions, both the size and the type of basis functions/dictionary $\bar{z}_{k}$ can be determined theoretically (at least in the asymptotic sense). Developing data-driven basis selection methods is an important missing piece for HOIF estimators to be routinely employed in practice. We expect that a further sample splitting is needed to avoid model selection bias, but the criterion used for basis selection remains a difficult open problem. Finally, as mentioned at the end of Section~\ref{section_simulations}, we provide some further practical guidance on our proposed estimator in Appendix~\ref{app:practice}. The guidance is by no means definitive. We plan to report a more comprehensive empirical study in a future paper.
\end{itemize}

\begin{figure}[htbp]
\includegraphics[width=0.8\textwidth]{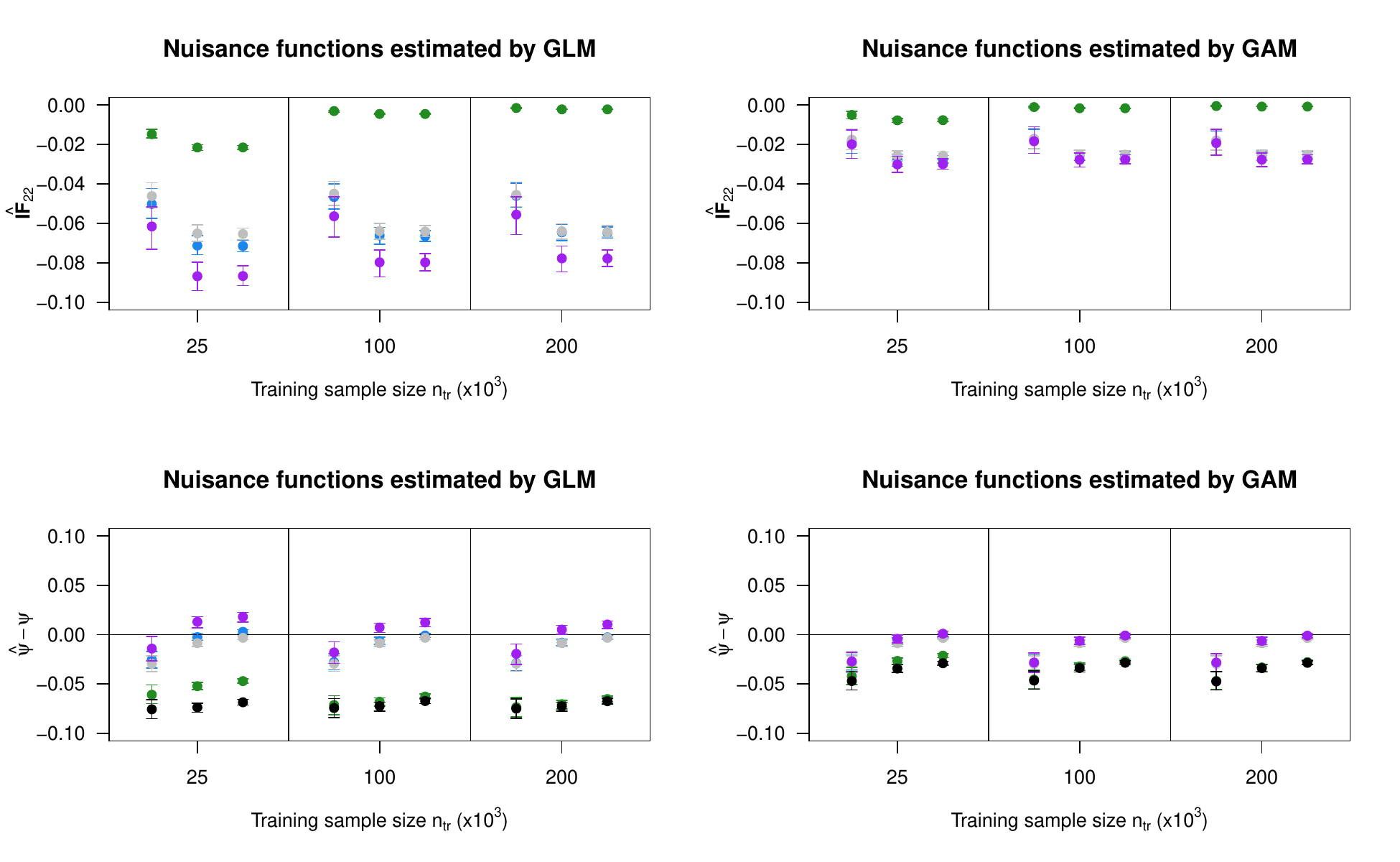} \centering
\caption{Results of simulation experiment. The upper panels compare $\widehat{\mathbb{IF}}_{2, 2, k} (\Omega)$, $\widehat{\mathbb{IF}}_{2, 2, k} (\omegahat^{\rm emp})$, $\widehat{\mathbb{IF}}_{2, 2, k} (\omegahat^{\rm ac})$ and $\widehat{\mathbb{IF}}_{2, 2, k} (\hat{g})$. Color code: black--$\hat{\psi}_{1} - \psi (\theta)$; grey--$\widehat{\mathbb{IF}}_{2, 2, k}(\Omega)$; blue--$\widehat{\mathbb{IF}}_{2, 2, k}(\omegahat^{\rm emp})$; purple--$\widehat{\mathbb{IF}}_{2, 2, k} (\omegahat^{\rm ac})$; green--$\widehat{\mathbb{IF}}_{2, 2, k} (\hat{g})$. The lower panels compare the estimators before and after being corrected by different versions of $\widehat{\mathbb{IF}}_{2, 2, k}$, i.e. $\hat{\psi}_{2, k}(\Omega) = \hat{\psi}_{1} + \widehat{\mathbb{IF}}_{2, 2, k} (\Omega)$, $\hat{\psi}_{2, k}^{\rm emp} = \hat{\psi}_{1} + \widehat{\mathbb{IF}}_{2, 2, k} (\omegahat^{\rm emp})$, $\hat{\psi}_{2, k}^{\rm ac} = \hat{\psi}_{1} + \widehat{\mathbb{IF}}_{2, 2, k} (\omegahat^{\rm ac})$, and $\hat{\psi}_{2, k}(\hat{g}) = \hat{\psi}_{1} + \widehat{\mathbb{IF}}_{2, 2, k} (\hat{g})$. Color code: black--$\hat{\psi}_{1} - \psi (\theta)$; grey--$\hat{\psi}_{2, k}(\Omega) - \psi (\theta)$; blue--$\hat{\psi}_{2, k}^{\rm emp} - \psi (\theta)$; purple--$\hat{\psi}_{2, k}^{\rm ac} - \psi (\theta)$; green--$\hat{\psi}_{2, k}(\hat{g}) - \psi (\theta)$. In the panels on the left, the nuisance functions $b$ and $1 / p$ are estimated by GLMs whereas in panels on the right, they are estimated by GAMs. The dots in each plot are the Monte Carlo averages across 100 simulated datasets. The error bars in each plot correspond to the 10\% and 90\% percentiles out of 100 Monte Carlo simulations. Within each column of any panel, from left to right we display the simulation results for estimation sample sizes $n = 25000, 100000, 200000$.}
\label{fig:1}
\end{figure}

\begin{figure}[htbp]
\includegraphics[width=0.8\textwidth]{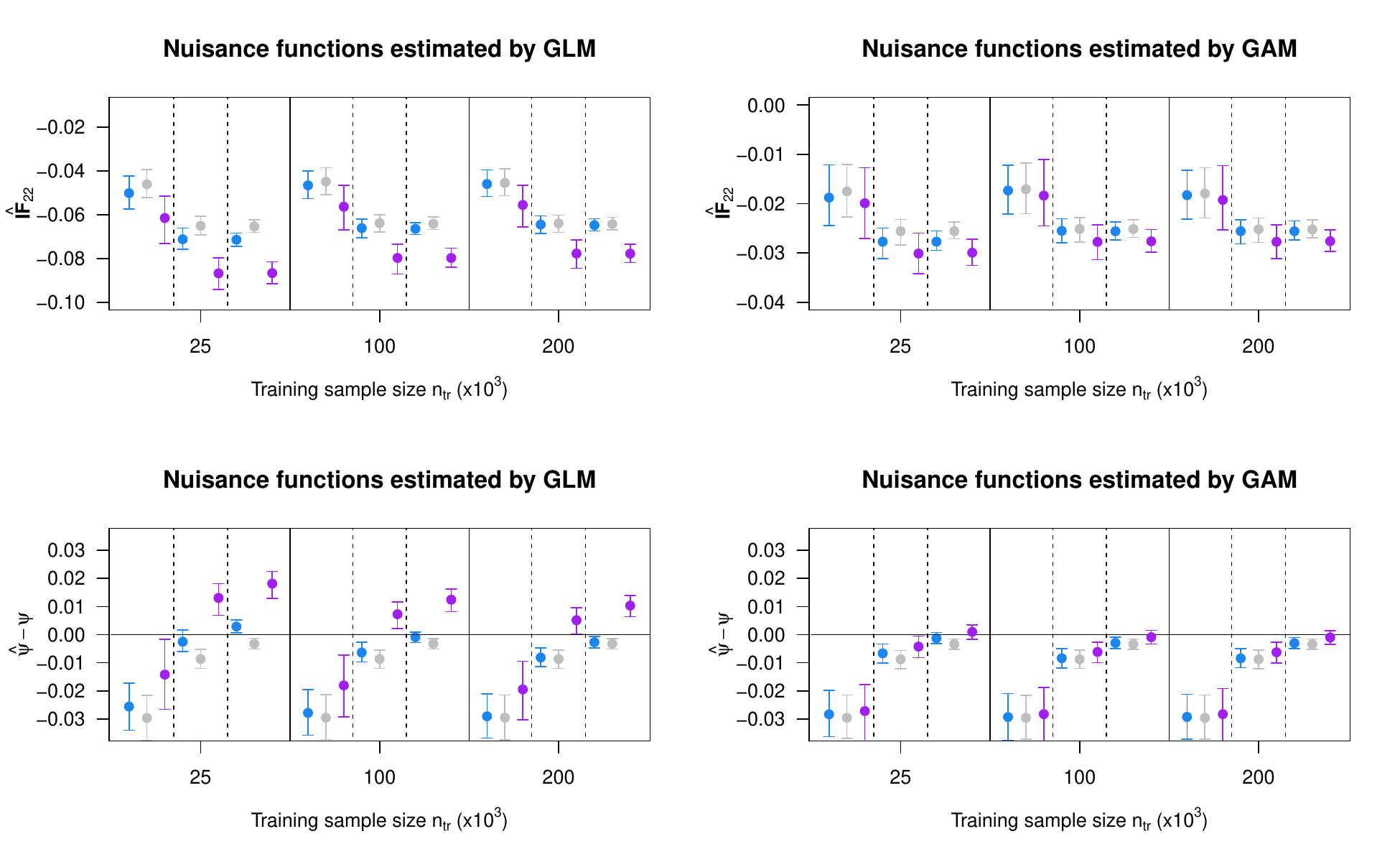} \centering
\caption{Results of simulation experiment. The color codes are the same as in Figure~\ref{fig:1}, except that the simulations for $\widehat{\mathbb{IF}}_{2, 2, k} (\hat{g})$ are removed from the upper panels and the simulations for $\hat{\psi}_{1} - \psi (\theta)$ and $\hat{\psi}_{2, k}(\hat{g}) - \psi (\theta)$ are removed from the lower panels. Within each column of any panel, from left to right we display the simulation results for estimation sample sizes $n = 25000, 100000, 200000$.}
\label{fig:2}
\end{figure}

\begin{figure}[htbp]
\includegraphics[width=0.8\textwidth]{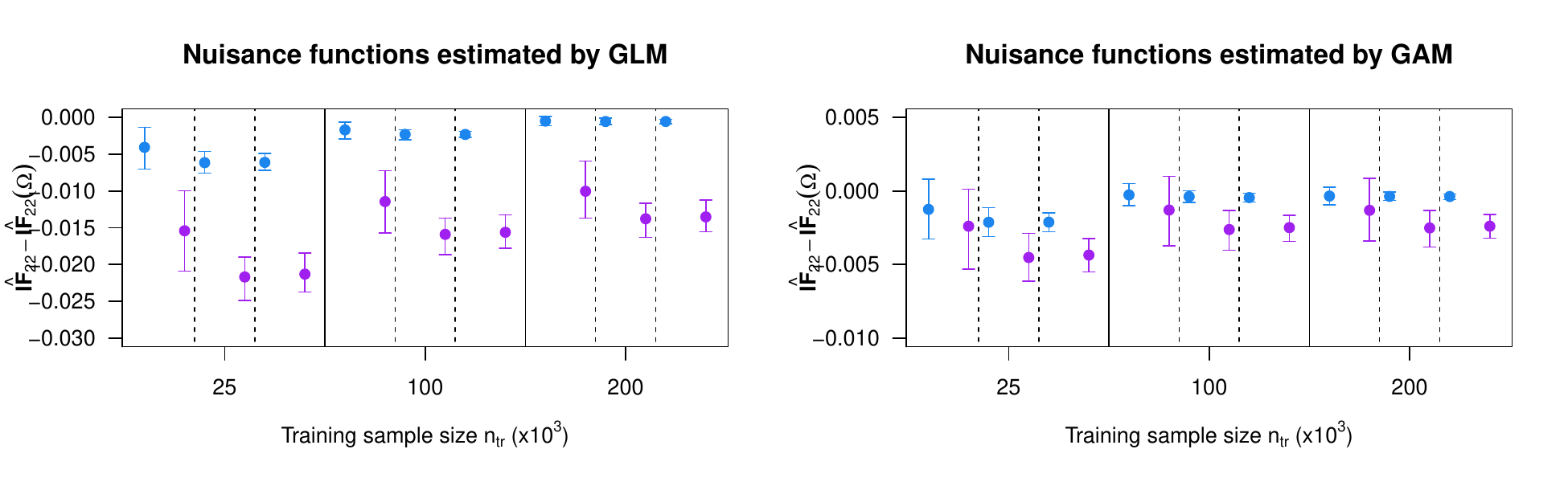} \centering
\caption{Results of simulation experiment. The color codes are the same as in Figure~\ref{fig:1}, except that only $\widehat{\mathbb{IF}}_{2, 2, k} (\omegahat^{\rm emp}) - \widehat{\IIFF}_{2, 2, k} (\Omega)$ and $\widehat{\mathbb{IF}}_{2, 2, k} (\omegahat^{\rm ac}) - \widehat{\IIFF}_{2, 2, k} (\Omega)$ are displayed to highlight the observation that $\widehat{\mathbb{IF}}_{2, 2, k} (\omegahat^{\rm emp})$ is closer to the oracle $\widehat{\IIFF}_{2, 2, k} (\Omega)$ than $\widehat{\mathbb{IF}}_{2, 2, k} (\omegahat^{\rm ac})$. Within each column of any panel, from left to right we display the simulation results for estimation sample sizes $n = 25000, 100000, 200000$.}
\label{fig:3}
\end{figure}

\section*{Acknowledgment}
We thank the anonymous associate editor and referees for constructive comments, Yulin Zhang (Renmin University of China), \href{https://sites.google.com/view/zheng-zhang}{Zheng Zhang} and \href{https://scholar.google.com/citations?user=Ys5ZVhEAAAAJ&hl=en&oi=sra}{Ling Guo} for helpful discussions, and \href{https://cxy0714.github.io/}{Xingyu Chen}, Sihui Zhao and Lei Ma for help on simulations. The computations in this paper were run on the Siyuan-1 and $\pi$-2.0 cluster supported by the Center for High Performance Computing at Shanghai Jiao Tong University. Lin Liu was supported by NSFC Grant No.12471274. Rajarshi Mukherjee was partially supported by NSF CAREER 8529216-01. James M. Robins was supported by the U.S. Office of Naval Research grant N000141912446, and National Institutes of Health (NIH) awards R01 AG057869 and R01 AI127271.

\putbib[free_lunch_biblio]
\end{bibunit}

\allowdisplaybreaks

\newpage

\appendix

\begin{bibunit}[imsart-nameyear]

\section{Derivation in the Introduction}
\label{app:intro}

\subsection{Derivation of the conditional bias of the first order doubly robust estimator}
\label{app:cbias}

In this section, we provide the details of the derivation of $\mathsf{cBias}_{\theta} (\hat{\psi}_{1})$ displayed in Section~\ref{section_missing_data_obs}. Note that the true parameter $\psi (\theta) = \mathbb{E}_{\theta} [b (X)] = \mathbb{E}_{\theta} [B]$, $\hat{P} = 1 / \hat{\Pi}$ and $P = 1 / \Pi$ in our notation. Then:
\begin{align*}
\mathsf{cBias}_{\theta} (\hat{\psi}_{1}) & = \mathbb{E}_{\theta} \left[ \frac{A}{\hat{\Pi}} (Y - \hat{B}) + \hat{B} - B \right] \\
& = \mathbb{E}_{\theta} \left[ \frac{A}{\hat{\Pi}} (B - \hat{B}) - (B - \hat{B}) \right] \\
& = \mathbb{E}_{\theta} \left[ \left( \frac{A}{\hat{\Pi}} - \frac{A}{\Pi} \right) (B - \hat{B}) \right] \\
& = \mathbb{E}_{\theta} \left[ A (\hat{P} - P) (B - \hat{B}) \right] \\
& = \int (\hat{p} (x) - p (x)) (b (x) - \hat{b} (x)) \pi (x) f (x) \diff x \\
& = \int (\hat{p} (x) - p (x)) (b (x) - \hat{b} (x)) g (x) \diff x,
\end{align*}
where the second equality follows from the definition of $b (x) = \mathbb{E}_{\theta} (Y | X = x, A = 1)$, the third and fifth equalities are results of using the law of iterated expectations, and the last equality applies the identity $g (\cdot) = \pi (\cdot) f (\cdot)$ because of how $g$ is defined.

\subsection{Derivation of the higher-order terms}
\label{app:higher-order}

In this section, we first heuristically explain why $\hat{\IIFF}_{3, 3, k} (\hat{\Omega})$ for any generic estimator $\hat{\Omega}$ of $\Omega$ is the statistic used to reduce the estimation bias $\EB_{2, k} (\theta)$, when $\hat{\psi}_{2, k} (\hat{\Omega})$ is used to estimate the oracle estimator $\hat{\psi}_{2, k} (\Omega)$. To this end, we explicitly write down the form of $\EB_{2, k} (\theta)$:
\begin{align*}
\EB_{2, k} (\theta) & = \Etheta [\hat{\psi}_{2, k} (\hat{\Omega}) - \hat{\psi}_{2, k} (\Omega)] \\
& = \Etheta [\hat{\mathbb{IF}}_{2, 2, k} (\hat{\Omega}) - \hat{\mathbb{IF}}_{2, 2, k} (\Omega)] \\
& = \Etheta [(1 - A \hat{P}) \bar{Z}_{k}^{\top}] (\hat{\Omega}^{-1} - \Omega^{-1}) \Etheta [\bar{Z}_{k} A (Y - \hat{B})] \\
& = \Etheta [(1 - A \hat{P}) \bar{Z}_{k}^{\top}] \Omega^{-1} (\Omega - \hat{\Omega}) \hat{\Omega}^{-1} \Etheta [\bar{Z}_{k} A (Y - \hat{B})].
\end{align*}
It is not difficult to see that the error due to estimating $\Omega$ by $\hat{\Omega}$ is of first-order in $\EB_{2, k} (\theta)$.

$\hat{\IIFF}_{3, 3, k} (\hat{\Omega})$ is a natural estimator of $- \EB_{2, k} (\theta)$ because:
\begin{align*}
\Etheta [\hat{\IIFF}_{3, 3, k} (\hat{\Omega})] & = - \Etheta \Big[ [(1 - A \hat{P}) \bar{Z}_{k}]_{1}^{\top} \hat{\Omega}^{-1} ([A \bar{Z}_{k} \bar{Z}_{k}^{\top}]_{3} - \hat{\Omega}) \hat{\Omega}^{-1} [\bar{Z}_{k} A (Y - \hat{B})]_{2} \Big] \\
& = - \Etheta [(1 - A \hat{P}) \bar{Z}_{k}^{\top}] \hat{\Omega}^{-1} (\Omega - \hat{\Omega}) \hat{\Omega}^{-1} \Etheta [\bar{Z}_{k} A (Y - \hat{B})].
\end{align*}
Recall that the third-order influence function estimator takes the form $\hat{\psi}_{3, k} (\hat{\Omega}) = \hat{\psi}_{2, k} (\hat{\Omega})$, which leads to the following estimation bias
\begin{align*}
\EB_{3, k} (\theta) & = \Etheta [\hat{\psi}_{3, k} (\hat{\Omega}) - \hat{\psi}_{2, k} (\Omega)] \\
& = \Etheta [\hat{\IIFF}_{3, 3, k} (\hat{\Omega})] + \EB_{2, k} (\theta) \\
& = - \Etheta [(1 - A \hat{P}) \bar{Z}_{k}^{\top}] \hat{\Omega}^{-1} (\Omega - \hat{\Omega}) \hat{\Omega}^{-1} \Etheta [\bar{Z}_{k} A (Y - \hat{B})] + \EB_{2, k} (\theta) \\
& = - \Etheta [(1 - A \hat{P}) \bar{Z}_{k}^{\top}] \Big\{ \hat{\Omega}^{-1} (\Omega - \hat{\Omega}) \hat{\Omega}^{-1} - \Omega^{-1} (\Omega - \hat{\Omega}) \hat{\Omega}^{-1} \Big\} \Etheta [\bar{Z}_{k} A (Y - \hat{B})] \\
& = - \Etheta [(1 - A \hat{P}) \bar{Z}_{k}^{\top}] \Big\{ (\hat{\Omega}^{-1} - \Omega^{-1}) (\Omega - \hat{\Omega}) \hat{\Omega}^{-1} \Big\} \Etheta [\bar{Z}_{k} A (Y - \hat{B})].
\end{align*}
Therefore, the bias due to estimating $\Omega$ by $\hat{\Omega}$ reduces to second-order. Similar reasoning applies to explain why a general $m$-th order influence function estimator $\hat{\psi}_{m, k} (\hat{\Omega})$ takes the form given in Section~\ref{section_missing_data_estimator}. More rigorous derivation can also be found in Appendix~\ref{app:bias bound} later. We also refer interested reader to Theorem 3.17 in \citet{robins2008higher} for similar derivations.

\section{Main Proofs}
\label{section_proofs} 
\begin{proof}[Proof of Proposition~\ref{theorem_bias_variance_ghat} and Theorem~\ref{theorem_bias_variance_emp}]
We divide our proof into bias and variance computations respectively. Throughout we assume $I(h_1 (W)\leq 0)=1$ almost surely. The case $I(h_1 (W)\geq 0$ requires obvious sign changes in various places. We first give the proof for a generic HOIF estimator $\hat{\psi}_{m, k}$ and then specialize to the empirical HOIF estimator $\hat{\psi}_{m, k}^{\rm emp}$.

\subsection{Bias bound}
\label{app:bias bound}

By the same analysis as in \cite{robins2008higher},
\[
\EB_{m, k} \left( \theta \right) = (-1)^{m} \E_{\theta} [H_{1} (P - \hat{P}) \bar{Z}_{k}^{\top}] \Omega^{-1} \left[ \left\{ \Omega - \hat{\Omega} \right\} \hat{\Omega}^{-1} \right]^{m - 1} \E_{\theta} [\bar{Z}_{k} H_{1} (B - \hat{B})].
\]
We next show that under the assumptions of Theorem~\ref{theorem_bias_variance_emp}
\[
\left\vert \EB_{m, k} \left( \theta \right) \right\vert = O \Big( \|\omegahat - \Omega\|_{\rm op}^{m - 1} \left\{ \E_{\theta} [(B - \hat{B})^{2}] \E_{\theta} [(P - \hat{P})^{2}] \right\}^{1 / 2} \Big),
\]
where $\| \cdot \|_{\rm op}$ denotes the operator norm of a matrix.

Now let $\hat{1}$ denote the indicator function for the event that $\lambda_{\max} (\hat{\Omega}^{-1}) \leq C$ for some $C > 0$. In the rest of the proof we take this $C$ sufficiently large but still bounded above such that $\lambda_{\max} ({\Omega}^{-1}) \leq C$ as well (note that this is allowed by Condition~\ref{def_conditions} assumed in the statement of the theorem).

By Cauchy-Schwarz inequality,
\[
\left\vert \EB_{m, k} (\theta) \right\vert \leq \Big\Vert \E_{\theta} [H_1 (P - \hat{P}) \bar{Z}_{k}^{\top}] \Omega^{-1/2} \Big\Vert \cdot \Big\Vert \Omega^{-1/2} [\{\Omega - \hat{\Omega}\} \hat{\Omega}^{-1}]^{m - 1} \Etheta [\bar{Z}_{k} H_1 (B - \hat{B})] \Big\Vert.
\]
Note that $\Vert\E_{\theta} [H_1 (P - \hat{P}) \bar{Z}_{k}^{\top}] \Omega^{-1/2}\Vert^{2}$ is the second moment of the linear projection of $P - \hat{P}$ on $\bar{Z}_{k}$ under $g$, so that
\[
\Big\Vert \E_{\theta} [H_1 (P - \hat{P}) \bar{Z}_{k}^{\top}] \Omega^{-1/2} \Big\Vert \leq \Big\{ \E_{\theta} [(P - \hat{P})^{2}] \Big\}^{1/2},
\]
by the norm contraction property of the linear projection. Denoting $\Sigma \coloneqq \Omega^{-1/2} [\{\Omega - \hat{\Omega}\} \hat{\Omega}^{-1}]^{m - 1}$, we then have, in the positive semi-definite sense,
\begin{align*}
\hat{1} \Sigma^{\top} \Sigma  & = \hat{1} \left[ \hat{\Omega}^{-1} \left\{ \Omega - \hat{\Omega} \right\} \right]^{m - 1} \Omega^{-1} \left[ \left\{ \Omega - \hat{\Omega} \right\} \hat{\Omega}^{-1} \right] ^{m - 1} \\
& \leq \hat{1} C \left[ \hat{\Omega}^{-1} \left\{ \Omega - \hat{\Omega} \right\} \right]^{m - 1} \left[ \left\{ \Omega - \hat{\Omega} \right\} \hat{\Omega}^{-1} \right]^{m - 1}\\
& = \hat{1} C \left[ \hat{\Omega}^{-1} \left\{ \Omega - \hat{\Omega} \right\} \right]^{m - 2} \hat{\Omega}^{-1} \left\{ \Omega - \hat{\Omega}\right\}^{2} \hat{\Omega}^{-1} \left[ \left\{ \Omega - \hat{\Omega} \right\} \hat{\Omega}^{-1} \right]^{m - 2} \\
& \leq \left\Vert \Omega - \hat{\Omega} \right\Vert_{\rm op}^{2} \hat{1} C \left[ \hat{\Omega}^{-1} \left\{ \Omega - \hat{\Omega} \right\} \right]^{m - 2} \hat{\Omega}^{-2} \left[ \left\{ \Omega - \hat{\Omega} \right\} \hat{\Omega}^{-1} \right]^{m - 2} \\
& \leq \left\Vert \Omega - \hat{\Omega} \right\Vert_{\rm op}^{2} \hat{1} C^{3} \left[ \hat{\Omega}^{-1} \left\{ \Omega - \hat{\Omega} \right\} \right]  ^{m - 2} \left[ \left\{ \Omega - \hat{\Omega} \right\} \hat{\Omega}^{-1} \right]^{m - 2}.
\end{align*}
Repeating this argument (i.e. by induction) we have
\[
\hat{1} \Sigma^{\top} \Sigma \leq \hat{1} \Vert \Omega - \hat{\Omega} \Vert_{\rm op}^{2 (m - 1)} C^{2 (m - 1) + 1} I.
\]
Next, since $I\leq\Omega^{-1} C^{-1}$ in the p.s.d. sense we have
\[
\hat{1} \Sigma^{\top} \Sigma \leq \hat{1} \Vert \Omega - \hat{\Omega} \Vert_{\rm op}^{2 (m - 1)} C^{2 (m - 1)} \Omega^{-1}.
\]
It then follows that
\begin{align*}
& \hat{1}\left\Vert \Omega^{-1/2} \left[ \left\{ \Omega - \hat{\Omega} \right\} \hat{\Omega}^{-1} \right]^{m - 1} \E_{\theta} [\bar{Z}_{k} H_1 (B - \hat{B})] \right\Vert^{2} \\
& = \hat{1} \Vert \Sigma \E_{\theta} [\bar{Z}_{k} H_1 (B - \hat{B})] \Vert^{2} = \hat{1} \E_{\theta} [H_1 (B - \hat{B}) \bar{Z}_{k}^{\top}] \Sigma^{\top} \Sigma \E_{\theta} [\bar{Z}_{k} H_1 (B - \hat{B})] \\
& \leq \hat{1} \Vert \Omega - \hat{\Omega} \Vert_{\rm op}^{2 (m - 1)} C^{2 (m - 1)} \E_{\theta} [H_1 (B - \hat{B}) \bar{Z}_{k}^{\top}] \Omega^{-1} \E_{\theta} [\bar{Z}_{k} H_1 (B - \hat{B})]. \\
& \leq \hat{1} \Vert \Omega - \hat{\Omega} \Vert_{\rm op}^{2 (m - 1)} C^{2 (m - 1)} \E_{\theta} [(B - \hat{B})^{2}],
\end{align*}
where the last inequality follows by $\E_{\theta} [H_1 (B - \hat{B}) \bar{Z}_{k}^{\top}] \Omega^{-1} \E_{\theta} [\bar{Z}_{k} H_1 (B - \hat{B})]$ being the expected square of the projection of $B - \hat{B}$ on $\bar{Z}_{k}$ under $g$. Therefore we have
\begin{align*}
\hat{1} \left\vert \EB_{m, k} \right\vert & \leq \hat{1} \Vert \Omega - \hat{\Omega} \Vert_{\rm op}^{m - 1} C^{m - 1} \left\{ \E_{\theta} [(B - \hat{B})^{2}] \E_{\theta} [(P - \hat{P})^{2}] \right\}^{1 / 2}.
\end{align*}
This completes the upper bound for the bias.

\subsection{Variance bound} \label{app:var_bound_1}
The strategy for the variance bound proof applies to both the empirical HOIF estimators and the HOIF estimators based on density estimation. In this section, we will first prove the variance bound for generic $\hat{\Omega}$, and then state the results for $\omegahat^{\rm emp}$ as direct consequences.

In the proof, the constant $C$, independent of the sample size, will change from line to line. In this section we are agnostic to the specific forms of $\epsb$ and $\epsp$ except that they are bounded with $\mathbb{P}_{\theta}$-probability 1.

For convenience, we introduce the $j$-th order U-statistic operator $\mathbb{U}_n (h(O_{\overline{i}_{j}}))$, for any nonnegative integer $j$ and any function $h: \mathbb{R}^{j} \rightarrow \mathbb{R}$:
\begin{align*}
\mathbb{U}_{n} (h(O_{\overline{i}_{j}})) \coloneqq \frac{(n - j)!}{n!} \sum_{\bar{i}_{j} \in I_{n, j}} h(O_{\overline{i}_{j}}).
\end{align*}

To control the variance of $\hat{\psi}_{m,k}$ we begin with the following variance bound of $\widehat{\mathbb{IF}}_{22} = \mathbb{U}_n(\IF_{2,2,k,\overline{i}_{2}})$. As expected, the proof makes use of Hoeffding decomposition.
\begin{lemma}\label{lem:var_if22}
Under the conditions of Proposition~\ref{theorem_bias_variance_ghat}, there exists a positive constant $C$, depending only on $(\bar{z}_{k}, p, \hat{p}, b, \hat{b}, g)$ such that
\begin{equation} \label{eqn:var_if22}
\begin{split}
\var_{\theta}\left(\mathbb{U}_n (\IF_{2,2,k,\overline{i}_{2}}) \right) & \leq \frac{C}{n} \left\{ \frac{k}{n} + \BL_{2, \hat{b}, k}^{2} + \BL_{2, \hat{p}, k}^{2} \right\}.
\end{split}
\end{equation}
\end{lemma}

\begin{proof}
By Hoeffding decomposition, 
\begin{align*}
& \ - \left\{ \mathbb{U}_n (\IF_{2,2,k,\overline{i}_{2}}) - \E_{\theta} [\mathbb{U}_n (\IF_{2,2,k,\overline{i}_{2}})] \right\} \\
= & \ \underbrace{\frac{1}{n} \sum_{i = 1}^{n} \E_{\theta} \left[ \epsb \bar{z}_{k} (X) \right]^{\top} \widehat{\Omega}^{-1} \left\{ \bar{z}_{k} (X_{i}) \varepsilon_{\hat{p}, i} - \E_{\theta} \left[ \epsp \bar{z}_{k} (X) \right] \right\}}_{T_{11}} \\
& + \underbrace{\frac{1}{n} \sum_{i = 1}^{n} \E_{\theta} \left[ \epsp \bar{z}_{k} (X) \right]^{\top} \widehat{\Omega}^{-1} \left\{ \bar{z}_{k} (X_{i}) \varepsilon_{\hat{b}, i} - \E_{\theta} \left[ \epsb \bar{z}_{k} (X) \right] \right\}}_{T_{12}} \\
& + \underbrace{\frac{1}{n (n - 1)} \sum_{\bar{i}_{2} \in I_{n, 2}} \left\{ \varepsilon_{\hat{b}, i_{1}} \bar{z}_{k} (X_{i_{1}}) - \E_{\theta} \left[ \epsb \bar{z}_{k} (X) \right] \right\}^{\top} \widehat{\Omega}^{-1} \left\{ \bar{z}_{k} (X_{i_{2}}) \varepsilon_{\hat{p}, i_{2}} - \E_{\theta} \left[ \bar{z}_{k} (X) \epsp \right] \right\}}_{T_{2}}.
\end{align*}

Define $\tilde\Omega \coloneqq \int \bar{z}_{k} (x) \bar{z}_{k} (x) f (x) \diff x$. Then under the assumptions (Condition~\ref{def_conditions}(2) and boundedness of $f$ and $\vert H_{1} \vert$) in our paper, there exists a universal constant $B' > 0$ such that $\frac{1}{B'} \leq \lambda_{\min}(\tilde\Omega) \leq \lambda_{\max}(\tilde\Omega) \leq B'$.

For the linear term $T_{11}$, we have
\begin{align*}
\var_{\theta} [T_{11}] & \leq \frac{1}{n} \Vert \epsp^{2} \Vert_{\infty} \E_{\theta} \left[ \epsb \bar{z}_{k} (X) \right]^{\top} \omegahat^{-1} \Etheta \left[ \bar{z}_{k} (X) \bar{z}_{k} (X)^{\top} \right] \omegahat^{-1} \E_{\theta} \left[ \bar{z}_{k} (X) \epsb \right] \\
& \leq \frac{1}{n} \Vert \epsp^{2} \Vert_{\infty} \E_{\theta} \left[ \epsb \bar{z}_{k} (X) \right]^{\top} \Omega^{- 1 / 2} \left( \Omega^{1 / 2} \omegahat^{-1} \tilde{\Omega}^{1 / 2} \right)^{2} \Omega^{- 1 / 2} \E_{\theta} \left[ \bar{z}_{k} (X) \epsb \right] \\
& \leq \frac{C}{n} \Vert \epsp^{2} \Vert_{\infty} \BL_{2, \hat{b}, k}^{2},
\end{align*}
where the last inequality follows from the definition of matrix operator norm. By symmetry,
\begin{align*}
\var_{\theta} [T_{12}] \leq \frac{C}{n} \Vert \epsb^{2} \Vert_{\infty} \BL_{2, \hat{p}, k}^{2}.
\end{align*}

For the second-order degenerate U-statistic term $T_{2}$, we have
\begin{align*}
\var_{\theta} [T_{2}] & \leq \frac{1}{n (n - 1)} \E_{\theta} \left[ \varepsilon_{\hat{b}, 1}^{2} \varepsilon_{\hat{p}, 2}^{2} \left( \bar{z}_{k} (X_{1})^{\top} \widehat{\Omega}^{-1} \bar{z}_{k} (X_{2}) \right)^{2} \right] \\
& \leq \frac{C}{n^{2}} \Vert \epsb^{2} \epsp^{2} \Vert_{\infty} \E_{\theta} \left[ \bar{z}_{k} (X)^{\top} \widehat{\Omega}^{-1} \tilde{\Omega} \widehat{\Omega}^{-1} \bar{z}_{k} (X) \right] \\
& \leq \frac{C}{n} \frac{k}{n} \Vert \epsb^{2} \epsp^{2} \Vert_{\infty}.
\end{align*}
Above the last inequality follows by Lemma~\ref{lemma_general_second_moment}.

Finally applying $\var_{\theta} [\mathbb{U}_n (\IF_{2,2,k,\overline{i}_{2}})] = \var_{\theta} [T_{11} + T_{12}] + \var_{\theta} [T_{2}] \leq 2 \var_{\theta} [T_{11}] + 2 \var_{\theta} [T_{12}] + \var_{\theta} [T_{2}]$, we obtain
$$
\var_{\theta} [\mathbb{U}_n (\IF_{2,2,k,\overline{i}_{2}})] \leq \frac{C}{n} \left\{ \BL_{2, \hat{b}, k}^{2} \Vert \epsp^{2} \Vert_{\infty} + \Vert \epsb^{2} \Vert_{\infty} \BL_{2, \hat{p}, k}^{2} + \frac{k}{n} \Vert \epsb^{2} \epsp^{2} \Vert_{\infty} \right\}.
$$
\end{proof}

Next we compute the variance bound of $\widehat{\mathbb{IF}}_{33} = \mathbb{U}_n(\IF_{3,3,k,\overline{i}_{3}})$. In particular, we have
\begin{lemma}\label{lem:var_if33}
Under the conditions of Proposition~\ref{theorem_bias_variance_ghat}, there exists a positive constant $C$, depending only on $(\bar{z}_{k}, p, \hat{p}, b, \hat{b}, g)$ such that
\begin{equation} \label{eqn:var_if33}
\begin{split}
\var_{\theta}\left(\mathbb{U}_n(\IF_{3,3,k,\overline{i}_{3}})\right) & \leq \frac{C}{n} \left\{ \begin{array}{c}
\left( \dfrac{k}{n} \right)^{2} + \dfrac{k}{n} \left( \BL_{2, \hat{b}, k}^{2} + \BL_{2, \hat{p}, k}^{2} + \BL_{2, \widehat{\Omega}, k}^{2} \right) \\
+ \left( \BL_{2, \hat{b}, k}^{2} \BL_{2, \widehat{\Omega}, k}^{2} + \BL_{2, \hat{p}, k}^{2} \BL_{2, \widehat{\Omega}, k}^{2} + \min\limits_{(\eta, \zeta): 1 / \eta + 1 / \zeta = 1} \BL_{2 \eta, \hat{b}, k}^{2} \BL_{2 \zeta, \hat{p}, k}^{2} \right)
\end{array} \right\}.
\end{split}
\end{equation}
\end{lemma}

The proof of Lemma~\ref{lem:var_if33} can be found in Appendix~\ref{proof:var_if33}. For general $j > 3$, we have the following result.

\begin{lemma} \label{lem:var_ifjj}
Under the conditions of Proposition~\ref{theorem_bias_variance_ghat}, up to a universal constant depending only on $(\bar{z}_{k}, p, \hat{p}, b, \hat{b}, g)$, we have
\begin{equation} \label{eqn:var_ifjj}
\begin{split}
& \ \var_{\theta} \left( \mathbb{U}_n (\IF_{j,j,k,\overline{i}_{j}}) \right) \lesssim \frac{j^{2}}{n} \BL_{2, \omegahat, k}^{2 (j - 3)} \left( \BL_{2, \hat{b}, k}^{2} \BL_{2, \omegahat, k}^{2} + \BL_{2, \hat{p}, k}^{2} \BL_{2, \omegahat, k}^{2} + \min\limits_{(\eta, \zeta): 1 / \eta + 1 / \zeta = 1} \BL_{2 \eta, \hat{b}, k}^{2} \BL_{2 \zeta, \hat{p}, k}^{2} \right) \\
& + \frac{1}{n} \sum_{\ell = 2}^{j - 1} j^{2 \ell} \left( \frac{C' k}{n} \right)^{\ell - 1} \BL_{2, \omegahat, k}^{2 (j - \ell - 2) \vee 0} \left( \begin{array}{c}
\BL_{2, \omegahat, k}^{4} + \BL_{2, \hat{b}, k}^{2} \BL_{2, \omegahat, k}^{2} + \BL_{2, \hat{p}, k}^{2} \BL_{2, \omegahat, k}^{2} \\
+ \min\limits_{(\eta, \zeta): 1 / \eta + 1 / \zeta = 1} \BL_{2 \eta, \hat{b}, k}^{2} \BL_{2 \zeta, \hat{p}, k}^{2}
\end{array} \right) + \frac{j^{2 j}}{n} \left( \frac{C' k}{n} \right)^{j - 1}.
\end{split}
\end{equation}
\end{lemma}

The proof of this lemma involves quite tedious calculations so we defer it to Appendix~\ref{proof:var_ifjj}.

Then combining Lemma~\ref{lem:var_if22},~\ref{lem:var_if33},~\ref{lem:var_ifjj} and the following two inequalities:
\begin{equation} \label{var_inequality}
\var_{\theta} \left(\sum_{\ell = 1}^{j} G_{\ell}\right) \leq \sum_{\ell = 1}^{j} 2^{\ell} \var_{\theta} [G_{\ell}], \var_{\theta} \left(\sum_{\ell = 1}^{j} G_{\ell}\right) \leq j \sum_{\ell = 1}^{j} \var_{\theta} [G_{\ell}],
\end{equation}
we have:
\begin{lemma} \label{lem:var_psim}
Under the conditions of Theorem~\ref{theorem_bias_variance_emp}, there exists a positive constant $C$ independent of $n, k, m$ but possibly dependent on $(\lambda_{\min} (\omegahat), \bar{z}_{k}, p, \hat{p}, b, \hat{b}, g)$ such that \allowdisplaybreaks
\begin{align}
& \ \var_{\theta} [\widehat{\psi}_{m, k} - \widehat{\psi}_{1}] = \var_{\theta} \left[ \sum_{j = 2}^{m} \hat{\IIFF}_{j, j, k} (\hat{\Omega}) \right] \label{eq:var_psim} \\
\leq & \ \frac{C}{n} \left( \dfrac{k}{n} + \BL_{2, \hat{b}, k}^{2} + \BL_{2, \hat{p}, k}^{2} \right) \notag \\
& + \frac{C}{n} \left( \left( \dfrac{k}{n} \right)^{2} + \frac{k}{n} \left\{ \BL_{2, \hat{b}, k}^{2} + \BL_{2, \hat{p}, k}^{2} + \BL_{2, \hat{\Omega}, k}^{2} \right\} + \left( \begin{array}{c}
\BL_{2, \hat{b}, k}^{2} \BL_{2, \omegahat, k}^{2} + \BL_{2, \hat{p}, k}^{2} \BL_{2, \omegahat, k}^{2} \\
+ \min\limits_{(\eta, \zeta): 1 / \eta + 1 / \zeta = 1} \BL_{2 \eta, \hat{b}, k}^{2} \BL_{2 \zeta, \hat{p}, k}^{2}
\end{array} \right) \right) \notag \\
& + \frac{C}{n} \sum_{j = 4}^{m} j^{2} \left( C \BL_{2, \omegahat, k} \right)^{2 (j - 3)} \left( \BL_{2, \hat{b}, k}^{2} \BL_{2, \omegahat, k}^{2} + \BL_{2, \hat{p}, k}^{2} \BL_{2, \omegahat, k}^{2} + \min\limits_{(\eta, \zeta): 1 / \eta + 1 / \zeta = 1} \BL_{2 \eta, \hat{b}, k}^{2} \BL_{2 \zeta, \hat{p}, k}^{2} \right) \notag \\
& + \frac{C}{n} \sum_{j = 4}^{m} j^{2} \sum_{\ell = 2}^{j - 1} \left( \frac{C k m^{2}}{n} \right)^{\ell - 1} \left( C \BL_{2, \omegahat, k} \right)^{2 (j - \ell - 2) \vee 0} \left( \begin{array}{c}
\BL_{2, \omegahat, k}^{4} + \BL_{2, \hat{b}, k}^{2} \BL_{2, \omegahat, k}^{2} + \BL_{2, \hat{p}, k}^{2} \BL_{2, \omegahat, k}^{2} \\
+ \min\limits_{(\eta, \zeta): 1 / \eta + 1 / \zeta = 1} \BL_{2 \eta, \hat{b}, k}^{2} \BL_{2 \zeta, \hat{p}, k}^{2}
\end{array} \right) \notag \\
& + \frac{C}{n} \sum_{j = 4}^{m} j^{2} \left( \frac{2 C k m^{2}}{n} \right)^{j - 1}. \notag
\end{align}
\end{lemma}

\begin{proof}
Using \eqref{var_inequality}, we have
\begin{align*}
& \ \var_{\theta} \left[ \sum_{j = 2}^{m} \hat{\IIFF}_{j, j, k} (\hat{\Omega}) \right] \leq 2 \var_{\theta} [\hat{\IIFF}_{2, 2, k} (\hat{\Omega}) + \hat{\IIFF}_{3, 3, k} (\hat{\Omega})] + \sum_{j = 4}^{m} 2^{j - 1} \var_{\theta} [\hat{\IIFF}_{j, j, k} (\hat{\Omega})] \\
\leq & \ 4 \left( \var_{\theta} [\hat{\IIFF}_{2, 2, k} (\hat{\Omega})] + \var_{\theta} [\hat{\IIFF}_{3, 3, k} (\hat{\Omega})] \right) + \sum_{j = 4}^{m} 2^{j - 1} \var_{\theta} [\hat{\IIFF}_{j, j, k} (\hat{\Omega})] \coloneqq I_{1} + I_{2} + I_{3}.
\end{align*}
By Lemma~\ref{lem:var_if22} and~\ref{lem:var_if33}, we have
\begin{align*}
I_{1} \leq \frac{C}{n} \left( \frac{k}{n} + \BL_{2, \hat{b}, k}^{2} + \BL_{2, \hat{p}, k}^{2} \right)
\end{align*}
and
\begin{align*}
I_{2} \leq \frac{C}{n} \left( \left( \dfrac{k}{n} \right)^{2} + \frac{k}{n} \left\{ \BL_{2, \hat{b}, k}^{2} + \BL_{2, \hat{p}, k}^{2} + \BL_{2, \hat{\Omega}, k}^{2} \right\} + \left( \begin{array}{c}
\BL_{2, \hat{b}, k}^{2} \BL_{2, \omegahat, k}^{2} + \BL_{2, \hat{p}, k}^{2} \BL_{2, \omegahat, k}^{2} \\
+ \min\limits_{(\eta, \zeta): 1 / \eta + 1 / \zeta = 1} \BL_{2 \eta, \hat{b}, k}^{2} \BL_{2 \zeta, \hat{p}, k}^{2}
\end{array} \right) \right).
\end{align*}
By Lemma~\ref{lem:var_ifjj}, we have
\begin{align*}
I_{3} \leq & \ \sum_{j = 4}^{m} \frac{j^{2}}{n} 2^{j - 1} \BL_{2, \omegahat, k}^{2 (j - 3)} \left( \BL_{2, \hat{b}, k}^{2} \BL_{2, \omegahat, k}^{2} + \BL_{2, \hat{p}, k}^{2} \BL_{2, \omegahat, k}^{2} + \min\limits_{(\eta, \zeta): 1 / \eta + 1 / \zeta = 1} \BL_{2 \eta, \hat{b}, k}^{2} \BL_{2 \zeta, \hat{p}, k}^{2} \right) \\
& + \sum_{j = 4}^{m} \frac{2^{j - 1}}{n} \sum_{\ell = 2}^{j - 1} j^{2 \ell} \left( \frac{C k}{n} \right)^{\ell - 1} \BL_{2, \omegahat, k}^{2 (j - \ell - 2) \vee 0} \left( \begin{array}{c}
\BL_{2, \omegahat, k}^{4} + \BL_{2, \hat{b}, k}^{2} \BL_{2, \omegahat, k}^{2} + \BL_{2, \hat{p}, k}^{2} \BL_{2, \omegahat, k}^{2} \\
+ \min\limits_{(\eta, \zeta): 1 / \eta + 1 / \zeta = 1} \BL_{2 \eta, \hat{b}, k}^{2} \BL_{2 \zeta, \hat{p}, k}^{2}
\end{array} \right) + \sum_{j = 4}^{m} \frac{j^{2 j}}{n} \left( \frac{2 C k}{n} \right)^{j - 1} \\
\leq & \ \sum_{j = 4}^{m} \frac{j^{2}}{n} 2^{j - 1} \BL_{2, \omegahat, k}^{2 (j - 3)} \left( \BL_{2, \hat{b}, k}^{2} \BL_{2, \omegahat, k}^{2} + \BL_{2, \hat{p}, k}^{2} \BL_{2, \omegahat, k}^{2} + \min\limits_{(\eta, \zeta): 1 / \eta + 1 / \zeta = 1} \BL_{2 \eta, \hat{b}, k}^{2} \BL_{2 \zeta, \hat{p}, k}^{2} \right) \\
& + \sum_{j = 4}^{m} \frac{2^{j - 1}}{n} \sum_{\ell = 2}^{j - 1} j^{2 \ell} \left( \frac{C k}{n} \right)^{\ell - 1} \BL_{2, \omegahat, k}^{2 (j - \ell - 2) \vee 0} \left( \begin{array}{c}
\BL_{2, \omegahat, k}^{4} + \BL_{2, \hat{b}, k}^{2} \BL_{2, \omegahat, k}^{2} + \BL_{2, \hat{p}, k}^{2} \BL_{2, \omegahat, k}^{2} \\
+ \min\limits_{(\eta, \zeta): 1 / \eta + 1 / \zeta = 1} \BL_{2 \eta, \hat{b}, k}^{2} \BL_{2 \zeta, \hat{p}, k}^{2}
\end{array} \right) + \sum_{j = 4}^{m} \frac{j^{2}}{n} \left( \frac{2 C k m^{2}}{n} \right)^{j - 1} \\
\leq & \ \frac{C}{n} \sum_{j = 4}^{m} j^{2} \left( C \BL_{2, \omegahat, k} \right)^{2 (j - 3)} \left( \BL_{2, \hat{b}, k}^{2} \BL_{2, \omegahat, k}^{2} + \BL_{2, \hat{p}, k}^{2} \BL_{2, \omegahat, k}^{2} + \min\limits_{(\eta, \zeta): 1 / \eta + 1 / \zeta = 1} \BL_{2 \eta, \hat{b}, k}^{2} \BL_{2 \zeta, \hat{p}, k}^{2} \right) \\
& + \frac{C}{n} \sum_{j = 4}^{m} j^{2} \sum_{\ell = 2}^{j - 1} \left( \frac{C k m^{2}}{n} \right)^{\ell - 1} \left( C \BL_{2, \omegahat, k} \right)^{2 (j - \ell - 2) \vee 0} \left( \begin{array}{c}
\BL_{2, \omegahat, k}^{4} + \BL_{2, \hat{b}, k}^{2} \BL_{2, \omegahat, k}^{2} + \BL_{2, \hat{p}, k}^{2} \BL_{2, \omegahat, k}^{2} \\
+ \min\limits_{(\eta, \zeta): 1 / \eta + 1 / \zeta = 1} \BL_{2 \eta, \hat{b}, k}^{2} \BL_{2 \zeta, \hat{p}, k}^{2}
\end{array} \right) \\
& + \frac{C}{n} \sum_{j = 4}^{m} j^{2} \left( \frac{2 C k m^{2}}{n} \right)^{j - 1}.
\end{align*}
\end{proof}

With Lemma~\ref{lem:var_psim}, we can further specify the orders of $m, k$ in terms of the sample size $n$, together with a condition on $\BL_{2, \hat{\Omega}, k}$ such that $\var_{\theta} [\widehat{\psi}_{m, k} - \widehat{\psi}_{1}]$ is dominated by that of $\var_{\theta} [\hat{\IIFF}_{2, 2, k}]$.

\begin{corollary} 
\label{lem:var_psim_emp}
Under the conditions of Lemma~\ref{lem:var_psim}, when $m \asymp \log n$, if we take $k \asymp \frac{n}{\log^{3} n}$, restricted to the event that $\hat{\Omega}$ is invertible and $\BL_{2, \hat{\Omega}, k} = o_{\mathbb{P}_{\theta}} (\log^{-1} n)$, there exists a positive constant $C$, depending only on $(\bar{z}_{k}, p, \hat{p}, b, \hat{b}, g)$
\begin{equation}
\label{new_var}
\var_{\theta} [\widehat{\psi}_{m, k} - \widehat{\psi}_{1}] \leq \frac{C}{n} \left( \dfrac{k}{n} + \left\{ \BL_{2, \hat{b}, k}^{2} + \BL_{2, \hat{p}, k}^{2} \right\} + \min\limits_{(\eta, \zeta): 1 / \eta + 1 / \zeta = 1} \BL_{2 \eta, \hat{b}, k}^{2} \BL_{2 \zeta, \hat{p}, k}^{2} \right).
\end{equation}
When $\hat{\Omega} = \hat{\Omega}^{\rm emp}$, $\var_{\theta} [\widehat{\psi}_{m, k}^{\rm emp} - \widehat{\psi}_{1}]$ directly satisfies the upper bound in \eqref{new_var}.
\end{corollary}

\begin{proof}
The conclusion of Corollary~\ref{lem:var_psim_emp} follows from the conclusion of Lemma~\ref{lem:var_psim}, together with the observation
\begin{equation}
\label{ignore}
\frac{k m^{2}}{n} \asymp \frac{n \log^{2} n}{n \log^{3} n} = \frac{1}{\log n} = o (1),
\end{equation}
when we choose $m \asymp \log n$ and $k \asymp \frac{n}{\log^{3} n}$. In the corollary, we have assumed that $\BL_{2, \hat{\Omega}, k} = o_{\mathbb{P}_{\theta}} (\log^{-1} n)$. Under this assumption and the choice of $m, k$, in \eqref{eq:var_psim}, for each summand in the summation $\sum_{j = 1}^{m} (\cdots)$, the multiplication factor $j^{2}$, despite growing at a rate $\log^{2} n$, will still be dominated by the other terms in each corresponding summand because they diminish to 0 at sufficiently fast rates. Hence the higher-order bias correction terms have variance dominated by the second- and third-order terms. When $\hat{\Omega} = \hat{\Omega}^{\rm emp}$, by Lemma~\ref{lemma_spectral_bound_rudelson} in the Appendix, $\BL_{2, \hat{\Omega}^{\rm emp}, k} = O_{\mathbb{P}_{\theta}} \left( \sqrt{\frac{k \log k}{n}} \right) = o_{\mathbb{P}_{\theta}} (\log^{-1} n)$ when $m \asymp \log n$ and $k \asymp \frac{n}{\log^{3} n}$. Therefore the same upper bound \eqref{new_var} holds for $\var_{\theta} [\widehat{\psi}_{m, k}^{\rm emp} - \widehat{\psi}_{1}]$.
\end{proof}
\end{proof}

Corollary~\ref{lem:var_psim_emp} implies the semiparametric efficiency result stated in Corollary~\ref{theorem_semipar_eff_emp}. When $\hat{\Omega} = \hat{\Omega}^{\rm ac}$, some smoothness assumptions on $g$, e.g. $s_{g} > 0$, need to be imposed to ensure $\BL_{2, \hat{\Omega}^{\rm ac}, k} = o_{\mathbb{P}_{\theta}} (1)$, following Lemma~\ref{lemma_op_ac} in Appendix~\ref{section_lemmas}.

\subsection{Details of the proof of the variance bound}
\label{app:var_bound_2}
\def\cp{\mathsf{c}}
We often use the following standard variance bound for general $m$-th order symmetric $U$-statistics.
\begin{lemma}
\label{lem:general}
For an $m$-th order $U$-statistic $\mathbb{U}_{n} h (O_{\bar{i}_{m}})$ with symmetric kernel $h: \mathbb{R}^{m} \rightarrow \mathbb{R}$, we have
\begin{equation}
\label{general bound}
\begin{split}
\var [\mathbb{U}_{n} h (O_{\bar{i}_{m}})] & = \sum_{\cp = 1}^{m} \frac{\binom{n - m}{m - \cp} \binom{m}{\cp}}{\binom{n}{m}} \cov \left[ h (X_{1}, \cdots, X_{m}), h (X_{1}, \cdots, X_{\cp}, X_{m + 1}, \cdots, X_{2 m - \cp}) \right] \\
& \leq \sum_{\cp = 1}^{m} \frac{(2 m^{2})^{\cp}}{n^{\cp}} \mathbb{E} \left[ \{\mathbb{E} [h (X_{1}, \cdots, X_{m}) | X_{1}, \cdots, X_{\cp}]\}^{2} \right].
\end{split}
\end{equation}
\end{lemma}

\begin{proof}
The proof is standard and hence omitted. But note that we use the following combinatorial inequality:
\begin{align*}
\frac{\binom{n - m}{m - \cp}}{\binom{n}{m}} \leq \frac{(2 m)^{\cp}}{n^{\cp}}
\end{align*}
for $n \geq 2 \cp$.
\end{proof}
As $\IF_{j, j, k, \overline{i}_{j}}$ is in general asymmetric, to facilitate the proof, we also introduce the $j$-th order symmetrization operator $\mathbb{S}_{j}$, for any $j$. With slight abuse of notation, we also denote all the permutations over $\{1, \cdots, j\}$ as $\mathbb{S}_{j}$.

\subsection{Proof of Lemma~\ref{lem:var_if33}}
\phantomsection
\label{proof:var_if33}
\begin{proof}
By using Lemma~\ref{lem:general},
\begin{align*}
& \ \var_{\theta} [\mathbb{U}_n [\mathbb{S}_{3} (\IF_{3,3,k,\overline{i}_{3}})]] \\
\leq & \sum_{\cp = 1}^{3} \frac{(3 \sqrt{2})^{2 \cp}}{n^{\cp}} \mathbb{E}_{\theta} \left[ \{\mathbb{E}_{\theta} [\mathbb{S}_{3} (\IF_{3,3,k} (O_{1}, O_{2}, O_{3})) | O_{1}, \cdots, O_{\cp}]\}^{2} \right] \\
\leq & \sum_{\cp = 1}^{3} \frac{(3 \sqrt{2})^{2 \cp}}{n^{\cp}} \max_{\sigma \in \mathbb{S}_{3}} \mathbb{E}_{\theta} \left[ \{\mathbb{E}_{\theta} [\sigma (\IF_{3,3,k} (O_{1}, O_{2}, O_{3})) | O_{1}, \cdots, O_{\cp}]\}^{2} \right] \\
\eqqcolon & \sum_{\cp = 1}^{3} S_{\cp}.
\end{align*}
Then as in the proof of Lemma~\ref{lem:var_if22}, below we repeatedly invoke Lemma~\ref{lemma_general_second_moment} to bound $S_{1}$, $S_{2}$, and $S_{3}$ separately.

For $S_{1}$, we have
\begin{align*}
S_{1} & \leq \frac{C}{n} \Vert \epsp^{2} \Vert_{\infty} \E_{\theta} \left[ \epsb \bar{z}_{k} (X) \right]^{\top} \widehat{\Omega}^{-1} \{\Omega - \omegahat\} \widehat{\Omega}^{-1} \Etheta \left[ \bar{z}_{k} (X) \bar{z}_{k} (X)^{\top} \right] \widehat{\Omega}^{-1} \{\Omega - \omegahat\} \widehat{\Omega}^{-1} \E_{\theta} \left[ \bar{z}_{k} (X) \epsb \right] \\
& \leq \frac{C}{n} \Vert \epsp^{2} \Vert_{\infty} \E_{\theta} \left[ \epsb \bar{z}_{k} (X) \right]^{\top} \widehat{\Omega}^{-1} \{\Omega - \omegahat\} \widehat{\Omega}^{-1} \tilde{\Omega} \widehat{\Omega}^{-1} \{\Omega - \omegahat\} \widehat{\Omega}^{-1} \E_{\theta} \left[ \bar{z}_{k} (X) \epsb \right] \\
& \leq \frac{C}{n} \Vert \epsp^{2} \Vert_{\infty} \E_{\theta} \left[ \epsb \bar{z}_{k} (X) \right]^{\top} \Omega^{- 1 / 2} (\Omega^{1 / 2} \widehat{\Omega}^{-1} \{\Omega - \omegahat\} \widehat{\Omega}^{-1} \tilde{\Omega}^{1 / 2})^{\otimes 2} \Omega^{- 1 / 2} \E_{\theta} \left[ \bar{z}_{k} (X) \epsb \right] \\
& \leq \frac{C}{n} \Vert \epsp^{2} \Vert_{\infty} \E_{\theta} \left[ \epsb \bar{z}_{k} (X) \right]^{\top} \Omega^{-1} \E_{\theta} \left[ \bar{z}_{k} (X) \epsb \right] \BL_{2, \omegahat, k}^{2} \\
& = \frac{C}{n} \Vert \epsp^{2} \Vert_{\infty} \BL_{2, \hat{b}, k}^{2} \BL_{2, \omegahat, k}^{2}.
\end{align*}
In the third line of the above display, given a matrix $A$, we use the short-hand outer product notation $A^{\otimes 2}$ to denote $A A^{\top}$.

By symmetry, we also have
\begin{align*}
S_{1} \leq \frac{C}{n} \Vert \epsb^{2} \Vert_{\infty} \BL_{2, \hat{p}, k}^{2} \BL_{2, \omegahat, k}^{2}.
\end{align*}
Similarly, we can also bound $S_{1}$ as follows:
\begin{align*}
S_{1} \leq & \ \frac{C}{n} \Etheta \left[ \left\{ \E_{\theta} \left[ \epsb \bar{z}_{k} (X) \right]^{\top} \widehat{\Omega}^{-1} \bar{z}_{k} (X) \bar{z}_{k} (X)^{\top} \widehat{\Omega}^{-1} \E_{\theta} \left[ \bar{z}_{k} (X) \epsp \right] \right\}^{2} \right] \\
\leq & \ \left\{ \frac{C}{n} \Etheta \left[ \left( \Etheta \left[ \epsb \bar{z}_{k} (X) \right]^{\top} \omegahat^{-1} \bar{z}_{k} (X) \right)^{2} \right] \left\Vert \E_{\theta} \left[ \epsp \bar{z}_{k} (X) \right]^{\top} \omegahat^{-1} \bar{z}_{k} (X) \right\Vert_{\infty}^{2} \right\} \\
& \ \wedge \left\{ \frac{C}{n} \left( \Etheta \left[ \left( \E_{\theta} \left[ \epsb \bar{z}_{k} (X) \right]^{\top} \omegahat^{-1} \bar{z}_{k} (X) \right)^{4} \right] \right)^{1 / 2} \left( \Etheta \left[ \left( \E_{\theta} \left[ \epsp \bar{z}_{k} (X) \right]^{\top} \omegahat^{-1} \bar{z}_{k} (X) \right)^{4} \right] \right)^{1 / 2} \right\} \\
\leq & \ \frac{C}{n} \left( \BL_{2, \hat{b}, k}^{2} \BL_{\infty, \hat{p}, k}^{2} \right) \wedge \left( \min\limits_{(\eta, \zeta): 1 / \eta + 1 / \zeta = 1} \BL_{2 \eta, \hat{b}, k}^{2} \BL_{2 \zeta, \hat{p}, k}^{2} \right).
\end{align*}
Again, by Lemma~\ref{lem:supnorm} and by symmetry, we have
\begin{align*}
S_{1} \leq \frac{C}{n} \min\limits_{(\eta, \zeta): 1 / \eta + 1 / \zeta = 1} \BL_{2 \eta, \hat{b}, k}^{2} \BL_{2 \zeta, \hat{p}, k}^{2}.
\end{align*}

For $S_{2}$, we have
\begin{align*}
S_{2} & \leq \frac{C}{n^{2}} \Etheta \left[ \left( \varepsilon_{\hat{b}, 1} \bar{z}_{k} (X_{1})^{\top} \omegahat^{-1} (\Omega - \omegahat) \omegahat^{-1} \bar{z}_{k} (X_{2}) \varepsilon_{\hat{p}, 2} \right)^{2} \right] \\
& \leq \frac{C}{n^{2}} \Vert \epsb^{2} \epsp^{2} \Vert_{\infty} \Etheta \left[ \bar{z}_{k} (X_{1})^{\top} \omegahat^{-1} (\Omega - \omegahat) \omegahat^{-1} \bar{z}_{k} (X_{2}) \bar{z}_{k} (X_{2})^{\top} \omegahat^{-1} (\Omega - \omegahat) \omegahat^{-1} \bar{z}_{k} (X_{1}) \right] \\
& = \frac{C}{n^{2}} \Vert \epsb^{2} \epsp^{2} \Vert_{\infty} \Etheta \left[ \bar{z}_{k} (X_{1})^{\top} \omegahat^{-1} (\Omega - \omegahat) \omegahat^{-1} \tilde{\Omega} \omegahat^{-1} (\Omega - \omegahat) \omegahat^{-1} \bar{z}_{k} (X_{1}) \right] \\
& \leq \frac{C}{n} \frac{k}{n} \Vert \epsb^{2} \epsp^{2} \Vert_{\infty} \BL_{2, \omegahat, k}^{2}.
\end{align*}
We can also bound $S_{2}$ as follows
\begin{align*}
S_{2} & \leq \frac{C}{n^{2}} \Etheta \left[ \left\{ \varepsilon_{\hat{b}, 1} \bar{z}_{k} (X_{1})^{\top} \omegahat^{-1} \bar{z}_{k} (X_{2}) \bar{z}_{k} (X_{2})^{\top} \omegahat^{-1} \Etheta \left[ \bar{z}_{k} (X) \epsp \right] \right\}^{2} \right] \\
& \leq \frac{C}{n^{2}} \Vert \epsb^{2} \Vert_{\infty} \Etheta \left[ \epsp \bar{z}_{k} (X) \right]^{\top} \omegahat^{-1} \Etheta \left[ \bar{z}_{k} (X_{2}) \left( \bar{z}_{k} (X_{1})^{\top} \omegahat^{-1} \bar{z}_{k} (X_{2}) \right)^{2} \bar{z}_{k} (X_{2})^{\top} \right] \omegahat^{-1} \Etheta \left[ \bar{z}_{k} (X) \epsp \right] \\
& \leq \frac{C}{n} \frac{1}{n} \Vert \epsb^{2} \Vert_{\infty} \Etheta \left[ \sPi_{g, \bar{z}_{k}} [\epsp] (X_{2})^{2} \bar{z}_{k} (X_{2})^{\top} \omegahat^{-1} \bar{z}_{k} (X_{1}) \bar{z}_{k} (X_{1})^{\top} \omegahat^{-1} \bar{z}_{k} (X_{2}) \right] \\
& = \frac{C}{n} \frac{1}{n} \Vert \epsb^{2} \Vert_{\infty} \Etheta \left[ \sPi_{g, \bar{z}_{k}} [\epsp] (X_{2})^{2} \bar{z}_{k} (X_{2})^{\top} \omegahat^{-1} \tilde{\Omega} \omegahat^{-1} \bar{z}_{k} (X_{2}) \right] \\
& \leq \frac{C}{n} \frac{k}{n} \Vert \epsb^{2} \Vert_{\infty} \Etheta \left[ \sPi_{g, \bar{z}_{k}} [\epsp] (X_{2})^{2} \right] \\
& \leq \frac{C}{n} \frac{k}{n} \BL_{2, \hat{p}, k}^{2} \Vert \epsb^{2} \Vert_{\infty}.
\end{align*}
By symmetry, we also have
\begin{align*}
S_{2} \leq \frac{C}{n} \frac{k}{n} \BL_{2, \hat{b}, k}^{2} \Vert \epsp^{2} \Vert_{\infty}.
\end{align*}
Finally, for $S_{3}$, we have
\begin{align*}
S_{3} & \leq \frac{C}{n^{3}} \E_{\theta} \left[ \varepsilon_{\hat{b}, 1}^{2} \varepsilon_{\hat{p}, 2}^{2} \left( \bar{z}_{k} (X_{1})^{\top} \widehat{\Omega}^{-1} \bar{z}_{k} (X_{3}) \bar{z}_{k} (X_{3})^{\top} \widehat{\Omega}^{-1} \bar{z}_{k} (X_{2}) \right)^{2} \right] \\
& \leq \frac{C}{n^{3}} \Vert \epsb^{2} \epsp^{2} \Vert_{\infty} \Etheta \left[ \bar{z}_{k} (X_{1})^{\top} \omegahat^{-1} \bar{z}_{k} (X_{3}) \bar{z}_{k} (X_{3})^{\top} \omegahat^{-1} \Omega \omegahat^{-1} \bar{z}_{k} (X_{3}) \bar{z}_{k} (X_{3})^{\top} \omegahat^{-1} \bar{z}_{k} (X_{1}) \right] \\
& \leq \frac{C}{n} \frac{k}{n^{2}} \Vert \epsb^{2} \epsp^{2} \Vert_{\infty} \Etheta \left[ \bar{z}_{k} (X_{1})^{\top} \omegahat^{-1} \bar{z}_{k} (X_{3}) \bar{z}_{k} (X_{3})^{\top} \omegahat^{-1} \bar{z}_{k} (X_{1}) \right] \\
& \leq \frac{C}{n} \left( \frac{k}{n} \right)^{2} \Vert \epsb^{2} \epsp^{2} \Vert_{\infty}.
\end{align*}

Taken the above analyses together, we obtain
\begin{align*}
\var_{\theta}\left(\mathbb{U}_n(\IF_{3,3,k,\overline{i}_{3}})\right) & \leq \frac{C}{n} \left\{ \begin{array}{c}
\left( \dfrac{k}{n} \right)^{2} \Vert \epsb^{2} \epsp^{2} \Vert_{\infty} + \dfrac{k}{n} \left( \BL_{2, \hat{b}, k}^{2} \Vert \epsp^{2} \Vert_{\infty} + \Vert \epsb^{2} \Vert_{\infty} \BL_{2, \hat{p}, k}^{2} + \Vert \epsb^{2} \epsp^{2} \Vert_{\infty} \BL_{2, \widehat{\Omega}, k}^{2} \right) \\
+ \left( \begin{array}{c} 
\BL_{2, \hat{b}, k}^{2} \Vert \epsp^{2} \Vert_{\infty} \BL_{2, \widehat{\Omega}, k}^{2} + \Vert \epsb^{2} \Vert_{\infty} \BL_{2, \hat{p}, k}^{2} \BL_{2, \widehat{\Omega}, k}^{2} \\
+ \min\limits_{(\eta, \zeta): 1 / \eta + 1 / \zeta = 1} \BL_{2 \eta, \hat{b}, k}^{2} \BL_{2 \zeta, \hat{p}, k}^{2}
\end{array} \right)
\end{array} \right\}.
\end{align*}
\end{proof}

\begin{remark}
\label{rem:cond-sw1}
Without Condition~\ref{stable}, we cannot guarantee $\Vert \sPi_{g, \bar{z}_{k}} [h] \Vert_{\infty}$ to be bounded even if $\Vert h \Vert_{\infty}$ is bounded. This will affect the variance bound for term $S_{1}$ in the proof above if we use the \Holder{} conjugate pairs $(1, \infty)$ and $(\infty, 1)$. We have instead
\begin{align*}
S_{1} \leq & \ \frac{C}{n} \Etheta \left[ \left\{ \E_{\theta} \left[ \epsb \bar{z}_{k} (X) \right]^{\top} \widehat{\Omega}^{-1} \bar{z}_{k} (X) \bar{z}_{k} (X)^{\top} \widehat{\Omega}^{-1} \E_{\theta} \left[ \bar{z}_{k} (X) \epsp \right] \right\}^{2} \right] \\
\leq & \ \frac{C}{n} \left\{ \Etheta \left[ \left( \E_{\theta} \left[ \epsb \bar{z}_{k} (X) \right]^{\top} \omegahat^{-1} \bar{z}_{k} (X) \right)^{4} \right] \right\}^{1 / 2} \left\{ \Etheta \left[ \left( \E_{\theta} \left[ \epsp \bar{z}_{k} (X) \right]^{\top} \omegahat^{-1} \bar{z}_{k} (X) \right)^{4} \right] \right\}^{1 / 2} \\
\lesssim & \ \frac{C}{n} \left\{ \Etheta \left[ \left( \E_{\theta} \left[ \epsb \bar{z}_{k} (X) \right]^{\top} \Omega^{-1} \bar{z}_{k} (X) \right)^{2} \right] \right\}^{1 / 2} \left\{ \Etheta \left[ \left( \E_{\theta} \left[ \epsp \bar{z}_{k} (X) \right]^{\top} \Omega^{-1} \bar{z}_{k} (X) \right)^{2} \right] \right\}^{1 / 2} \\
& \times \Vert \E_{\theta} \left[ \epsb \bar{z}_{k} (X) \right]^{\top} \Omega^{-1} \bar{z}_{k} (x) \Vert_{\infty} \Vert \E_{\theta} \left[ \epsp \bar{z}_{k} (X) \right]^{\top} \Omega^{-1} \bar{z}_{k} (x) \Vert_{\infty} \\
\leq & \ \frac{C}{n} k \BL_{2, \hat{b}, k}^{2} \BL_{2, \hat{p}, k}^{2}
\end{align*}
where the last line inequality follows from Cauchy-Schwarz inequality to bound the $L_{\infty}$ norms. This weakened bound in turn gives us
\begin{align*}
\var_{\theta}\left(\mathbb{U}_n(\IF_{3,3,k,\overline{i}_{3}})\right) & \leq \frac{C}{n} \left\{ \begin{array}{c}
\left( \dfrac{k}{n} \right)^{2} \Vert \epsb^{2} \epsp^{2} \Vert_{\infty} + \dfrac{k}{n} \left( \BL_{2, \hat{b}, k}^{2} \Vert \epsp^{2} \Vert_{\infty} + \BL_{2, \hat{p}, k}^{2} \Vert \epsb^{2} \Vert_{\infty} + \BL_{2, \widehat{\Omega}, k}^{2} \Vert \epsb^{2} \epsp^{2} \Vert_{\infty} \right) \\
+ \left( \BL_{2, \hat{b}, k}^{2} \BL_{2, \widehat{\Omega}, k}^{2} \Vert \epsp^{2} \Vert_{\infty} + \BL_{2, \hat{p}, k}^{2} \BL_{2, \widehat{\Omega}, k}^{2} \Vert \epsb^{2} \Vert_{\infty} + k \BL_{2, \hat{b}, k}^{2} \BL_{2, \hat{p}, k}^{2} \right)
\end{array} \right\}.
\end{align*}
Finally, we remark that the above upper bound might not be tight, but we believe it requires significant efforts to improve it or show a matching lower bound. Hence we leave it to future work.
\end{remark}

\subsection{Proof of Lemma~\ref{lem:var_ifjj}}
\label{proof:var_ifjj}
\begin{proof}
For general $j \geq 3$, by using Lemma~\ref{lem:general}, we similarly have
\begin{align*}
& \ \var_{\theta} [\mathbb{U}_n [\mathbb{S}_{j} (\IF_{j,j,k,\overline{i}_{j}})]] \\
\leq & \sum_{\cp = 1}^{j} \frac{(\sqrt{2} j)^{2 \cp}}{n^{\cp}} \mathbb{E}_{\theta} \left[ \{\mathbb{E}_{\theta} [\mathbb{S}_{j} (\IF_{j,j,k} (O_{1}, \cdots, O_{j})) | O_{1}, \cdots, O_{\cp}]\}^{2} \right] \\
\leq & \sum_{\cp = 1}^{j} \frac{(2 j^{2})^{\cp}}{n^{\cp}} \max_{\sigma \in \mathbb{S}_{j}} \mathbb{E}_{\theta} \left[ \{\mathbb{E}_{\theta} [\sigma (\IF_{j, j, k} (O_{1}, \cdots, O_{j})) | O_{1}, \cdots, O_{\cp}]\}^{2} \right] \\
\eqqcolon & \sum_{\cp = 1}^{j} S_{\cp}.
\end{align*}

We consider the $S_{\cp}$ for $\cp = 1, \ldots, j$ separately. When $\cp = j$, there is only one term $R_{j}$ to choose from to be the dominating term for $S_{j}$. When $\cp = 1$, we have $j$ different terms, denoted as $R_{11}, \ldots, R_{1j}$ to choose from to be the dominating term for $S_{1}$. Here we let $R_{1\ell}$ be not marginalized over $O_{\ell}$ for $\ell = 1, \cdots, j$. 

For general $\cp$, there will be $\binom{j}{\cp}$ terms in total. First denote the order of the elements in the product of the U-statistic kernel 
\begin{align*}
\IF_{j, j, k, \overline{i}_{j}} = \varepsilon_{\hat{b}, i_{1}} \bar{z}_{k} (X_{i_{1}})^{\top} \omegahat^{-1} \prod_{s = 3}^{j} \left[ \left( \vert H_{1, i_{s}} \vert \bar{z}_{k} (X_{i_{s}}) \bar{z}_{k} (X_{i_{s}})^{\top} - \omegahat \right) \omegahat^{-1} \right] \bar{z}_{k} (X_{i_{2}}) \varepsilon_{\hat{b}, i_{2}}
\end{align*}
by their corresponding ordered subscripts $\{1, \cdots, j\}$ in $\overline{i}_{j}$. Then we let the first $\binom{j - 2}{\cp - 2}$ terms $R_{\cp 1}, R_{\cp 2}, \cdots$ be not marginalized over element $1$, element $2$ and any combination of $\cp - 2$ elements out of $\{3, \cdots, j\}$; the next $\binom{j - 2}{\cp - 1}$ terms be not marginalized over element $1$ and any combination of $\cp - 1$ elements out of $\{3, \cdots, j\}$; the next $\binom{j - 2}{\cp - 1}$ terms be not marginalized over element $2$ and any combination of $\cp - 1$ elements out of $\{3, \cdots, j\}$; and the remaining $\binom{j - 2}{\cp}$ terms be not marginalized over any combination of $\cp$ elements out of $\{3, \cdots, j\}$. 

With the above notation, following calculations similar to those in the proofs of Lemma~\ref{lem:var_if22} and~\ref{lem:var_if33}, we can bound each of the above terms separately:

For $\cp = 1$, we first control the variance of $R_{11}$
\begin{align*}
\var_{\theta} [R_{11}] & \leq \frac{(\sqrt{2} j)^{2}}{n} \Vert \epsb^{2} \Vert_{\infty} \Etheta \left[ \epsp \bar{z}_{k} (X) \right]^{\top} \omegahat^{-1} \Etheta \left[ \left\{ (\Omega - \omegahat) \omegahat^{-1} \right\}^{j - 2} \tilde{\Omega} \left\{ (\Omega - \omegahat) \omegahat^{-1} \right\}^{j - 2} \right] \omegahat^{-1} \Etheta \left[ \bar{z}_{k} (X) \epsp \right] \\
& \leq \frac{C j^{2}}{n} \Vert \epsb^{2} \Vert_{\infty} \BL_{2, \hat{p}, k}^{2} \BL_{2, \omegahat, k}^{2 (j - 2)} \frac{\lambda_{\max} (\Omega) \lambda_{\max} (\tilde\Omega)}{\lambda_{\min} (\omegahat)^{2 (j - 1)}}.
\end{align*}
By symmetry, we also have
\begin{align*}
\var_{\theta} [R_{12}] & \leq \frac{C j^{2}}{n} \Vert \epsp^{2} \Vert_{\infty} \BL_{2, \hat{b}, k}^{2} \BL_{2, \omegahat, k}^{2 (j - 2)} \frac{\lambda_{\max} (\Omega) \lambda_{\max} (\tilde\Omega)}{\lambda_{\min} (\omegahat)^{2 (j - 1)}}.
\end{align*}
$\var_{\theta} [R_{1\ell}]$ for $3 \leq \ell \leq j$ is upper bounded as follows:
\begin{align*}
& \ \var_{\theta} [R_{1\ell}] \\
\leq & \ \frac{C j^{2}}{n} \Etheta \left[ \left\{ \Etheta \left[ \epsb \bar{z}_{k} (X) \right]^{\top} \omegahat^{-1} \left[ \left( \Omega - \omegahat \right) \omegahat^{-1} \right]^{h - 3} \vert H_{1, h} \vert \bar{z}_{k} (X_{h}) \bar{z}_{k} (X_{h})^{\top} \omegahat^{-1} \left[ \left( \Omega - \omegahat \right) \omegahat^{-1} \right]^{j - h} \Etheta \left[ \bar{z}_{k} (X) \epsp \right] \right\}^{2} \right] \\
\leq & \ \frac{C j^{2}}{n} \Etheta \left[ \left\{ \Etheta \left[ \epsb \bar{z}_{k} (X) \right]^{\top} \omegahat^{-1} \left[ \left( \Omega - \omegahat \right) \omegahat^{-1} \right]^{h - 3} \bar{z}_{k} (X_{h}) \right\}^{2} \right] \left\Vert \bar{z}_{k} (X_{h})^{\top} \omegahat^{-1} \left[ \left( \Omega - \omegahat \right) \omegahat^{-1} \right]^{j - h} \Etheta \left[ \bar{z}_{k} (X) \epsp \right] \right\Vert_{\infty}^{2} \\
\leq & \ \frac{C j^{2}}{n} \BL_{2, \hat{b}, k}^{2} \BL_{2, \omegahat, k}^{2 (h - 3)} \frac{\lambda_{\max} (\Omega) \lambda_{\max} (\tilde\Omega)}{\lambda_{\min} (\omegahat)^{2 (h - 2)}} \left\Vert \bar{z}_{k} (X_{h})^{\top} \omegahat^{-1} \left[ \left( \Omega - \omegahat \right) \omegahat^{-1} \right]^{j - h} \Etheta \left[ \bar{z}_{k} (X) \epsp \right] \right\Vert_{\infty}^{2}.
\end{align*}
Then by Lemma~\ref{lem:supnorm} and \Holder{} inequality with \Holder{} conjugate pair $(1, \infty)$, with $C'$ a constant depending on $\lambda_{\min} (\hat{\Omega})$,
\begin{align*}
\var_{\theta} [R_{1\ell}] & \leq \frac{C}{n} \BL_{2, \hat{b}, k}^{2} \BL_{2, \omegahat, k}^{2 (h - 3)} \frac{\lambda_{\max} (\Omega) \lambda_{\max} (\tilde\Omega)}{\lambda_{\min} (\omegahat)^{2 (h - 2)}} \BL_{\infty, \hat{p}, k}^{2} \BL_{2, \omegahat, k}^{2 (j - h)} \frac{1}{\lambda_{\min} (\omegahat)^{2 (j - h + 1)}} \\
& \leq \frac{C j^{2}}{n} \BL_{2, \hat{b}, k}^{2} \BL_{\infty, \hat{p}, k}^{2} (C' \BL_{2, \omegahat, k})^{2 (j - 3)}.
\end{align*}
Note that $\var_{\theta} [R_{1 \ell}]$ (and other similar terms in the proof) can also be bounded by \Holder{} inequality with any valid \Holder{} conjugate pair $(\eta, \zeta)$ \citep{valiant2017automatic}. 
\begin{align*}
\var_{\theta} [R_{1\ell}] & \leq \frac{C j^{2}}{n} \min\limits_{(\eta, \zeta): 1 / \eta + 1 / \zeta = 1} \BL_{2 \eta, \hat{b}, k}^{2} \BL_{2 \zeta, \hat{p}, k}^{2} \cdot (C' \BL_{2, \omegahat, k})^{2 (j - 3)}.
\end{align*}
By symmetry, we have
\begin{align*}
\var_{\theta} [R_{1\ell}] & \leq \frac{C j^{2}}{n} \min\limits_{(\eta, \zeta): 1 / \eta + 1 / \zeta = 1} \BL_{2 \eta, \hat{b}, k}^{2} \BL_{2 \zeta, \hat{p}, k}^{2} \cdot (C' \BL_{2, \omegahat, k})^{2 (j - 3)}
\end{align*}
Hence
\begin{align*}
S_{1} & \leq \max_{\ell = 1, \ldots, j} \var_{\theta} [R_{1 \ell}] \\
& \leq \frac{C j^{2}}{n} \left( C' \BL_{2, \omegahat, k} \right)^{2 (j - 1)} \left( \BL_{2, \hat{b}, k}^{2} + \BL_{2, \hat{p}, k}^{2} + \min\limits_{(\eta, \zeta): 1 / \eta + 1 / \zeta = 1} \BL_{2 \eta, \hat{b}, k}^{2} \BL_{2 \zeta, \hat{p}, k}^{2} \right).
\end{align*}

For $1 < \cp < j$, for any of the first $\binom{j - 2}{\cp - 2}$ terms, it has to be of the following form: denote the indices of the $\ell$ elements that are conditioned on as $t_{1} = 1, t_{2} = 2$, and $\{ t_{3}, \cdots, t_{\cp} \} \subseteq \{3, \cdots, j\}$.
\begin{align*}
R_{\cp \cdot} = \left\{ \begin{array}{c}
\ \left[ \epsb \bar{z}_{k} (X)^{\top} \right]_{i_{1}} \omegahat^{-1} \\
\times \ \prod\limits_{s = 3}^{\cp} \left[ \left( \Omega - \omegahat \right) \omegahat^{-1} \right]^{t_{s} - t_{s - 1} - 1} \left[ \vert H_{i_{s}} \vert \bar{z}_{k} (X_{i_{s}}) \bar{z}_{k} (X_{i_{s}})^{\top} - \Omega \right] \omegahat^{-1} \\
\left[ \left( \Omega - \omegahat \right) \omegahat^{-1} \right]^{j - t_{\cp}} \left[ \bar{z}_{k} (X) \epsb \right]_{i_{2}}
\end{array} \right\}
\end{align*}
We have, by Lemma~\ref{lemma_general_second_moment},
\begin{align*}
& \var_{\theta} [R_{\cp \cdot}] \\
& \leq \frac{C (2 j^{2})^{\cp}}{n^{\cp}} \Vert \epsb^{2} \epsp^{2} \Vert_{\infty} \Etheta \left[ \left\{ \begin{array}{c}
\bar{z}_{k} (X_{i_{1}})^{\top} \omegahat^{-1} \prod\limits_{s = 3}^{\cp} \left[ \left( \Omega - \omegahat \right) \omegahat^{-1} \right]^{t_{s} - t_{s - 1} - 1} \bar{z}_{k} (X_{i_{s}}) \bar{z}_{k} (X_{i_{s}})^{\top} \omegahat^{-1} \\
\left[ \left( \Omega - \omegahat \right) \omegahat^{-1} \right]^{j - t_{\cp}} \bar{z}_{k} (X_{i_{2}})
\end{array} \right\}^{2} \right] \\
& \leq \frac{C (2 j^{2})^{\cp}}{n} \left( \frac{C' k}{n} \right)^{\cp - 1} \Vert \epsb^{2} \epsp^{2} \Vert_{\infty} (C' \BL_{2, \omegahat, k})^{2 (j - \cp)}.
\end{align*}

For any of the next $\binom{j - 2}{\cp - 1}$ terms, it must be of the following form: denote the indices of the $\cp$ elements that are conditioned on as $t_{1} = 1$, and $\{t_{2}, \ldots, t_{\cp}\} \subseteq \{3, \ldots, j\}$.
\begin{align*}
R_{\cp \cdot} = \left\{ \begin{array}{c}
\ \left[ \epsb \bar{z}_{k} (X) \right]_{i_{1}}^{\top} \omegahat^{-1} \\
\times \ \prod\limits_{s = 2}^{\cp} \left[ \left( \Omega - \omegahat \right) \omegahat^{-1} \right]^{t_{s} - t_{s - 1} - 1} \left[ \vert H_{1, i_{s}} \vert \bar{z}_{k} (X_{i_{s}}) \bar{z}_{k} (X_{i_{s}})^{\top} - \Omega \right] \omegahat^{-1} \\
\left[ \left( \Omega - \omegahat \right) \omegahat^{-1} \right]^{j - t_{\cp}} \Etheta \left[ \bar{z}_{k} (X) \epsb \right]
\end{array} \right\}
\end{align*}
We have, by Lemma~\ref{lemma_general_second_moment},
\begin{align*}
\var_{\theta} [R_{\cp \cdot}] & \leq \frac{C (2 j^{2})^{\cp}}{n^{\cp}} \Vert \epsb^{2} \Vert_{\infty} \Etheta \left[ \left\{ \begin{array}{c}
\bar{z}_{k} (X_{i_{1}})^{\top} \omegahat^{-1} \prod\limits_{s = 2}^{\cp} \left[ \left( \Omega - \omegahat \right) \omegahat^{-1} \right]^{t_{s} - t_{s - 1} - 1} \bar{z}_{k} (X_{i_{s}}) \bar{z}_{k} (X_{i_{s}})^{\top} \omegahat^{-1} \\
\left[ \left( \Omega - \omegahat \right) \omegahat^{-1} \right]^{j - t_{\cp}} \Etheta \left[ \bar{z}_{k} (X) \epsb \right]
\end{array} \right\}^{2} \right] \\
& \leq \frac{C (2 j^{2})^{\cp}}{n} \left( \frac{C' k}{n} \right)^{\cp - 1} \Vert \epsb^{2} \Vert_{\infty} \BL_{2, \hat{b}, k}^{2} (C' \BL_{2, \omegahat, k})^{2 (j - 1 - \cp)}.
\end{align*}

By symmetry, for any of the next $\binom{j - 2}{\cp - 1}$ terms, we also have
\begin{align*}
\var_{\theta} [R_{\cp \cdot}] & \leq \frac{C (2 j^{2})^{\cp}}{n} \left( \frac{C' k}{n} \right)^{\cp - 1} \Vert \epsp^{2} \Vert_{\infty} \BL_{2, \hat{p}, k}^{2} (C' \BL_{2, \omegahat, k})^{2 (j - 1 - \cp)}.
\end{align*}
For any one of the last $\binom{j - 2}{\cp}$ terms, it has to be of the following form: denote the indices of the $\cp$ elements that are conditioned on as $\{ t_{1}, \ldots, t_{\cp} \} \subseteq \{3, \ldots, j\}$. Also define $t_{0} = 2$, and then
\begin{align*}
R_{\cp \cdot} = \left\{ \begin{array}{c}
\ \Etheta \left[ \epsb \bar{z}_{k} (X)^{\top} \right] \omegahat^{-1} \\
\times \ \prod\limits_{s = 1}^{\cp} \left[ \left( \Omega - \omegahat \right) \omegahat^{-1} \right]^{t_{s} - t_{s - 1} - 1} \left[ \vert H_{i_{s}} \vert \bar{z}_{k} (X_{i_{s}}) \bar{z}_{k} (X_{i_{s}})^{\top} - \Omega \right] \omegahat^{-1} \\
\left[ \left( \Omega - \omegahat \right) \omegahat^{-1} \right]^{j - t_{\cp}} \Etheta \left[ \bar{z}_{k} (X) \epsb \right]
\end{array} \right\}.
\end{align*}
We in turn have
\begin{align*}
& \ \var_{\theta} [R_{\cp \cdot}] \\
\leq & \ \frac{C (2 j^{2})^{\cp}}{n^{\cp}} \Etheta \left[ \left\{ \begin{array}{c}
\ \Etheta \left[ \epsb \bar{z}_{k} (X)^{\top} \right] \omegahat^{-1} \\
\times \ \prod\limits_{s = 1}^{\cp} \left[ \left( \Omega - \omegahat \right) \omegahat^{-1} \right]^{t_{s} - t_{s - 1} - 1} \bar{z}_{k} (X_{i_{s}}) \bar{z}_{k} (X_{i_{s}})^{\top} \omegahat^{-1} \\
\left[ \left( \Omega - \omegahat \right) \omegahat^{-1} \right]^{j - t_{\cp}} \Etheta \left[ \bar{z}_{k} (X) \epsp \right]
\end{array} \right\}^{2} \right] \\
\leq & \ \frac{C (2 j^{2})^{\cp}}{n^{\cp}} \underbrace{\Etheta \left[ \left\{ \begin{array}{c}
\Etheta \left[ \epsb \bar{z}_{k} (X)^{\top} \right] \omegahat^{-1} \left[ \left( \Omega - \omegahat \right) \omegahat^{-1} \right]^{t_{1} - 3} \bar{z}_{k} (X_{i_{1}}) \bar{z}_{k} (X_{i_{1}})^{\top} \omegahat^{-1} \\
\ \left[ \prod\limits_{s = 2}^{\cp - 1} \left[ \left( \Omega - \omegahat \right) \omegahat^{-1} \right]^{t_{s} - t_{s - 1} - 1} \bar{z}_{k} (X_{i_{s}}) \bar{z}_{k} (X_{i_{s}})^{\top} \omegahat^{-1} \right] \left[ \left( \Omega - \omegahat \right) \omegahat^{-1} \right]^{t_{\cp} - t_{\cp - 1} - 1} \bar{z}_{k} (X_{i_{\cp}}) 
\end{array} \right\}^{2} \right]}_{V_{1}} \\
& \times \underbrace{\left\Vert \bar{z}_{k} (X_{i_{\ell}})^{\top} \omegahat^{-1} \left[ \left( \Omega - \omegahat \right) \omegahat^{-1} \right]^{j - t_{\cp}} \Etheta \left[ \bar{z}_{k} (X) \epsp \right] \right\Vert_{\infty}^{2}}_{V_{2}}.
\end{align*}
By Lemma~\ref{lem:supnorm},
$$
V_{2} \leq C \BL_{\infty, \hat{p}, k}^{2} (C' \BL_{2, \omegahat, k})^{2 (j - t_{\cp})}.
$$
By Lemma~\ref{lemma_general_second_moment},
\begin{align*}
V_{1} \leq C (C' k)^{\cp - 1} \BL_{2, \hat{b}, k}^{2} (C' \BL_{2, \omegahat, k})^{2 (t_{\cp} - 2 - \cp)}.
\end{align*}
Combining the above bounds on $V_{1}$ and $V_{2}$ and by symmetry, we have
\begin{align*}
\var_{\theta} [R_{\cp \cdot}] \leq \frac{C (2 j^{2})^{\cp}}{n} \left( \frac{C' k}{n} \right)^{\cp - 1} \min\limits_{(\eta, \zeta): 1 / \eta + 1 / \zeta = 1} \BL_{2 \eta, \hat{b}, k}^{2} \BL_{2 \zeta, \hat{p}, k}^{2} \cdot (C' \BL_{2, \omegahat, k})^{2 (j - 2 - \cp) \vee 0}.
\end{align*}


Then
\begin{align*}
S_{\cp} & \leq \frac{C (2 j^{2})^{\cp}}{n} \left( \frac{C' k}{n} \right)^{\cp - 1} \left( \begin{array}{c}
\BL_{2, \omegahat, k}^{4} + \BL_{2, \hat{b}, k}^{2} \BL_{2, \omegahat, k}^{2} + \BL_{2, \hat{p}, k}^{2} \BL_{2, \omegahat, k}^{2} \\
+ \min\limits_{(\eta, \zeta): 1 / \eta + 1 / \zeta = 1} \BL_{2 \eta, \hat{b}, k}^{2} \BL_{2 \zeta, \hat{p}, k}^{2}
\end{array} \right) \left( C' \BL_{2, \omegahat, k} \right)^{2 (j - 2 - \cp) \vee 0}.
\end{align*}

For $\cp = j$, we have
\begin{align*}
S_{j} & \leq \frac{C (2 j^{2})^{j}}{n^{j}} \Etheta \left[ \varepsilon_{\hat{b}, 1}^{2} \varepsilon_{\hat{p}, 2}^{2} \left( \bar{z}_{k} (X_{1})^{\top} \omegahat^{-1} \left\{ \prod_{\ell = 3}^{j} \bar{z}_{k} (X_{\ell}) \bar{z}_{k} (X_{\ell})^{\top} \omegahat^{-1} \right\} \bar{z}_{k} (X_{2}) \right)^{2} \right] \\
& \leq \frac{C}{n} (j^{2})^{j} \left( \frac{C' k}{n} \right)^{j - 1} \Vert \varepsilon_{\hat{b}, 1}^{2} \varepsilon_{\hat{p}, 2}^{2} \Vert_{\infty}.
\end{align*}

Finally, summarizing the above calculations, we have
\begin{align*}
& \ \var_{\theta} \left(\mathbb{U}_n(\IF_{j,j,k,\overline{i}_{j}})\right) \\
\leq & \ \frac{C j^{2}}{n} \left( C' \BL_{2, \omegahat, k} \right)^{2 (j - 3)} \left( \BL_{2, \hat{b}, k}^{2} \BL_{2, \omegahat, k}^{2} + \BL_{2, \hat{p}, k}^{2} \BL_{2, \omegahat, k}^{2} + \min\limits_{(\eta, \zeta): 1 / \eta + 1 / \zeta = 1} \BL_{2 \eta, \hat{p}, k}^{2} \BL_{2 \zeta, \hat{b}, k}^{2} \right) \\
& + \frac{C}{n} \sum_{\ell = 2}^{j - 1} (j^{2})^{\ell} \left( \frac{C' k}{n} \right)^{\ell - 1} \left( C' \BL_{2, \omegahat, k} \right)^{2 (j - 2 - \ell) \vee 0} \left( \begin{array}{c}
\BL_{2, \omegahat, k}^{4} + \BL_{2, \hat{b}, k}^{2} \BL_{2, \omegahat, k}^{2} + \BL_{2, \hat{p}, k}^{2} \BL_{2, \omegahat, k}^{2} \\
+ \min\limits_{(\eta, \zeta): 1 / \eta + 1 / \zeta = 1} \BL_{2 \eta, \hat{p}, k}^{2} \BL_{2 \zeta, \hat{b}, k}^{2} 
\end{array} \right) \\
& + \frac{C}{n} (j^{2})^{j} \left( \frac{C' k}{n} \right)^{j - 1}.
\end{align*}
\end{proof}

\begin{remark}
\label{rem:cond-sw2}
Similar to the calculations in Remark~\ref{rem:cond-sw1}, if we do not have Condition~\ref{stable}, we will have instead, for any one of the last $\binom{j - 2}{\ell}$ terms,
\begin{align*}
\var_{\theta} [R_{\ell \cdot}] \leq \frac{C}{n} \left( \frac{k}{n} \right)^{\ell - 1} k \BL_{2, \hat{b}, k}^{2} \BL_{2, \hat{p}, k}^{2} \BL_{2, \omegahat, k}^{2 (j - 2 - \ell) \vee 0},
\end{align*}
which in turn gives us
\begin{align*}
& \var_{\theta} \left[ \frac{1}{\binom{j}{\ell}} \sum_{h = 1}^{\binom{j}{\ell}} R_{\ell h} \right] \leq \max_{1 \leq h \leq {\binom{j}{\ell}}} \var_{\theta} [R_{\ell h}] \\
& \leq \frac{C}{n} \left( \frac{k}{n} \right)^{\ell - 1} \left( \BL_{2, \omegahat, k}^{4} + \BL_{2, \hat{b}, k}^{2} \BL_{2, \omegahat, k}^{2} + \BL_{2, \hat{p}, k}^{2} \BL_{2, \omegahat, k}^{2} + k \BL_{2, \hat{b}, k}^{2} \BL_{2, \hat{p}, k}^{2} \right) \left( \BL_{2, \omegahat, k} \right)^{2 (j - \ell - 2) \vee 0},
\end{align*}
and thus
\begin{align*}
& \ \var_{\theta} \left(\mathbb{U}_n(\IF_{j,j,k,\overline{i}_{j}})\right) = \var_{\theta} \left[ \frac{1}{j} \sum_{h = 1}^{j} R_{1 h} \right] + \sum_{\ell = 2}^{j - 1} \var_{\theta} \left[ \frac{1}{\binom{j}{\ell}} \sum_{h = 1}^{\binom{j}{\ell}} R_{\ell h} \right] + \var_{\theta} [R_{j}] \\
\leq & \ \frac{C}{n} \left( \BL_{2, \omegahat, k} \right)^{2 (j - 3)} \left( \BL_{2, \hat{b}, k}^{2} \BL_{2, \omegahat, k}^{2} + \BL_{2, \hat{p}, k}^{2} \BL_{2, \omegahat, k}^{2} + k \BL_{2, \hat{b}, k} \BL_{2, \hat{p}, k} \right) \\
& + \frac{C}{n} \sum_{\ell = 2}^{j - 1} \left( \frac{k}{n} \right)^{\ell - 1} \left( \BL_{2, \omegahat, k} \right)^{2 (j - \ell - 2) \vee 0} \left( \BL_{2, \omegahat, k}^{4} + \BL_{2, \hat{b}, k}^{2} \BL_{2, \omegahat, k}^{2} + \BL_{2, \hat{p}, k}^{2} \BL_{2, \omegahat, k}^{2} + k \BL_{2, \hat{b}, k}^{2} \BL_{2, \hat{p}, k}^{2} \right) \\
& + \frac{C}{n} \left( \frac{k}{n} \right)^{j - 1}.
\end{align*}
\end{remark}

\section{Technical Lemmas}
\label{section_lemmas} 
In this section we collect some technical lemmas that we use in our proofs.

First, we prove the following $L_{\infty}$ norm control (Condition~\ref{stable}) of series projection estimators for Daubechies wavelets, B-splines and local polynomial partition series \citep{huang2003local, belloni2015some, chen2015optimal, chen2018optimal, cattaneo2020large}. All these three types of basis functions are local bases in the following sense. Let $x = (x_1, \cdots, x_d)^{\top} \in [0, 1]^{d}$.

\begin{lemma}\label{lem:supnorm}
For any function $h \in L_{2} (\P_{\theta})$, denote its $L_{\infty}$ norm as $\Vert h \Vert_{\infty}$. Assume that the density $f_{X} (\cdot)$ of $X$ is bounded between $\sigma_{f} \leq f_{X} (x) \leq \sigma^{f}, \forall x \in [0, 1]^{d}$ for some fixed constants such that $\infty > \sigma^{f} > \sigma_{f} > 0$. When $\bar{z}_{k}$ is Daubechies wavelets, B-spline (with equal-spaced knots) or local polynomial partition series (with intervals of sizes of the same order) with resolution $\lceil \log_{2} (k) \rceil$, for any $k \times k$ matrix $\Sigma$ with operator norm bounded by some constant $M > 0$, we have the following:
\begin{align*}
\left\Vert \bar{z}_{k} (\cdot)^{\top} \Sigma \Etheta \left[ \bar{z}_{k} (X) h(X) \right] \right\Vert_{\infty} \lesssim M \Vert h \Vert_{\infty}
\end{align*}
\end{lemma}

\begin{proof}
Denote the $(i, j)$-th elements of $\Sigma$ as $\sigma_{ij}$. Because $\Vert \Sigma \Vert_{\rm op} \leq M$, $\vert \sigma_{i j} \vert \leq M$ for all $i, j = 1, \ldots, k$. Here, $k^{-1}$ corresponds to the order of the maximum size of the support of each $z_{i}$ in $\bar{z}_{k}$. Given any $x \in [0, 1]^{d}$, denote $I_{x} \coloneqq \{ I \subset \{1, \ldots, k\}: z_{j} (x) \neq 0 \text{ if } j \in I \}$. In particular, for (Daubechies) wavelets with father wavelets/scaling functions, B-splines and local polynomial partition series, $\vert I_{x} \vert \leq k_{0}$ for some fixed integer $k_{0}$ that, unlike $k$, does not depend on $n$. When $d = 1$, the above statement holds for some fixed $j_{0}$, because for all three types of basis functions, each $x \in [0, 1]$ belongs to the supports of at most a constant number of $z_{j}$'s for $j = 1, \cdots, k$ \citep{belloni2015some}. When $d > 1$, since the corresponding basis functions are formed by taking the tensor product of the basis functions for each dimension, we have $k_{0} = j_{0}^{d}$. Both $j_{0}$ and $d$ are fixed in our setting, so $k_{0}$ is also fixed.

Similarly, for each $i \in I_{x}$, $\sum_{j = 1}^{k} z_{i} (x) z_{j} (x')$ also contains at most $O (k_{0})$ nonzero summands for any $x' \in [0, 1]^{d}$. We denote the set of $x' \in [0, 1]^{d}$ such that $z_{i} (x) z_{j} (x') \neq 0$ as $J_{x}^{(j)}$. When $f_{X}$ is bounded between $\sigma^{f} > \sigma_{f}$ such that $\infty > \sigma^{f} > \sigma_{f} > 0$, we have that 
\begin{equation}
\label{eq:prob_bound}
\P_{f_{X}} [ X \in J_{x}^{(j)} ] \lesssim \frac{1}{k},
\end{equation}
because when $i = j$, the probability that $X$ lies in $J_{x}^{(j)}$ should be the largest, and by construction, the support of each $z_{i}$ is at most of order $k^{-1}$. To see the above claims for wavelets, let $\varphi$ be the compactly-supported father wavelet/scaling function with certain regularity (see the description after Definition~\ref{def_Holder}). We then dilate and shift the scaling function $\varphi$ at level $j = \log_{2} (k) / d$ for each dimension of $X$ as $\{2^{j / 2} \varphi (2^{j} x_{m} - \ell), \ell = 0, 1, \cdots, 2^{j}\}$ for $m = 1, \cdots, d$ and then take the tensor product over $d$ dimensions to eventually form the basis functions $\bar{z}_{k}$. Since $\varphi$ is compactly supported, each $x \in [0, 1]^{d}$ is in the support of at most $k_{0}$ basis functions in $\bar{z}_{k}$ for some bounded $k_{0}$. By construction, the support of each basis function $z_{i}$ in $\bar{z}_{k}$ also has size of order $k^{-1}$. As for B-splines, when $d = 1$, B-splines are local in the sense that B-spline basis $z_{j}$ is supported on the interval $[l_{j (1)}, l_{j(2)}]$ for some $j (1)$ and $j (2)$ satisfying $j (2) - j (1) \lesssim 1$, where $l_{1}, \cdots, l_{O (k)}$ are equally knots. There are at most $k_{0}$ non-zero B-splines on each interval of size $O (1 / k)$ defined by the knots. For $d > 1$, the resulting B-spline basis functions are also tensor products of each dimension. From this property of
B-splines, it is easy to see that B-spline series satisfy \eqref{eq:prob_bound}. In terms of local polynomial partition series, as in Example~3.5 of \citet{belloni2015some}, when $d = 1$, $[0, 1]$ is again partitioned into $k$ intervals, each with size of order $k^{-1}$ and each basis function $z_{j}$ is supported in at most $k_{0}$ many intervals. Therefore, \eqref{eq:prob_bound} is again met for general $d$ by taking the tensor product of local polynomial partition series constructed as above for each dimension.

Then
\begin{align*}
& \ \left\vert \bar{z}_{k} (x)^{\top} \Sigma \Etheta \left[ \bar{z}_{k} (X) h(X) \right] \right\vert \\
= & \ \left\vert \Etheta \left[ \bar{z}_{k} (x)^{\top} \Sigma \bar{z}_{k} (X) h(X) \right] \right\vert \\
= & \ \left\vert \Etheta \left[ \sum_{i = 1}^{k} \sum_{j = 1}^{k} \sigma_{i j} z_{i} (x) z_{j} (X) h(X) \right] \right\vert \\
\leq & \ \sum_{i \in I_{x}} \left\vert \sum_{j = 1}^{k} \int \sigma_{i j} z_{i} (x) z_{j} (x') h(x') f_{X} (x') \diff x' \right\vert \\
\lesssim & \ \sum_{i \in I_{x}} \left\vert \sup_{x' \in [0, 1]^{d}} \left\vert \sum_{j = 1}^{k} \sigma_{i j} z_{i} (x) z_{j} (x') h(x') \right\vert \frac{1}{k} \right\vert \\
\leq & \ \frac{1}{k} \sum_{i \in I_{x}} \left\vert \sup_{x' \in [0, 1]^{d}} \sum_{j \in I_{x'}} \sigma_{i j} z_{i} (x) z_{j} (x') h(x') \right\vert \\
\lesssim & \ \frac{1}{k} \sum_{i \in I_{x}} \sup_{x' \in [0, 1]^{d}} \sum_{j \in I_{x'}} |\sigma_{i j} z_{i} (x) z_{j} (x') h(x')| \\
\leq & \ \frac{1}{k} k_{0}^{2} M k \Vert h \Vert_{\infty} = M k_{0}^{2} \Vert h \Vert_{\infty},
\end{align*}
where we use \Holder{} inequality and \eqref{eq:prob_bound} in the second inequality in the above display and the last inequality follows because at level $k$, $\sup_{x, x'} z_{i} (x) z_{j} (x') \lesssim k$ again for wavelets, B-splines and local polynomial partition series \citep[Examples 3.3--3.5]{belloni2015some}.

\end{proof}

The following lemma regarding the operator norm rate of convergence of the sample Gram matrix is used to establish the results in Corollary~\ref{theorem_semipar_eff_emp}.

\begin{lemma}[\cite{rudelson1999random} or Theorem 5.44 of \citet{vershynin2010introduction}]\label{lemma_spectral_bound_rudelson} 
Let $Q_{1}, \ldots, Q_{n}$ be a sequence of independent symmetric non-negative $k \times k$-matrix valued random variables with $k \geq 2$ such that $Q = \frac{1}{n} \sumin \E (Q_{i})$ and $\sup\limits_{i = 1, \ldots, n} \|Q_{i}\|_{\rm op} \leq M$ a.s. where $\| \cdot \|_{\rm op}$ denotes the operator norm of a matrix. Then for $\hat{Q} = \frac{1}{n} \sumin Q_{i}$ and a constant $C > 0$
\begin{align*}
\E \| \hat{Q} - Q \|_{\rm op} \leq C \left( \frac{M \log k}{n} + \sqrt{\frac{M \|Q\|_{\rm op} \log k}{n}} \right).
\end{align*}
Following Theorem 5.44 of \citet{vershynin2010introduction}, with probability converging to 1 as $n \rightarrow \infty$,
\begin{align*}
\| \hat{Q} - Q \|_{\rm op} \leq C \left( \frac{M \log k}{n} + \sqrt{\frac{M \|Q\|_{\rm op} \log k}{n}} \right).
\end{align*}
\end{lemma}

\begin{remark}
Had $\bar{z}_{k} (X)$ satisfied certain light-tail or bounded higher-order moments assumption, results from \cite{vershynin2012close} and \cite{koltchinskii2017concentration} could help get rid of the extra $\log k$ factor in Lemma~\ref{lemma_spectral_bound_rudelson}. However, since $\bar{z}_{k}$'s are wavelets or B-spline transformations in our paper, neither light-tail nor bounded higher-order moments assumption is satisfied. It is unclear how to get rid of the $\log k$ factor in our context and results from \cite{vershynin2012close} and \cite{koltchinskii2017concentration} do not immediately apply.
\end{remark}


The following lemma is the main technical result used to control the variance bound, in particular Lemma~\ref{lem:var_ifjj}. 
\begin{lemma}\label{lemma_general_second_moment} For any given sequences of $k\times k$ matrices $M_0,M_1,\ldots,M_{l}$, with $l \ge 2$, one has for a constant $C$ depending on the choice of basis functions
	\begin{align} 
	\Etheta \left( \left\{ 
	\begin{array}{c}
		\left[ \bar{Z}_{k}^{\top}\right]_{0} M_0 \times \\ 
		\prod\limits_{r=1}^{l-1}\left[M_r[H_1\zk\zk^\top]_{r}
		\right] \\ 
		\times M_l\left[ \bar{Z}_{k}\right]_{l}
	\end{array}
	\right\}\right)^2 \nonumber \leq\left(\| H_1^{2} \|_{\infty}^{l - 1} \lambda_{\max} \left( \Etheta^{l} \left[ \zk\zk^\top \right] \right) \prod\limits_{r=0}^l(\lambda_{\max}(M_r))^2\right) {(Ck)}^{l}, \nonumber
	\end{align}
	where the expectation is taken over the distribution of $X_1,\ldots,X_l$ with $M_0,\ldots,M_l$ treated as fixed.
\end{lemma}

\begin{proof}
The proof follows by writing out the expectation as a multiple integral and then arguing as Lemma 13.4 of \cite{robins2017minimax} in conjunction with repeated use of the variational formula of the largest eigenvalue of a matrix.
\allowdisplaybreaks
Denote $\Omega \coloneqq \Etheta \left[ \zk \zk^\top \right]$.
\begin{align*}
& \; \Etheta \left( \left\{ 
	\begin{array}{c}
		\left[ \bar{Z}_{k}^{\top} \right]_{0} M_{0} \times \\ 
		\prod\limits_{r = 1}^{l - 1} \left[ M_{r} [H_1 \zk\zk^{\top}]_{r} 
		\right] \\ 
		\times M_l \left[ \bar{Z}_{k} \right]_{l}
	\end{array}
	\right\}\right)^2 \\
= & \; \Etheta \left( \left[ \bar{Z}_{k}^{\top} \right]_0 M_0 \cdot \prod\limits_{r=1}^{l-1}\left[M_r[H_1\zk\zk^\top]_{r}\right] \cdot M_l\left[ \bar{Z}_{k} \bar{Z}_{k}^\top \right]_{l} M_l^\top \cdot \prod\limits_{r=1}^{l-1}\left[[H_1\zk\zk^\top]_{r} M_r^\top\right] \cdot M_0^\top \left[ \bar{Z}_{k} \right]_0 \right) \\
= & \; \Etheta \left( \left[ \bar{Z}_{k}^{\top} \right]_0 M_0 \cdot \prod\limits_{r=1}^{l-1}\left[M_r[H_1\zk\zk^\top]_{r}\right] \cdot M_l \Omega M_l^\top \cdot \prod\limits_{r=1}^{l-1}\left[[H_1\zk\zk^\top]_{r} M_r^\top\right] \cdot M_0^\top \left[ \bar{Z}_{k} \right]_0 \right) \\
\le & \; \lambda_{\max}^2 \left( M_l \right) \lambda_{\max} \left( \Omega \right) \Etheta \left( \begin{array}{c}
\left[ \bar{Z}_{k}^{\top}\right]_0 M_0 \cdot \prod\limits_{r=1}^{l-2} \left[ M_r[H_1\zk\zk^\top]_{r} \right] M_{l - 1} \left[ H_1^2 \zk \zk^\top \zk \zk^\top \right]_{l - 1} M_{l - 1}^\top \\
\times \prod\limits_{r=1}^{l-2} \left[ [H_1\zk\zk^\top]_{r} M_r^\top \right] \cdot M_0^\top \left[ \bar{Z}_{k} \right]_0
\end{array} \right) \\
\le & \; \lambda_{\max}^2 \left( M_l \right) \lambda_{\max} \left( \Omega \right) \Vert H_1^2 \zk^\top \zk \Vert_\infty \Etheta \left( \begin{array}{c}
\left[ \bar{Z}_{k}^{\top} \right]_0 M_0 \cdot \prod\limits_{r=1}^{l-2} \left[ M_r[H_1\zk\zk^\top]_{r} \right] M_{l - 1} \left[ \zk \zk^\top \right]_{l - 1} M_{l - 1}^\top \\
\times \prod\limits_{r=1}^{l-2} \left[ [H_1 \zk\zk^\top]_{r} M_r^\top \right] \cdot M_0^\top \left[ \bar{Z}_{k} \right]_0
\end{array} \right) \\
= & \; \lambda_{\max}^2 \left( M_l \right) \lambda_{\max} \left( \Omega \right) \Vert H_1^2 \zk^\top \zk \Vert_\infty \Etheta \left( \begin{array}{c}
\left[ \bar{Z}_{k}^{\top}\right]_0 M_0 \cdot \prod\limits_{r=1}^{l-2} \left[ M_r[H_1\zk\zk^\top]_{r} \right] M_{l - 1} \Omega M_{l - 1}^\top \\
\times \prod\limits_{r=1}^{l-2} \left[ [H_1 \zk\zk^\top]_{r} M_r^\top \right] \cdot M_0^\top \left[ \bar{Z}_{k} \right]_0
\end{array} \right) \\
\le & \; \lambda_{\max}^2 \left( M_l \right) \lambda_{\max}^2 \left( M_{l - 1} \right) \cdot \lambda_{\max}^2 \left( \Omega \right) \Vert H_1^2 \zk^\top \zk \Vert_\infty \Etheta \left( \begin{array}{c}
\left[ \bar{Z}_{k}^{\top}\right]_0 M_0 \cdot \prod\limits_{r=1}^{l-2} \left[ M_r [H_1\zk\zk^\top]_{r} \right] \\
\times \prod\limits_{r=1}^{l-2} \left[ [H_1 \zk\zk^\top]_{r} M_r^\top \right] \cdot M_0^\top \left[ \bar{Z}_{k} \right]_0
\end{array} \right) \\
\overset{(*)}{\le} & \; \prod_{r = 1}^l \lambda_{\max}^2 \left( M_r \right) \cdot \lambda_{\max}^l \left( \Omega \right) \cdot \Vert H_1^2 \zk^\top \zk \Vert_\infty^{l - 1} \cdot \Etheta \left( \left[ \bar{Z}_{k}^{\top} M_0 M_0^\top \bar{Z}_{k} \right]_0 \right) \\
\le & \; \prod_{r = 0}^l \lambda_{\max}^2 \left( M_r \right) \cdot \lambda_{\max}^l \left( \Omega \right) \cdot \Vert H_1^2 \Vert_\infty^{l - 1} \cdot \left( C_1 k \right)^{l - 1} \cdot \Etheta \left( \zk^\top \zk \right) \\
= & \; \prod_{r = 0}^l \lambda_{\max}^2 \left( M_r \right) \cdot \lambda_{\max}^l \left( \Omega \right) \cdot \Vert H_1^2 \Vert_\infty^{l - 1} \cdot \left( C_1 k \right)^{l - 1} \cdot C_2 k \\
\le & \; \Vert H_1^2 \Vert_\infty^{l - 1} \cdot \lambda_{\max}^l \left( \Omega \right) \cdot \prod_{r = 0}^l \lambda_{\max}^2 \left( M_r \right) \cdot \left( C k \right)^l,
\end{align*}
where in (*) we iteratively upper bound $M_r \left[ H_1^2 \zk \zk^\top \zk \zk^\top \right]_r M_r^\top$ by $\Vert H_1^2 \zk^\top \zk \Vert_\infty \cdot M_r \left[ \zk \zk^\top \right]_r M_r^\top$ and take expectation over the $r$-th subject for $r = 1, \dots, l - 2$.
\end{proof}


Finally, we have the following bound on $\BL_{2, \omegahat^{\rm ac}, k} = \Vert \Omega - \hat{\Omega}^{\rm ac} \Vert_{\rm op}$:
\begin{lemma}\label{lemma_op_ac}
Under Condition~\ref{def_conditions},
\begin{align*}
\BL_{2, \omegahat^{\rm ac}, k} = \Vert \Omega - \hat{\Omega}^{\rm ac} \Vert_{\rm op} \leq \Vert \Omega \Vert_{\rm op} \left\Vert \frac{g - \hat{g}}{g} \right\Vert_{\infty}.
\end{align*}
\end{lemma}

\begin{proof}
\begin{align*}
\BL_{2, \omegahat^{\rm ac}, k} & = \sup_{y: \Vert y \Vert_{2} \leq 1} \left\vert \int y^{\top} \bar{z}_{k} (x) \bar{z}_{k} (x)^{\top} y g(x) \frac{g(x) - \hat{g}(x)}{g(x)} \diff x \right\vert \\
& \leq \sup_{y: \Vert y \Vert_{2} \leq 1} \int y^{\top} \bar{z}_{k} (x) \bar{z}_{k} (x)^{\top} y g(x) \left\vert \frac{g(x) - \hat{g}(x)}{g(x)} \right\vert \diff x \\
& \leq \sup_{y: \Vert y \Vert_{2} \leq 1} \int y^{\top} \bar{z}_{k} (x) \bar{z}_{k} (x)^{\top} y g(x) \diff x \left\Vert \frac{g - \hat{g}}{g} \right\Vert_{\infty} \\
& = \Vert \Omega \Vert_{\rm op} \left\Vert \frac{g - \hat{g}}{g} \right\Vert_{\infty}.
\end{align*}
\end{proof}

\section{Adaptive consistent estimators for the nuisance functions}\label{app:adapt}
In this section, we construct adaptive consistent estimators for \Holder{} nuisance functions, which is sufficient for achieving semiparametric efficiency using the empirical HOIF estimators developed in this paper, as can be seen in Corollary~\ref{theorem_semipar_eff_emp}. Without loss of generality, we will construct an adaptive consistent estimator for the nuisance function $\pi (x) = \mathbb{E}_{\theta} [A | X = x]$. In particular, for any $\pi \in H (\beta, C)$, we will construct an estimator $\hat{\pi} (x) \in H (\beta, C)$ such that $\Vert \hat{\pi} - \pi \Vert_{2} = o_{\mathbb{P}_{\theta}} (1)$ without knowing $\beta$ explicitly. $\hat{\pi} (x)$ is of the following form:
\begin{align*}
& \hat{\pi} (x) = \bar{z}_{k^{\dag}} (x)^{\top} \hat{\alpha}_{k^{\dag}} \\
& \text{where } \hat{\alpha}_{k^{\dag}} \coloneqq \left\{ \frac{1}{n_{\rm tr}} \sum_{i = 1}^{n_{\rm tr}} \bar{z}_{k^{\dag}} (X_{i}) \bar{z}_{k^{\dag}} (X_{i})^{\top} \right\}^{-1} \frac{1}{n_{\rm tr}} \sum_{i = 1}^{n_{\rm tr}} A_{i} \bar{z}_{k^{\dag}} (X_{i}) \text{ is the usual OLS estimator.}
\end{align*}
Here we choose a sequence $k^{\dag} = k^{\dag} (n) \rightarrow \infty$ as $n \rightarrow \infty$ and $\bar{z}_{k^{\dag}}$ the Daubechies wavelets at resolution $\log_{2} (k^{\dag})$ \citep{belloni2015some}. For convenience, we also define
\begin{align*}
\tilde{\alpha}_{k^{\dag}} \coloneqq & \ \left\{ \mathbb{E}_{\theta} [\bar{z}_{k^{\dag}} (X) \bar{z}_{k^{\dag}} (X)^{\top}] \right\}^{-1} \frac{1}{n_{\rm tr}} \sum_{i = 1}^{n_{\rm tr}} A_{i} \bar{z}_{k^{\dag}} (X_{i}), \\
\bar{\alpha}_{k^{\dag}} \coloneqq & \ \left\{ \mathbb{E}_{\theta} [\bar{z}_{k^{\dag}} (X) \bar{z}_{k^{\dag}} (X)^{\top}] \right\}^{-1} \mathbb{E}_{\theta} [\pi (X) \bar{z}_{k^{\dag}} (X)], \\
\tilde{\pi} (x) \coloneqq & \ \tilde{\alpha}_{k^{\dag}}^{\top} \bar{z}_{k^{\dag}} (x) \text{ and } \bar{\pi} (x) \coloneqq \bar{\alpha}_{k^{\dag}}^{\top} \bar{z}_{k^{\dag}} (x).
\end{align*}
Since $\pi \in H (\beta, C)$, we immediately have $\bar{\pi} \in H (\beta, C')$ for some $C' > 0$.

Obviously, if $k^{\dag} \rightarrow \infty$, $\hat{\pi}$ is an $L_{2}$-consistent estimator for $\pi$ when $\pi \in H (\beta, C)$ for some $\beta > 0$. So we are only left to specify $k^{\dag}$ such that $\hat{\pi} \in H (\beta, C'')$ for some $C'' > 0$. Following the proof strategy in \citet[Appendix B]{mukherjee2016adaptive}, we need to show the following probability is negligible: for any positive integer $\ell \leq k^{\dag}$ and some appropriately chosen $M > 0$,
\begin{align*}
& \ \mathbb{P}_{\theta} \left( \ell^{\beta / d + 1 / 2} \Vert \langle \hat{\pi}, \bar{z}_{\ell} \rangle \Vert_{\infty} > M \right) \\
\leq & \ \mathbb{P}_{\theta} \left( \ell^{\beta / d + 1 / 2} \Vert \langle \hat{\pi}, \bar{z}_{\ell} \rangle - \langle \tilde{\pi}, \bar{z}_{\ell} \rangle \Vert_{\infty} > M / 2 \right) + \mathbb{P}_{\theta} \left( \ell^{\beta / d + 1 / 2} \Vert \langle \tilde{\pi}, \bar{z}_{\ell} \rangle \Vert_{\infty} > M / 2 \right) \\
\leq & \ \mathbb{P}_{\theta} \left( \ell^{\beta / d + 1 / 2} \Vert \langle \hat{\pi}, \bar{z}_{\ell} \rangle - \langle \tilde{\pi}, \bar{z}_{\ell} \rangle \Vert_{\infty} > M / 4 \right) + \mathbb{P}_{\theta} \left( \ell^{\beta / d + 1 / 2} \Vert \langle \tilde{\pi}, \bar{z}_{\ell} \rangle - \langle \bar{\pi}, \bar{z}_{\ell} \rangle \Vert_{\infty} > M / 2 \right) \\
& + \mathbbm{1} \left( \ell^{\beta / d + 1 / 2} \Vert \langle \bar{\pi}, \bar{z}_{\ell} \rangle \Vert_{\infty} > M / 4 \right) \\
= & \ \mathbb{P}_{\theta} \left( \ell^{\beta / d + 1 / 2} \Vert \langle \hat{\pi}, \bar{z}_{\ell} \rangle - \langle \tilde{\pi}, \bar{z}_{\ell} \rangle \Vert_{\infty} > M / 4 \right) + \mathbb{P}_{\theta} \left( \ell^{\beta / d + 1 / 2} \Vert \langle \tilde{\pi}, \bar{z}_{\ell} \rangle - \langle \bar{\pi}, \bar{z}_{\ell} \rangle \Vert_{\infty} > M / 4 \right) \\
= & \ \mathbb{P}_{\theta} \left( \ell^{\beta / d + 1 / 2} \Vert \left( \hat{\alpha}_{k^{\dag}} - \tilde{\alpha}_{k^{\dag}} \right)^{\top} \mathbb{E}_{\theta} [\bar{z}_{k^{\dag}} (X) \bar{z}_{\ell} (X)^{\top}] \Vert_{\infty} > M / 4 \right) \\
& + \mathbb{P}_{\theta} \left( \ell^{\beta / d + 1 / 2} \left\Vert \left( \frac{1}{n_{\rm tr}} \sum_{i = 1}^{n_{\rm tr}} A_{i} \bar{z}_{k^{\dag}} (X_{i}) - \mathbb{E}_{\theta} [A \bar{z}_{k^{\dag}} (X)] \right)^{\top} \left\{ \mathbb{E}_{\theta} [\bar{z}_{k^{\dag}} (X) \bar{z}_{k^{\dag}} (X)^{\top}] \right\}^{-1} \mathbb{E}_{\theta} [\bar{z}_{k^{\dag}} (X) \bar{z}_{\ell} (X)^{\top}] \right\Vert_{\infty} > M / 4 \right)
\end{align*}
where in the third line we use $\bar{\pi} \in H (\beta, C')$. Now:
\begin{itemize}
\item For the first term in the above display, with probability going to 1,
\begin{align*}
\Vert \left( \hat{\alpha}_{k^{\dag}} - \tilde{\alpha}_{k^{\dag}} \right)^{\top} \mathbb{E}_{\theta} [\bar{z}_{k^{\dag}} (X) \bar{z}_{\ell} (X)^{\top}] \Vert_{\infty} \lesssim \frac{k^{\dag} \log k^{\dag}}{n}.
\end{align*}
\item Similarly, for the second term in the above display, with probability going to 1,
\begin{align*}
\left\Vert \left( \frac{1}{n_{\rm tr}} \sum_{i = 1}^{n_{\rm tr}} A_{i} \bar{z}_{k^{\dag}} (X_{i}) - \mathbb{E}_{\theta} [A \bar{z}_{k^{\dag}} (X)] \right)^{\top} \left\{ \mathbb{E}_{\theta} [\bar{z}_{k^{\dag}} (X) \bar{z}_{k^{\dag}} (X)^{\top}] \right\}^{-1} \mathbb{E}_{\theta} [\bar{z}_{k^{\dag}} (X) \bar{z}_{\ell} (X)^{\top}] \right\Vert_{\infty} \lesssim \frac{k^{\dag}}{n}.
\end{align*}
\end{itemize}
In consequence, any $k^{\dag} \rightarrow \infty$ with a very slow rate (e.g. $k^{\dag} (n) \asymp \log n$) suffices to ensure $\hat{\pi} \in H (\beta, C'')$ with probability going to 1 for some sufficiently large constant $C'' > 0$.

\section{More Details and Results on Simulation Studies}
\label{app:simulations}

\subsection{Numerically generating nuisance functions from \Holder{} classes with given smoothness}
\label{app:holder}
The functions $h_{f}$, $h_{b}$ and $h_{p}$ appearing in Section~\ref{section_simulations} are of the following forms:
\begin{align}
h_{f} (x; \beta_{f}) & \propto 1 + \text{exp} \left\{ \frac{1}{2} \sum_{j \in \mathcal{J}, \ell \in \mathbb{Z}} 2^{- j (\beta_{f} + 0.5)} \omega_{j, \ell} (x) \right\}, \label{f} \\
h_{b} (x; \beta_{b}) & = \sum_{j \in \mathcal{J}, \ell \in \mathbb{Z}} 2^{- j (\beta_{b} + 0.5)} \omega_{j, \ell} (x), \label{b} \\
h_{p} (x; \beta_{p}) & = \sum_{j \in \mathcal{J}, \ell \in \mathbb{Z}} 2^{- j (\beta_{p} + 0.5)} \omega_{j, \ell} (x), \label{p}
\end{align}
where $\mathcal{J} = \{ 0, 3, 6, 9, 10, 16 \}$ and $\omega_{j, \ell} (\cdot)$ is the D12 (or equivalently db6) mother wavelets function dilated at resolution $j$, shifted by $\ell$ \citep{daubechies1992ten, mallat1999wavelet}. The equivalent characterization of Besov-Triebel spaces by the corresponding wavelet coefficients in the frequency domain (see equation (4.89) on page 331 of \cite{gine2016mathematical}) and the embedding of \Holder{} into Besov-Triebel spaces (see page 350 of \cite{gine2016mathematical}) together imply that $h_{f} (\cdot; \beta_{f}) \in H (\beta_{f}, C)$, $h_{b} (\cdot; \beta_{b}) \in H (\beta_{b}, C)$ and $h_{p} (\cdot; \beta_{p}) \in H (\beta_{p}, C)$. For R packages of generating such complex simulations, we refer readers to \citet{xu2022deepmed}. In Table~\ref{tab:s1}, we provide the numerical values for $\left( \zeta_{b, j}, \zeta_{p, j} \right)_{j = 1}^{8}$ used in generating the simulation experiments in Section~\ref{section_simulations}.

\begin{table}[!htbp]
\begin{tabular}{c|c|c}
\hline
$j$ & $\zeta_{b, j}$ & $\zeta_{p, j}$ \\
\hline
1 & -0.2819 & 0.09789 \\
2 & 0.4876 & 0.08800 \\
3 & -0.1515 & -0.4823 \\
4 & -0.1190 & 0.4588 \\
\hline
\end{tabular}
\caption{Coefficients used in constructing $b$ and $\pi$ in Section~\ref{section_simulations}.}\label{tab:s1}
\end{table}

\subsection{Numerically generating correlated multidimensional covariates \texorpdfstring{$X$}{X} with fixed non-smooth marginal densities}
\label{app:multiX}

In the simulation study conducted in Section~\ref{section_simulations}, one key step of generating the simulated datasets is to draw correlated multidimensional covariates $X \in [0, 1]^{d}$ with fixed non-smooth marginal densities. First, we fix the marginal densities of $X$ in each dimension proportional to $h_{f} (\cdot)$ as in \eqref{f}. Then we make $2K$ independent draws of $\tilde{X}_{i, j}$, $i = 1, \ldots, 2K$, from $h_{f}$ for every $j = 1, \ldots, d$ so $\tilde{X} = (\tilde{X}_{1, \cdot}, \ldots, \tilde{X}_{2K, \cdot})^{\top} \in [0, 1]^{2K \times d}$. Next, to introduce correlations between different dimensions, we follow the strategy proposed in \citet{baker2008order}. First we group every two consecutive draws: $(\tilde{X}_{1, \cdot}, \tilde{X}_{2, \cdot})^{\top}, (\tilde{X}_{3, \cdot}, \tilde{X}_{4, \cdot})^{\top}, \ldots, (\tilde{X}_{2K - 1, \cdot}, \tilde{X}_{2K, \cdot})^{\top}$. Then for each pair $(\tilde{X}_{2 i - 1, \cdot}, \tilde{X}_{2 i, \cdot})^{\top}$ for $i = 1, \ldots, K$, we form the following $d$-dimensional random vectors 
\begin{align*}
U_{i} \coloneqq (\max(\tilde{X}_{2 i - 1, 1}, \tilde{X}_{2 i, 1}), \ldots, \max(\tilde{X}_{2 i - 1, d}, \tilde{X}_{2 i, d}))^{\top}, \\ 
V_{i} \coloneqq (\min(\tilde{X}_{2 i - 1, 1}, \tilde{X}_{2 i, 1}), \ldots, \min(\tilde{X}_{2 i - 1, d}, \tilde{X}_{2 i, d}))^{\top}.
\end{align*} 
Lastly, we construct $K$ independent $d$-dimensional vectors $X$ by the following rule: for each $i = 1, \ldots, K$, we draw a Bernoulli random variable $B_{i}$ with probability 1 / 2, and if $B_{i} = 0$, $X_{i, \cdot} = U_{i}$, otherwise $X_{i, \cdot} = V_{i}$. Following the above strategy, we conserve the marginal density of $X_{\cdot, j}$ as that of $\tilde{X}_{\cdot, j}$ but create dependence between different dimensions.

\subsection{Tables presented in the main text}
\label{app:tables}

In this section, we present all the tables (Table~\ref{tab:1}--Table~\ref{tab:6}) referred to in Section~\ref{section_simulations}, where we present the main simulation results in the main text.

\begin{table}[htbp]
\begin{tabular}{c|c|c|c|c|c}
\hline
$n_{\rm tr}$ & $n$ & $\widehat{\mathbb{IF}}_{2, 2, k} (\Omega)$ & $\widehat{\mathbb{IF}}_{2, 2, k} (\omegahat^{\rm emp})$ & $\widehat{\mathbb{IF}}_{2, 2, k} (\omegahat^{\rm ac})$ & $\widehat{\mathbb{IF}}_{2, 2, k} (\hat{g})$ \\ \hline
25,000 & 25,000 & -4.62 (0.489) & -5.02 (0.605) & -6.16 (0.908) & -1.48 (0.197) \\ 
100,000 & 25,000 & -4.50 (0.481) & -4.67 (0.496) & -5.64 (0.801) & -0.303 (0.0408) \\ 
200,000 & 25,000 & -4.56 (0.480) & -4.60 (0.483) & -5.56 (0.749) & -0.148 (0.0197) \\ \hline
25,000 & 100,000 & -6.51 (0.318) & -7.13 (0.373) & -8.68 (0.584) & -2.15 (0.116) \\ 
100,000 & 100,000 & -6.39 (0.306) & -6.62 (0.317) & -7.98 (0.514) & -0.450 (0.0233) \\ 
200,000 & 100,000 & -6.40 (0.307) & -6.45 (0.311) & -7.78 (0.490) & -0.217 (0.0113) \\ \hline
25,000 & 200,000 & -6.54 (0.213) & -7.15 (0.251) & -8.67 (0.418) & -2.15 (0.0804) \\ 
100,000 & 200,000 & -6.42 (0.205) & -6.65 (0.209) & -7.98 (0.365) & -0.452 (0.0158) \\ 
200,000 & 200,000 & -6.43 (0.206) & -6.48 (0.211) & -7.78 (0.342) & -0.218 (0.00761) \\ \hline
\end{tabular}
\caption{Results of simulation experiment (nuisance functions estimated by GLMs): Column 1: training sample size; column 2: estimation sample size; columns 3--7: Monte Carlo means (and standard deviations) of $10^{-2} \times \widehat{\mathbb{IF}}_{2, 2, k} (\Omega)$, $10^{-2} \times \widehat{\mathbb{IF}}_{2, 2, k} (\omegahat^{\rm emp})$, $10^{-2} \times \widehat{\mathbb{IF}}_{2, 2, k} (\omegahat^{\rm ac})$ and $10^{-2} \times \widehat{\mathbb{IF}}_{2, 2, k} (\hat{g})$.}
\label{tab:1}
\end{table}

\begin{table}[htbp]
\begin{tabular}{c|c|c|c|c|c|c}
\hline
$n_{\rm tr}$ & $n$ & $\hat{\psi}_{1} - \psi (\theta)$ & $\hat{\psi}_{2, k} (\Omega) - \psi (\theta)$ & $\hat{\psi}_{2, k}^{\rm emp} - \psi (\theta)$ & $\hat{\psi}_{2, k}^{\rm ac} - \psi (\theta)$ & $\hat{\psi}_{2, k} (\hat{g}) - \psi (\theta)$ \\ \hline
25,000 & 25,000 & -7.57 (0.759) & -2.96 (0.662) & -2.55 (0.718) & -1.42 (0.944) & -6.10 (0.719) \\ 
100,000 & 25,000 & -7.44 (0.756) & -2.95 (0.662) & -2.78 (0.646) & -1.80 (0.868) & -7.14 (0.744) \\ 
200,000 & 25,000 & -7.50 (0.755) & -2.95 (0.662) & -2.90 (0.662) & -1.95 (0.831) & -7.36 (0.749) \\ \hline
25,000 & 100,000 & -7.38 (0.353) & -0.866 (0.251) & -0.252 (0.289) & 1.30 (0.453) & -5.23 (0.288) \\ 
100,000 & 100,000 & -7.25 (0.346) & -0.865 (0.248) & -0.636 (0.258) & 0.725 (0.398) & -6.81 (0.322) \\ 
200,000 & 100,000 & -7.27 (0.347) & -0.866 (0.249) & -0.811 (0.255) & 0.512 (0.378) & -7.05 (0.339) \\ \hline
25,000 & 200,000 & -6.86 (0.225) & -0.322 (0.153) & 0.288 (0.183) & 1.81 (0.356) & -4.71 (0.184) \\ 
100,000 & 200,000 & -6.74 (0.221) & -0.321 (0.152) & 0.0900 (0.154) & 1.24 (0.306) & -6.29 (0.211) \\ 
200,000 & 200,000 & -6.75 (0.221) & -0.321 (0.152) & 0.0267 (0.157) & 1.03 (0.284) & -6.53 (0.217) \\ \hline
\end{tabular}
\caption{Results of simulation experiment (nuisance functions estimated by GLMs): Column 1: training sample size; column 2: estimation sample size; columns 3--7: Monte Carlo means (and standard deviations) of $10^{-2} \times (\hat{\psi}_{1} - \psi (\theta))$, $10^{-2} \times (\hat{\psi}_{2, k} (\Omega) - \psi (\theta))$, $10^{-2} \times (\hat{\psi}_{2, k}^{\rm emp} - \psi (\theta))$, $10^{-2} \times (\hat{\psi}_{2, k}^{\rm ac} - \psi (\theta))$ and $10^{-2} \times (\hat{\psi}_{2, k} (\hat{g}) - \psi (\theta))$.}
\label{tab:2}
\end{table}

\begin{table}[htbp]
\begin{tabular}{c|c|c|c|c|c}
\hline
$n_{\rm tr}$ & $n$ & $\widehat{\mathbb{IF}}_{2, 2, k} (\Omega)$ & $\widehat{\mathbb{IF}}_{2, 2, k} (\omegahat^{\rm emp})$ & $\widehat{\mathbb{IF}}_{2, 2, k} (\omegahat^{\rm ac})$ & $\widehat{\mathbb{IF}}_{2, 2, k} (\hat{g})$ \\ \hline
25,000 & 25,000 & -1.75 (0.489) & -1.88 (0.494) & -1.99 (0.616) & -0.500 (0.151) \\ 
100,000 & 25,000 & -1.71 (0.481) & -1.73 (0.411) & -1.84 (0.551) & -0.101 (0.0321) \\ 
200,000 & 25,000 & -1.79 (0.480) & -1.83 (0.418) & -1.93 (0.527) & -0.0497 (0.0154) \\ \hline
25,000 & 100,000 & -2.56 (0.202) & -2.77 (0.231) & -3.01 (0.310) & -0.772 (0.0718) \\ 
100,000 & 100,000 & -2.51 (0.198) & -2.55 (0.201) & -2.78 (0.275) & -0.162 (0.0152) \\ 
200,000 & 100,000 & -2.52 (0.196) & -2.56 (0.198) & -2.77 (0.265) & -0.0768 (0.00724) \\ \hline
25,000 & 200,000 & -2.56 (0.135) & -2.77 (0.155) & -3.00 (0.211) & -0.765 (0.0469) \\ 
100,000 & 200,000 & -2.52 (0.134) & -2.56 (0.139) & -2.76 (0.185) & -0.162 (0.00983) \\ 
200,000 & 200,000 & -2.52 (0.135) & -2.56 (0.138) & -2.76 (0.177) & -0.0768 (0.00470) \\ \hline
\end{tabular}
\caption{Results of simulation experiment (nuisance functions estimated by GAMs): Column 1: training sample size; column 2: estimation sample size; columns 3--6: Monte Carlo means (and standard deviations) of $10^{-2} \times \widehat{\mathbb{IF}}_{2, 2, k} (\Omega)$, $10^{-2} \times \widehat{\mathbb{IF}}_{2, 2, k} (\omegahat^{\rm emp})$, $10^{-2} \times \widehat{\mathbb{IF}}_{2, 2, k} (\omegahat^{\rm ac})$ and $10^{-2} \times \widehat{\mathbb{IF}}_{2, 2, k} (\hat{g})$.}
\label{tab:3}
\end{table}

\begin{table}[htbp]
\begin{tabular}{c|c|c|c|c|c|c}
\hline
$n_{\rm tr}$ & $n$ & $\hat{\psi}_{1} - \psi (\theta)$ & $\hat{\psi}_{2, k} (\Omega) - \psi (\theta)$ & $\hat{\psi}_{2, k}^{\rm emp} - \psi (\theta)$ & $\hat{\psi}_{2, k}^{\rm ac} - \psi (\theta)$ & $\hat{\psi}_{2, k} (\hat{g}) - \psi (\theta)$ \\ \hline
25,000 & 25,000 & -4.71 (0.709) & -2.95 (0.648) & -2.83 (0.691) & -2.71 (0.767) & -4.21 (0.670) \\ 
100,000 & 25,000 & -4.66 (0.708) & -2.95 (0.648) & -2.93 (0.654) & -2.82 (0.738) & -4.56 (0.697) \\ 
200,000 & 25,000 & -4.75 (0.705) & -2.95 (0.649) & -2.92 (0.656) & -2.83 (0.723) & -4.79 (0.700) \\ \hline
25,000 & 100,000 & -3.44 (0.304) & -0.877 (0.242) & -0.664 (0.264) & -0.424 (0.304) & -2.66 (0.269) \\ 
100,000 & 100,000 & -3.39 (0.300) & -0.876 (0.240) & -0.838 (0.242) & -0.614 (0.281) & -3.23 (0.292) \\ 
200,000 & 100,000 & -3.40 (0.299) & -0.877 (0.240) & -0.841 (0.242) & -0.626 (0.275) & -3.32 (0.295) \\ \hline
25,000 & 200,000 & -2.90 (0.167) & -0.340 (0.151) & -0.128 (0.159) & 0.0961 (0.211) & -2.13 (0.151) \\ 
100,000 & 200,000 & -2.85 (0.164) & -0.339 (0.151) & -0.294 (0.151) & -0.0901 (0.192) & -2.69 (0.160) \\ 
200,000 & 200,000 & -2.86 (0.165) & -0.340 (0.151) & -0.303 (0.153) & -0.100 (0.184) & -2.77 (0.163) \\ \hline
\end{tabular}
\caption{Results of simulation experiment (nuisance functions estimated by GAMs): Column 1: training sample size; column 2: estimation sample size; columns 3--6: Monte Carlo means (and standard deviations) of $10^{-2}
\times (\hat{\psi}_{1} - \psi (\theta))$, $10^{-2} \times (\hat{\psi}_{2, k} (\Omega) - \psi (\theta))$, $10^{-2} \times (\hat{\psi}_{2, k}^{\rm emp} - \psi (\theta))$, $10^{-2} \times (\hat{\psi}_{2, k}^{\rm ac} - \psi (\theta))$ and $10^{-2} \times (\hat{\psi}_{2, k} (\hat{g}) - \psi (\theta))$.}
\label{tab:4}
\end{table}

\begin{table}[htbp]
\begin{tabular}{c|c|c|c}
\hline
$n_{\rm tr}$ & $n$ & $\widehat{\mathbb{IF}}_{2, 2, k} (\omegahat^{\rm emp}) - \widehat{\mathbb{IF}}_{2, 2, k} (\Omega)$ & $\widehat{\mathbb{IF}}_{2, 2, k} (\omegahat^{\rm ac}) - \widehat{\mathbb{IF}}_{2, 2, k} (\Omega)$ \\ \hline
25,000 & 25,000 & -4.06 (2.18) & -15.41 (3.43) \\ 
100,000 & 25,000 & -1.69 (0.865) & -11.44 (2.89) \\ 
200,000 & 25,000 & -0.491 (0.537) & -10.04 (2.59) \\ \hline
25,000 & 100,000 & -6.15 (1.12) & -21.70 (1.88) \\ 
100,000 & 100,000 & -2.29 (0.514) & -15.90 (1.56) \\ 
200,000 & 100,000 & -0.545 (0.340) & -13.78 (1.42) \\ \hline
25,000 & 200,000 & -6.11 (0.835) & -21.31 (1.52) \\ 
100,000 & 200,000 & -2.31 (0.356) & -15.62 (1.26) \\ 
200,000 & 200,000 & -0.540 (0.224) & -13.52 (1.13) \\ \hline
\end{tabular}
\caption{Results of simulation experiment (nuisance functions estimated by GLMs): Column 1: training sample size; column 2: estimation sample size; columns 3--4: Monte Carlo means (and standard deviations) of $10^{-3} \times (\widehat{\mathbb{IF}}_{2, 2, k} (\omegahat^{\rm emp}) - \widehat{\mathbb{IF}}_{2, 2, k} (\Omega))$ and $10^{-3} \times (\widehat{\mathbb{IF}}_{2, 2, k} (\omegahat^{\rm ac}) - \widehat{\mathbb{IF}}_{2, 2, k} (\Omega))$.}
\label{tab:5}
\end{table}

\begin{table}[htbp]
\begin{tabular}{c|c|c|c}
\hline
$n_{\rm tr}$ & $n$ & $\widehat{\mathbb{IF}}_{2, 2, k} (\omegahat^{\rm emp}) - \widehat{\mathbb{IF}}_{2, 2, k} (\Omega)$ & $\widehat{\mathbb{IF}}_{2, 2, k} (\omegahat^{\rm ac}) - \widehat{\mathbb{IF}}_{2, 2, k} (\Omega)$ \\ \hline
25,000 & 25,000 & -1.25 (1.63) & -2.40 (2.01) \\ 
100,000 & 25,000 & -0.269 (0.633) & -1.30 (1.68) \\ 
200,000 & 25,000 & -0.344 (0.472) & -1.32 (1.50) \\ \hline
25,000 & 100,000 & -2.12 (0.768) & -4.53 (0.902) \\ 
100,000 & 100,000 & -0.381 (0.326) & -2.62 (0.719) \\ 
200,000 & 100,000 & -0.366 (0.232) & -2.51 (0.657) \\ \hline
25,000 & 200,000 & -2.12 (0.497) & -4.36 (0.711) \\ 
100,000 & 200,000 & -0.450 (0.231) & -2.49 (0.567) \\ 
200,000 & 200,000 & -0.374 (0.152) & -2.40 (0.501) \\ \hline
\end{tabular}
\caption{Results of simulation experiment (nuisance functions estimated by GAMs): Column 1: training sample size; column 2: estimation sample size; columns 3--4: Monte Carlo means (and standard deviations) of $10^{-3} \times (\widehat{\mathbb{IF}}_{2, 2, k} (\omegahat^{\rm emp}) - \widehat{\mathbb{IF}}_{2, 2, k} (\Omega))$ and $10^{-3} \times (\widehat{\mathbb{IF}}_{2, 2, k} (\omegahat^{\rm ac}) - \widehat{\mathbb{IF}}_{2, 2, k} (\Omega))$.}
\label{tab:6}
\end{table}

\newpage

\subsection{More simulation results on the effects of the choice of \texorpdfstring{$k$}{} and basis functions}
\label{app:practice}

As suggested by a referee, we discuss some further issues of the finite-sample performance of the proposed empirical HOIF estimator. We divide our discussions into the following aspects. In the comparison below, we focus mainly on how the scaling between $k$ and $n$ affects the finite-sample performance of the empirical HOIF estimator. Therefore, we consider a relatively simple simulation setting to focus on the main issues. In particular, we first generate a one-dimensional covariate $X$ uniformly distributed over $[0, 1]$ and then generate the same function for the outcome regression $b$ and the propensity score after logit transformation $\mathrm{logit} (p)$ with \Holder{} smoothness 0.25 using the same strategy \citep{xu2022deepmed} as in Section~\ref{section_simulations}. We estimate nuisance functions using generalized linear models (GLMs). Choosing $d = 1$ simplifies the comparison, as we do not need to further consider whether we should aggregate the basis transformations in an additive manner as in Section~\ref{section_simulations} or take their tensor products. We plan to explore such a comparison in a future paper. 

\paragraph*{Fixing a set of basis functions, the influence of $k$}

First, we examine how the scaling between $k$ and $n$ could affect the finite-sample performance of the empirical HOIF estimators. In this section, we focus on the case where $\bar{z}_{k}$ is constructed by wavelets. From Table~\ref{tab:wavelets}, we observe that as $k$ increases, $\hat{\IIFF}_{2, 2, k} (\hat{\Omega}^{\rm emp})$ can recover a larger fraction of the bias. However, the empirical performance is determined by the scaling between $k$ and $n$. In particular, empirically, we observe that when $k$ is roughly $0.1 n$ to $0.2 n$, $\hat{\IIFF}_{2, 2, k} (\hat{\Omega}^{\rm emp})$ generally exhibits stable numerical performance. For example, when $k = 512$, we need about $n = 20000$ for $\hat{\IIFF}_{2, 2, k} (\hat{\Omega}^{\rm emp})$ to perform well. However, when $k$ increases beyond this range, $\hat{\IIFF}_{2, 2, k} (\hat{\Omega}^{\rm emp})$ begins to become numerically unstable. We suggest that one can check whether $\hat{\IIFF}_{3, 3, k} (\hat{\Omega}^{\rm emp})$ is comparable to $\hat{\IIFF}_{2, 2, k} (\hat{\Omega}^{\rm emp})$ as a heuristic method to empirically verify the numerical stability of $\hat{\IIFF}_{2, 2, k} (\hat{\Omega}^{\rm emp})$. This is because when $\hat{\IIFF}_{2, 2, k} (\hat{\Omega}^{\rm emp})$ performs well, $\hat{\IIFF}_{3, 3, k} (\hat{\Omega}^{\rm emp})$ should be of a magnitude smaller than $\hat{\IIFF}_{2, 2, k} (\hat{\Omega}^{\rm emp})$.

\begin{table}[htbp]
\centering
\begin{tabular}{c|c|c|c}
\hline
$n$ & $k$ & $\hat{\psi}_{1} - \psi (\theta)$ & $\hat{\IIFF}_{2, 2, k} (\hat{\Omega}^{\rm emp})$ \\
\hline
5000 & 32 & 0.464 (0.028) & -0.315 (0.026) \\
5000 & 64 & 0.464 (0.028) & -0.332 (0.032) \\
5000 & 128 & 0.464 (0.028) & -0.453 (0.048) \\
5000 & 256 & 0.464 (0.028) & blow up (blow up) \\
5000 & 512 & 0.464 (0.028) & blow up (blow up) \\
\hline
10000 & 32 & 0.461 (0.020) & -0.315 (0.017) \\
10000 & 64 & 0.461 (0.020) & -0.326 (0.020) \\
10000 & 128 & 0.461 (0.020) & -0.471 (0.033) \\
10000 & 256 & 0.461 (0.020) & -0.560 (0.046)* \\
10000 & 512 & 0.461 (0.020) & blow up (blow up) \\
\hline
20000 & 32 & 0.461 (0.012) & -0.305 (0.008) \\
20000 & 64 & 0.461 (0.012) & -0.330 (0.013) \\
20000 & 128 & 0.461 (0.012) & -0.426 (0.017) \\
20000 & 256 & 0.461 (0.012) & -0.450 (0.020) \\
20000 & 512 & 0.461 (0.012) & -0.452 (0.024)* \\
\hline
50000 & 32 & 0.466 (0.009) & -0.317 (0.008) \\
50000 & 64 & 0.466 (0.009) & -0.322 (0.008) \\
50000 & 128 & 0.466 (0.009) & -0.409 (0.011) \\
50000 & 256 & 0.466 (0.009) & -0.416 (0.011) \\
50000 & 512 & 0.466 (0.009) & -0.431 (0.011) \\
\hline
\end{tabular}
\caption{Simulation results when $\bar{z}_{k}$ is chosen to be wavelets with $\log_{2} (k)$ being the resolution used. In particular, we report the Monte Carlo means and standard deviations (in parentheses) over 200 Monte Carlo repetitions.}
\label{tab:wavelets}
\end{table}

\paragraph*{The choice of basis functions}

Next, we discuss whether our estimators are robust to the choice of basis functions. We first discuss the difference between using father wavelets and using a mixed father/mother wavelets, which is more standard in the harmonic analysis literature \citep{mallat1999wavelet}. However, for the purpose of bias correction using empirical HOIF estimators, it requires more basis functions to span the same space when using mixed father/mother wavelets by the standard multiresolution analysis properties of wavelets. In finite samples, we have seen that using less basis functions leads to smaller variance. In particular, we examine the simulation setting when $n = 5000$ and $k = 32$ and find that using mother wavelets can result in $\hat{\IIFF}_{2, 2, k} (\hat{\Omega}^{\rm emp})$ with slightly larger standard deviation ($0.032$ vs. $0.026$)

Since we have examined the performance of the empirical HOIF estimators that use wavelets, we will focus on two alternative choices, including Fourier series and Chebyshev orthogonal polynomials. In particular, both cases violate Condition~\ref{stable}. We nonetheless observe that the empirical HOIF estimator performs well empirically when $k$ is sufficiently small compared to $n$. In particular, for Fourier series, we consider the same set of $k$'s as in wavelets and report the results in Table~\ref{tab:fourier} below. Similar conclusions hold for Fourier series to those for wavelets in the previous section. We also observe that at the same $k$, $\hat{\IIFF}_{2, 2, k} (\hat{\Omega}^{\rm emp})$ is roughly on a slightly smaller magnitude than the case using wavelets. This is not surprising as the data is generated using wavelets but using the ``wrong'' Fourier series does not significantly deteriorate the performance of $\hat{\IIFF}_{2, 2, k} (\hat{\Omega}^{\rm emp})$.

However, for Chebyshev orthogonal polynomials, $k$ corresponds to the highest degree of polynomials involved and cannot be set to a very large value. In general, based on our own experience, the largest $k$ in which Chebyshev orthogonal polynomials can still be computed in the default R setting is 27, using the R package \texttt{orthopolynom}. Therefore, we report the results for $k \in \{9, 18, 27\}$ and $n \in \{5000, 10000\}$ only. The results can be found in Table~\ref{tab:chebyshev}. At similar $k$'s, $\hat{\IIFF}_{2, 2, k} (\hat{\Omega}^{\rm emp})$'s can recover a similar amount of bias to the ones using Fourier series or wavelets. When $k = 27$ (the largest $k$ at which Chebyshev orthogonal polynomials can still be computed), $\hat{\IIFF}_{2, 2, k} (\hat{\Omega}^{\rm emp})$ starts to suffer from numerical instability, as exemplified by the corresponding Monte Carlo standard deviations and the Monte Carlo means of $\hat{\IIFF}_{3, 3, k} (\hat{\Omega}^{\rm emp})$.

\begin{table}[htbp]
\centering
\begin{tabular}{c|c|c|c}
\hline
$n$ & $k$ & $\hat{\psi}_{1} - \psi (\theta)$ & $\hat{\IIFF}_{2, 2, k} (\hat{\Omega}^{\rm emp})$ \\
\hline
5000 & 32 & 0.464 (0.028) & -0.241 (0.027) \\
5000 & 64 & 0.464 (0.028) & -0.250 (0.029) \\
5000 & 128 & 0.464 (0.028) & -0.302 (0.037) \\
5000 & 256 & 0.464 (0.028) & -0.308 (0.475)* \\
5000 & 512 & 0.464 (0.028) & -0.284 (1.719)* \\
\hline
10000 & 32 & 0.461 (0.020) & -0.243 (0.017) \\
10000 & 64 & 0.461 (0.020) & -0.252 (0.019) \\
10000 & 128 & 0.461 (0.020) & -0.312 (0.024) \\
10000 & 256 & 0.461 (0.020) & -0.340 (0.039) \\
10000 & 512 & 0.461 (0.020) & 0.009 (3.685)** \\
\hline
20000 & 32 & 0.461 (0.012) & -0.231 (0.008) \\
20000 & 64 & 0.461 (0.012) & -0.253 (0.013) \\
20000 & 128 & 0.461 (0.012) & -0.301 (0.015) \\
20000 & 256 & 0.461 (0.012) & -0.305 (0.017) \\
20000 & 512 & 0.461 (0.012) & -0.317 (0.043)* \\
\hline
50000 & 32 & 0.466 (0.009) & -0.245 (0.008) \\
50000 & 64 & 0.466 (0.009) & -0.254 (0.008) \\
50000 & 128 & 0.466 (0.009) & -0.302 (0.010) \\
50000 & 256 & 0.466 (0.009) & -0.302 (0.010) \\
50000 & 512 & 0.466 (0.009) & -0.308 (0.011) \\
\hline
\end{tabular}
\caption{Simulation results when $\bar{z}_{k}$ is chosen to be Fourier series with $k$ being the largest frequency used. In particular, we report the Monte Carlo means and standard deviations (in parentheses) over 200 Monte Carlo repetitions. *: $\hat{\IIFF}_{3, 3, k} (\hat{\Omega}^{\rm emp})$ is close to $\hat{\IIFF}_{2, 2, k} (\hat{\Omega}^{\rm emp})$ in the numerical value; **: $\hat{\IIFF}_{3, 3, k} (\hat{\Omega}^{\rm emp})$ numerically blows up.}
\label{tab:fourier}
\end{table}

\begin{table}[htbp]
\centering
\begin{tabular}{c|c|c|c}
\hline
$n$ & $k$ & $\hat{\psi}_{1} - \psi (\theta)$ & $\hat{\IIFF}_{2, 2, k} (\hat{\Omega}^{\rm emp})$ \\
\hline
5000 & 10 & 0.464 (0.028) & -0.104 (0.013) \\
5000 & 20 & 0.464 (0.028) & -0.153 (0.021) \\
5000 & 25 & 0.464 (0.028) & -0.250 (0.014) \\
5000 & 27 & 0.464 (0.028) & -0.284 (1.719)** \\
\hline
10000 & 10 & 0.461 (0.020) & -0.116 (0.011) \\
10000 & 20 & 0.461 (0.020) & -0.158 (0.014) \\
10000 & 25 & 0.461 (0.020) & -0.257 (0.016) \\
10000 & 27 & 0.461 (0.020) & -0.180 (1.776)** \\
\hline
\end{tabular}
\caption{Simulation results when $\bar{z}_{k}$ is chosen to be Chebyshev orthogonal polynomials with $k$ being the largest degree of polynomials used. In particular, we report the Monte Carlo means and standard deviations (in parentheses) over 200 Monte Carlo repetitions. **: $\hat{\IIFF}_{3, 3, k} (\hat{\Omega}^{\rm emp})$ numerically blows up.}
\label{tab:chebyshev}
\end{table}

\putbib[free_lunch_biblio]
\end{bibunit}

\end{document}